\documentclass[reqno]{amsart}
\usepackage{amsmath,amssymb,amsthm,amsfonts,graphicx,color}
\definecolor{limegreen}{rgb}{0.196,0.804,0.196}
\definecolor{darkgreen}{rgb}{0.0,0.5,0.0}
\definecolor{darkbluegreen}{rgb}{0,0.3,0.6}
\definecolor{badgerred}{rgb}{0.715,0.004,0.004}
\usepackage[font=small,labelfont={bf,sf}, textfont={sf},margin=0em]{caption}
\usepackage{url,hyperref}
\hypersetup{colorlinks, 
            citecolor=badgerred,
            filecolor=black,
            linkcolor=blue,
            urlcolor=darkgreen,
            bookmarksopen=false,
            pdftex}

\newcommand{\bA}{{\bar{\area}}}

\newcommand{\bhm}{{\bar{H}_{\max}}}
\newcommand{\bd}{{\bar{d}}}
\newcommand{\bu}{{\bar{u}}}
\newcommand{\bv}{{\bar{v}}}
\newcommand{\bw}{{\bar{w}}}
\newcommand{\bphi}{{\bar{\phi}}}

\newcommand{\vN}{{\boldsymbol N}}
\newcommand{\vX}{{\boldsymbol X}}
\newcommand{\vx}{\mathbf{x}}
\newcommand{\cL}{\mathcal{L}}

\newcommand{\cO}{\mathcal{O}}
\newcommand{\tE}{{\tilde E}}
\newcommand{\tu}{{\tilde u}}

\newcommand{\tS}{{\tilde \Sigma}}
\newcommand{\hu}{\mathcal{H}}
\newcommand{\hilb}{\mathfrak{H}}
\newcommand{\pd}{\partial}
\newcommand{\nm}{{\boldsymbol\nu}}

\newcommand{\area}{{\mathcal{A}}} 
\newcommand{\eps}{\varepsilon}
\newcommand{\bomega}{{\boldsymbol\omega}}
\newcommand{\ud}{z} 

\newcommand{\R}{{\mathbb R}}

\newtheorem{theorem}{Theorem}[section]
\newtheorem{lemma}[theorem]{Lemma}
\newtheorem{definition}[theorem]{Definition}
\newtheorem{prop}[theorem]{Proposition}
\newtheorem{corollary}[theorem]{Corollary}
\newtheorem{conjecture}[theorem]{Conjecture}
\newtheorem{remark}[theorem]{Remark}

\newtheorem{claim}[theorem]{Claim}

\numberwithin{equation}{section}
\numberwithin{theorem}{section}

\title[Ancient Convex Flows]{Unique asymptotics of ancient convex mean curvature flow solutions }
\author[Angenent]{Sigurd Angenent}
\address{Department of Mathematics, University of Wisconsin -- Madison}
\author[Daskalopoulos]{Panagiota Daskalopoulos}
\address{Department of Mathematics, Columbia University}
\author[Sesum]{Natasa Sesum}
\address{Department of Mathematics, Rutgers University}

\thanks{
P. Daskalopoulos thanks the NSF for support in DMS-1266172.
N. Sesum thanks the NSF for support in DMS-1056387.
}

\begin{document}

\begin{abstract}
We study the compact noncollapsed ancient convex solutions to Mean Curvature Flow in $\R^{n+1}$ with $O(1)\times O(n)$ symmetry.  We show they all have unique asymptotics as $t\to -\infty$ and we give precise asymptotic description of these solutions. In particular, solutions constructed by White, and Haslhofer and Hershkovits have those asymptotics (in the case of those particular solutions the asymptotics was predicted and formally computed by Angenent \cite{AngFormal}).
\end{abstract}

\maketitle
\section{Introduction}

\subsection{Ancient solutions}
\label{sec-ancient-solutions}
A solution to a geometric evolution equation such as the MCF, the Ricci flow, or the Yamabe flow is called \emph{ancient} if it exists for all $t\in(-\infty, t_0]$, for some $t_0$.  While solutions starting from arbitrary smooth initial data can be constructed on a short enough time interval for all these flows, the requirement that a solution should exist for \emph{all} time $t\leq t_0$, combined with some sort of positive curvature condition, turns out to be very restrictive.  In a number of cases there are results which state that the list of possible ancient solutions to some given geometric flow consists of self similar solutions (``solitons'') and a shorter list of non self similar solutions.

For instance, for two dimensional Ricci flow, Daskalopoulos, Hamilton and Sesum \cite{DHS} classified all compact ancient solutions.  It turns out the complete list contains only the shrinking sphere solitons and the King-Rosenau solutions \cite{K1, R}.  The latter are not solitons and can be visualized as two steady solitons, called ``cigars'', coming from spatial infinities and glued together.

Solutions analogous to the King-Rosenau solution exist in a higher dimensional ($n \ge 3$) Yamabe flow as well.  Again they are not solitons, although they are given in an explicit form discovered by King \cite{K1} (and later independently by Brendle in a private communication with the authors). They can also be visualized as two shrinking solitons, called the Barenblatt solutions, coming from spatial infinities and glued together.  In \cite{DDS} Daskalopoulos, del~Pino, and Sesum constructed infinitely many closed ancient solutions (which they called a {\it tower of bubbles}) thus showing that the classification of closed ancient solutions to the Yamabe flow is very difficult, if not impossible.  Unlike the above mentioned closed ancient solutions, the Ricci curvature of the tower of bubbles solutions changes its sign (they still have nonnegative scalar curvature).

We turn now to the Mean Curvature Flow.  Recall that a family of immersed hypersurfaces $\vX:M^n\times [0, T) \to \R^{n+1}$ evolves by Mean Curvature Flow (MCF) if it satisfies
\begin{equation}
\label{eq-mcf}
\left(\frac{\pd \vX}{\pd t}\right)^\perp = H\nm,
\end{equation}
where $\nm$ is a unit normal vector of the surface $M_t = \vX(M^n, t)$, $H$ is the mean curvature in the direction of the normal $\nm$, and $\bigl(\vX_t(\xi, t)\bigr)^\perp$ is the component of the velocity $\vX_t(\xi, t)$ that is perpendicular to $M_t$ at $\vX(\xi,t)$.

A smooth solution $\{M_t\}_{0<\leq t <T}$ to MCF exists on a sufficiently short time interval $0\leq t<T$ for any prescribed smooth initial immersed hypersurface $M_0$.  If the initial hypersurface $M_0$ is convex, then the solution $M_t$ will also be convex.  The simplest possible convex ancient solution is the shrinking sphere, i.e.~if $M_t$ is the sphere of radius $\sqrt{-2nt}$ centered at the origin, then $\{M_t\}_{t<0}$ is a self similar ancient solution.  It is the only compact and convex self-similar solution to MCF. In the next subsection we will give the notion of a {\em non-collapsed} solution to MCF, which was first introduced in \cite{SW} and then  in \cite{An}. With this in
mind we give the following definition.
\begin{definition}
An ancient oval is any ancient compact non-collapsed (in the sense of Definition \ref{def-andrews})  solution to MCF that is not self similar (i.e.~that is not the sphere).
\end{definition}
Note that the ``non-collapsedness'' condition from Definition~\ref{def-andrews} is necessary due to other numerical ``pancake'' type examples which become collapsed as $t\to -\infty$. On the other hand, it has been shown in \cite{HK} that all
non-collapsed ancient compact  solutions to the mean curvature flow are convex, hence the ancient ovals are convex solutions. 

For Curve Shortening, i.e.~MCF for curves in the plane, Angenent found such
solutions (see \cite{AngDoughnuts} and also \cite{97Nakayama}).  These solutions,
which can be written in closed form, may be visualized as two ``Grim Reapers''
with the same asymptotes that approach each other from opposite ends of the
plane.  Daskalopoulos, Hamilton, and Sesum \cite{DHS1} classified all ancient
convex solutions to Curve Shortening by showing that there are no other ancient
ovals for Curve Shortening.

Natural questions to ask are whether there exists an analog of the Ancient Curve Shortening Ovals from \cite{AngDoughnuts,97Nakayama} in higher dimensional Mean Curvature Flow and whether a classification of ancient ovals similar to the Daskalopoulos-Hamilton-Sesum \cite{DHS1} result is possible.

The existence question was already settled by White in \cite{Wh} who gave a construction of ancient ovals for which
\[
\frac{\text{in-radius~}M_t}{\text{out-radius~}M_t} \to 0 \quad \text{ as } t\to-\infty.
\]
Haslhofer and Hershkovits \cite{HO} provided recently more details on White's construction.  If one represents $\R^{n+1}$ as $\R^{n+1} = \R^k\times \R^l$ with $k+l=n+1$, then the White-Haslhofer-Hershkovits construction proves the existence of an ancient solution $M_t$ with $O(k)\times O(l)$ symmetry.  In contrast with the Ancient Curve Shortening Ovals, this solution cannot be
written in closed form.  Formal matched asymptotics, as $t \to -\infty$, were
given by Angenent in \cite{AngFormal}.

The classification question is more complicated in higher dimensions.

\begin{conjecture}[Uniqueness of ancient ovals] \label{conj-uniqueness} For each $(k, l)$ with $k+l=n+1$ there is only one ``ancient oval'' solution with $O(k)\times O(l)$ symmetry, up to time translation and parabolic rescaling of space-time.
\end{conjecture}

Since the ``ancient oval'' solutions are not given in closed form and they are not solitons their classification as stated in the above conjecture poses a difficult question. 
In fact, up to now the only known classification results for ancient or eternal solutions involve either solitons or other special solutions that can be written in closed form. \

In the present paper we make partial progress towards the above conjecture by showing that any ancient, non-collapsed solution of MCF with $O(1)\times O(n)$ symmetry satisfies the {\em detailed asymptotic expansions}  described in \cite{AngFormal}.  In particular, our results give precise estimates on the {\em extrinsic diameter}  and {\em maximum curvature}  of all such solutions near $t \to -\infty$.

\subsection{The non-collapsedness condition}
Instead of an evolving family of convex hypersurfaces $\{M_t\}$ we can also think in terms of the evolving family $\{K_t\}$ of compact domains enclosed by $M_t$ (thus $M_t=\pd K_t$).  Sheng and Wang in \cite{SW} and then later Andrews in \cite{An} introduced the following notion of ``non-collapsedness'' for any compact mean convex subset $K\subset \R^{n+1}$.  Recall that a domain $K\subset\R^{n+1}$ with smooth boundary is mean convex if $H>0$ on $\pd K$.
\begin{definition}
\label{def-andrews}
If $K\subset \R^{n+1}$ is a smooth, compact, mean convex domain and if $\alpha > 0$, then $K$ is {\bfseries$\alpha$-noncollapsed} if for every $p\in \pd K$ there are closed balls $\bar{B}_{int}\subset K$ and $\bar{B}_{ext}\subset \R^{n+1}\backslash Int(K)$ of radius at least $\frac{\alpha}{H(p)}$ that are tangent to $\pd K$ at $p$ from the interior and exterior of $K$, respectively (in the limiting case $H(p) = 0$ this means that $K$ is a half space).
\end{definition}
Every compact, smooth, strictly mean convex domain is $\alpha$-noncollapsed for some $\alpha > 0$.  Andrews showed that if the initial condition $K_0$ of a smooth compact mean curvature flow is $\alpha$-noncollapsed, then so is the whole flow $K_t$ for all later times $t$.
\begin{definition}
We say that an ancient solution $\{M_t\}_{t\in (-\infty, T]}$ to MCF is \textbf{noncollapsed} if there exists a constant $\alpha > 0$ so that the flow $M_t$ is $\alpha$-noncollapsed for all $t \in (-\infty, T]$, in the sense of Definition \ref{def-andrews}.
\end{definition}

In order to say more about the classification of closed ancient noncollapsed solutions to the mean curvature flow, we need to understand first the geometry of those solutions and their more precise asymptotics.
\subsection{MCF for hypersurfaces with $O(1)\times O(n)$ symmetry}

In this paper we consider non collapsed and therefore convex ancient solutions that are $O(1) \times O(n)$-invariant hypersurfaces in $\R^{n+1}$.  Such hypersurfaces can be represented as
\begin{equation}
\label{eq-O1xOn-symmetry}
M_t = \bigl\{ (x, x') \in \R\times\R^{n} : -d(t)<x<d(t), \|x'\|=U(x, t)\bigr\}
\end{equation}
for some function $\|x'\|=U(x,t)$.  The points $(\pm d(t), 0)$ are called \emph{the tips} of the surface.  The function $U(x, t)$, which we call the \emph{profile} of the hypersurface $M_t$, is only defined for $x\in[-d(t), d(t)]$.

Any surface $M_t$ defined by \eqref{eq-O1xOn-symmetry} is automatically invariant under $O(n)$ acting on $\R\times\R^n$.  The surface will also be invariant under the $O(1)$ action on $\R\times\R^n$ if $U$ is even, i.e.~if $U(-x, t)=U(x,t)$.
 
Convexity of the surface $M_t$ is equivalent to concavity of the profile $U$, i.e.~$M_t$ is convex if and only if $U_{xx}\leq0$.

For a family of surfaces defined by $\|x'\|=U(x, t)$, equation \eqref{eq-mcf} for MCF holds if and only if the profile $U(x,t)$ satisfies the evolution equation
\begin{equation}
\label{eq-u-original}
\frac{\pd U}{\pd t} = \frac{U_{xx}}{1+U_x^2} - \frac{n-1}{U}.
\end{equation}
We know by Huisken's result (\cite{Hu}) that the surfaces $M_t$ will contract to a point in finite time.

Self-similar solutions to MCF are of the form $M_t = \sqrt{T-t}\,\bar M$ for some fixed surface $\bar M$ and some ``blow-up time'' $T$.  We rewrite a general ancient solution $\{M_t : t<t_0\}$ as
\begin{equation}
\label{eq-type-1-blow-up}
M_t = \sqrt{T-t} \, \bar M_{-\log (T-t)}.
\end{equation}
The family of surfaces $\bar M_\tau$ with $\tau = -\log(T-t)$, is called a type-I or \emph{parabolic blow-up} of the original solution $M_t$.  These are again $O(1)\times O(n)$ symmetric with profile function $u$, which is related to $U$ by
\begin{equation}
\label{eq-cv1}
U(x,t) = \sqrt{T-t}\, u(y, \tau), \qquad y=\frac x{ \sqrt{T-t}}, \quad \tau=-\log (T-t).
\end{equation}
If the $M_t$ satisfy MCF, then the hypersurfaces $\bar{M}_\tau$ evolve by the \emph{rescaled MCF}
\begin{equation}
\label{eq-RMCF}
\nm\cdot\frac{\pd \vX}{\pd \tau}
= H + \tfrac12 \vX\cdot\nm.
\end{equation}
For the parabolic blow-up $u$ this is equivalent with the equation
\begin{equation}
\label{eq-u}
\frac{\pd u}{\pd \tau} = \frac{u_{yy}}{1+u_y^2} - \frac y2 \, u_y - \frac{n-1}{u}+ \frac u2.
\end{equation}
Regarding notation, we denote by $H(\cdot,t)$, $d(t)$, etc., the mean curvature and extrinsic diameter of the surface $M_t$, respectively, and by $\bar{H}(\cdot,\tau)$, $\bar{d}(\tau)$, etc., the mean curvature and extrinsic diameter of a corresponding parabolic blow-up $\bar{M}_{\tau}$, respectively. In general, we will use the bar to denote geometric quantities for $\bar{M}_{\tau}$.

The following theorem, which will be shown in section \ref{sec-tip}, describes certain geometric properties of the ancient solutions described above.

\begin{theorem}
\label{thm-comparable-geometry}
Let $\{M_t\}$ be any compact smooth noncollapsed ancient mean curvature flow with $O(1)\times O(n)$ symmetry.  Then there exist uniform constants $c, C > 0$ so that the extrinsic diameter $\bd(\tau)$, the area $\bA(\tau)$ and the maximum mean curvature $\bhm(\tau)$ of the rescaled mean curvature flow $\bar{M}_{\tau}$ satisfy
\begin{equation}
\label{eq-comparable} 
c \, \bd(\tau) \le \bhm(\tau) \le \bd(\tau),
\qquad c\, \bd(\tau) \le\bA(\tau) \le C\, \bd(\tau).
\end{equation}
\end{theorem}

Corollary 6.3 in \cite{W} implies that the dilations $\{\bar{X}\in \R^{n+1}\,\,\, |\,\,\, (-t)^{1/2} \bar{X} \in M_t\}$, of hypersurfaces $M_t$ which evolve by \eqref{eq-mcf} and which satisfy conditions of Theorem \ref{thm-comparable-geometry} that {\it sweep out the whole space}, converge as $t\to -\infty$
\begin{enumerate}
\item[(a)] to either a sphere of radius $\sqrt{2n}$ or
\item[(b)] a cylinder $S^{n-1}\times \R$, where $S^{n-1}$ is a sphere of radius $\sqrt{2(n-1)}$.
\end{enumerate}

In the present paper we show that any compact convex ancient solution to \eqref{eq-mcf} as in Theorem \ref{thm-comparable-geometry} has unique asymptotics as $t\to -\infty$.  We hope to use that to eventually prove Conjecture \ref{conj-uniqueness}.  More precisely, we show that the following holds.

\begin{theorem}
\label{thm-asymptotics}
Let $\{M_t\}$ be any compact smooth noncollapsed ancient mean curvature flow as in Theorem \ref{thm-comparable-geometry}.  Then, either $M_t$ is a family of contracting spheres or the solution $u(y,\tau)$ to \eqref{eq-u}, defined on $\mathbb{R}\times \mathbb{R}$, has the following asymptotics in the parabolic and the intermediate region:
\begin{enumerate}
\item[(i)] {\bf Parabolic region:} For every $M > 0$,
\[
u(y,\tau) = \sqrt{2(n-1)} \Bigl(1 - \frac{y^2 - 2}{4|\tau|}\Bigr) + o(|\tau|^{-1}), \qquad |y| \le M
\]
as $\tau \to -\infty$.
\item[(ii)] {\bf Intermediate region:} Define $z := {y}/{\sqrt{|\tau|}}$ and $\bar{u}(z,\tau) := u(z\sqrt{|\tau|}, \tau)$.  Then, $\bar{u}(z,\tau)$ converges, as $\tau\to -\infty$ and uniformly on compact subsets in $z$, to the function $\sqrt{2 - z^2}$.

\item[(iii)] {\bf Tip region:} Denote by $p_t$ the tip of $M_t \subset \mathbb{R}^{n+1}$, and for any $t_*<0$ we define the rescaled flow
\[
\tilde{M}_{t_*}(t) = \lambda(t_*) \bigl(M_{t_* + t \lambda(t_*)^{-2}} - p _{t_*}\bigr)
\]
where $\lambda(t) := H(p_{t}, t) = H_{\max}(t)$.  Then, as $t_*\to-\infty$, the family of solutions $M_{t_*}(\cdot)$ to MCF converges to the unique Bowl soliton, i.e.~the unique rotationally symmetric translating soliton with velocity one.
\end{enumerate}
\end{theorem}

As a consequence of Theorem \ref{thm-comparable-geometry} and Theorem \ref{thm-asymptotics} we have the following corollary.

\begin{corollary}
\label{cor-exact-diam}
Let $\{M_t\}$ be any compact smooth noncollapsed ancient mean curvature flow as in Theorem \ref{thm-comparable-geometry}.  Then there exist uniform constants $c, C > 0$ and $\tau_0 <0$ so that
\[
c\, \sqrt{|\tau|} \le \bd(\tau) \le C\, \sqrt{|\tau|}, \qquad \tau \leq \tau_0.
\]
\end{corollary}

\smallskip
\subsection{Outline of the paper}
\label{sec-outline}
In section \ref{sec-sphere} we give a characterization of the sphere via backward limits. In section \ref{sec-geometric-prop} we show a few geometric properties of ancient solutions one of which is that the maximum of $H$ occurs at the tip. In section \ref{sec-apriori} we first show some a priori derivative estimates for our solution and then we construct lower barriers for our solution which turn out to play a crucial role in bounding the diameter $\bar{d}(\tau)$ from below by $C\sqrt{|\tau|}$.  In the same section we show the key estimate that holds for any hypersurface which, on a long piece, can be written as a graph over a cylinder and whose Huisken integral is below the Huisken integral of the same cylinder. The key estimate is actually a quantitative version of the closeness statement of that hypersurface to the cylinder. In section \ref{sec-parabolic} we prove the asymptotics of the parabolic region. In section \ref{sec-inter} we show the asymptotics in the intermediate region. In section \ref{sec-tip} we give a description of the tip region in terms of its blow-up limit. In section \ref{sec-constructing-the-foliation} we give a detailed construction of the minimizing foliation which has been used in the proof of the key estimate and in the construction of the lower barriers for our solution.

\section{Characterization of the sphere via backward limits}
\label{sec-sphere}
Let $M_t$ be an ancient convex solution of MCF that sweeps out all of $\R^{n+1}$, and define the rescaled hypersurfaces $\bar M_\tau$ as in \eqref{eq-type-1-blow-up}.  Then White showed in \cite[corollary 6.3]{W} that the blow-ups $\bar M_\tau$ converge as $\tau\to -\infty$
\begin{enumerate}
\item[(a)] either to a sphere of radius $\sqrt{2n}$ or
\item[(b)] or to a cylinder $S^{n-1}\times \R$, where $S^{n-1}$ is a sphere of radius $\sqrt{2(n-1)}$.
\end{enumerate}
White's result includes other possible limits, all of the form $S^k\times\R^{n-k} \subset \R^{n+1}$, but none of these are compatible with the $O(1)\times O(n)$ symmetry, which we assume.

\begin{lemma}
\label{lem-only-cylinders}
If $M_t$ is an $O(n)\times O(1)$ invariant ancient convex noncollapsed solution of MCF, then either $M_t$ is the spherical soliton, or the type-I blow-ups $\bar{M}_\tau$ converge to the cylinder $S^{n-1}\times\R$ as $\tau\to-\infty$.
\end{lemma}
To prove the lemma, we first recall that Huisken's integral for hypersurfaces $M\subset\R^{n+1}$ is given by
\[
\hu({M}) := (4\pi)^{-n/2}\, \int_{{M}} e^{-\|{X}\|^2/4}\, d\bar{\mu}.
\]
and
\begin{equation}
\label{eq-mon-huisken}
\frac{d}{d\tau} \hu(\bar{M}_\tau)
= -(4\pi)^{-n/2}\int _{\bar{M}_{\tau}}\, 
e^{-\|\bar{X}\|^2/4}\,
\left\|\bar{H} - \tfrac{1}{2}\langle \bar{X}, \bar{\mu} \rangle\right\|^2 \,
d\bar{\mu}.
\end{equation}

First we show that $\hu(\bar{M}_\tau)$ is uniformly bounded for all $\tau\in (-\infty,\infty)$.  More precisely, we have the following lemma.

\begin{lemma}
\label{lem-upp-bound}
There is a universal constant $C_n<\infty$ such that $\hu(M)\leq C_n$ for every convex hypersurface $M\subset \R^{n+1}$.

\end{lemma}
\begin{proof}
One can cover $S^n\subset\R^{n+1}$ with a finite number of sets $\omega_i\subset S^n$ such that for each $i$ there is a unit vector $a_i\in\R^{n+1}$ for which $ \langle a_i, \nm \rangle \geq \frac12$ for all $\nm\in \omega_i$.  The required number of sets is bounded by some constant $N_n$ that only depends on the dimension $n$.

The sets $U_i\subset M$ defined by $U_i = \{p\in M | \nm(p)\in \omega_i\}$ form a covering of $M$.
We estimate $\int_{U_i} e^{-\|X\|^2/4}d\mu$.  Without loss of generality we may assume that $a_i=e_{n+1}$, and that $U_i$ is the graph of a function $x_{n+1} = h(x_1, \dots, x_n)$, where $h$ is defined on the projection $U_i'$ of $U_i$ on $\R^n$.  Since $ \langle a_i, \nm(p) \rangle \geq \frac{1}{2}$ on $U_i$, and since
\[
\nm(p) = \frac{1}{\sqrt{1+\|\nabla h\|^2}}
\begin{pmatrix}
-\nabla h \\ 1
\end{pmatrix}
\]
we have
\[
\sqrt{1+\|\nabla h\|^2} \leq 2.
\]
From $X= (x, h(x))$ we get $\|X\|^2 = \|x\|^2 + h(x)^2$, and thus
\begin{align*}
\int_{U_i} e^{-\|X\|^2/4} d\mu
&= \int_{U_i'} e^{-\|x\|^2/4} e^{-h(x)^2/4} \sqrt{1+\|\nabla h\|^2} dx\\
&\leq 2 \int_{U_i'} e^{-\|x\|^2/4} dx \\
&\leq 2 \int_{\R^n} e^{-\|x\|^2/4} dx \\
&= 2 (4\pi)^{n/2}.
\end{align*}
Since $M$ is covered by $N_n$ sets $U_i$, we find that $\hu(M)\leq 2N_n$.
\end{proof}
\begin{prop}
\label{prop-sphere}
If the limit of $\bar{M}_\tau$ as $\tau\to -\infty$ is the sphere of radius $\sqrt{2n}$ then $\bar{M}_\tau$ is the sphere of radius $\sqrt{2n}$ for every $\tau\in (-\infty,\infty)$.
\end{prop}
\begin{proof}
The monotonicity of the Huisken functional \eqref{eq-mon-huisken} and Lemma \ref{lem-upp-bound} imply the existence of the limits
\[
\hu_- := \lim_{\tau\to -\infty} \hu(\bar{M}_\tau), \quad \text{and} \quad \hu_+ := \lim_{\tau\to\infty} \hu(\bar{M}_\tau).
\]
From the earlier work of Huisken we know that $\bar{M}_\tau$ converges as $\tau\to \infty$ to a sphere of radius $\sqrt{2n}$ as well and therefore $\hu_- = \hu_+$.  Since $\hu(\bar{M}_\tau)$ is monotone along the flow we have $\hu(\tau) = \hu_- = \hu_+$ is constant along the flow and therefore,
\[
\bar{H} + \frac{1}{2} \langle\bar{X}, \nm \rangle = 0.
\]
All these imply $\bar{M}_\tau$ is the sphere of radius $\sqrt{2n}$ for all $\tau\in \R$.
\end{proof}

This proposition directly implies the main Lemma~\ref{lem-only-cylinders} of this section.

\section{Geometric properties of ancient solutions}
\label{sec-geometric-prop}
Let $\{M_t\}_{t<0}$ be an ancient convex $O(1)\times O(n)$ symmetric solution to MCF, and let $\bar{M}_\tau = e^{\tau/2}M_{-e^{-\tau}}$ be its parabolic blow-up.  Our goal in this section is to prove the following lemma.
\begin{lemma}
\label{lem-Hmax-at-tip}
The mean curvature on $M_t$ attains its maximum at the tips $x=\pm d(t)$.
\end{lemma}
\begin{corollary}
\label{cor-Hmax-less-than-diam}
If the solution is defined for all $t<0$, then the maximal mean curvature and the extrinsic diameter $d(t)$ satisfy
\begin{equation}
\label{eq-Hmax-diam-inequality}
(-t)\, H_{\max}(t) \leq d(t).
\end{equation}
For the blown-up solution $\bar{M}_\tau$ this implies
\begin{equation}
\label{eq-rescaled-Hmax-diam-inequality}
\bhm(\tau) \leq \bd(\tau).
\end{equation}
The rate at which the rescaled extrinsic diameter changes is bounded by
\begin{equation}
\label{eq-dbar-growth-bound}
\left|\bd\,'(\tau)\right| \leq \frac12 \bd(\tau).
\end{equation}
\end{corollary}

\begin{proof}[Proof of the corollary.]
The velocity of the tip is $d'(t) = -H_{\max}(t)$, and by Hamilton's Harnack inequality \cite{Ha} we know that $H_{\max}(t)$ is increasing.  Therefore we have for any $t<0$
\[
d(t) \geq \int_t^0 H_{\max}(t') \, dt' \geq H_{\max}(t) \cdot (-t),
\]
which proves \eqref{eq-Hmax-diam-inequality}.

The growth rate of the rescaled diameter follows from \eqref{eq-RMCF}, which tells us that
\[
\bd\,'(\tau) = -\bhm(\tau) + \tfrac12 \bd(\tau).
\]
The estimate \eqref{eq-dbar-growth-bound} now directly follows from \eqref{eq-rescaled-Hmax-diam-inequality} and $\bhm>0$.
\end{proof}

To prove Lemma~\ref{lem-Hmax-at-tip} we begin with listing a few consequences of the convexity of the surface. First, recall that by convexity and symmetry we have
\begin{equation}
\label{eq-p1}
u_y(\cdot,\tau) \leq 0 \quad \text{and} \quad 
u_{yy}(\cdot,\tau) \leq 0
\quad \text{for }  y >0.
\end{equation}
\begin{lemma}
\label{lem-P}
Set
\[
P:= - (\log u )_y = - u_y /u \geq 0, \qquad \text{on} \,\, y >0.
\]
We have
\begin{equation}
\label{eq-p2}
P_y \geq 0, \qquad  \qquad \text{on} \,\, y \geq 0
\end{equation}
with $P=0$ at $y=0$.
\end{lemma}
\begin{proof}
By direct calculation we have
\[
P_y = \frac{u_y^2- u\, u_{yy}}{u^2}.
\]
Hence, $P_y \geq 0$ for $y >0$ since $u_{yy} \leq 0$.
\end{proof}
\begin{lemma}
\label{lem-Q}
Set
\[
Q:=\frac{u_y^2}{u^2\, (1+u_y^2)}.
\]
We have
\[
Q_y = 2 \, (1+u_y^2) Q^{3/2} \geq 2\, Q^{3/2} \geq 0 \quad \text{for } y \geq 0.
\]
\end{lemma}
\begin{proof} By direct calculation and \eqref{eq-p1}, we have
\[
Q_y= - 2 \,\frac{ u_y \, \big (u_y^2 + u_y^4 - u\, u_{yy} \big )}{u^3 (1+u_y^2)^2 } \geq - 2 \,\frac{ u_y^3 }{u^3 (1+u_y^2) } = 2 \, (1+u_y)^2 Q^{3/2} \geq 2\, Q^{3/2}.
\]
\end{proof}
In this rotationally symmetric setting we have the following formulas for the principal curvatures $\lambda_i$ ($i=1, \dots, n$) in terms of $u$,
\begin{equation}
\label{eq-principal-curvatures}
\lambda_1 = \dots = \lambda_{n-1} = \frac{1}{u\, (1+u_y^2)^{1/2}}, \qquad \lambda_n = -\frac{u_{yy}}{(1+u_y^2)^{3/2}} \, .
\end{equation}
We consider the quantity
\[
R:= \frac{\lambda_n}{\lambda_1} = - \frac{u\, u_{yy}}{(1+u_y^2)} \geq 0.
\]
At umbilic points one has $R=1$.
\begin{lemma}
\label{lem-R}
On an ancient non collapsed convex solution of MCF with $O(1)\times O(n)$ symmetry one has $R \leq 1$.
\end{lemma}
\begin{proof}
We observe first that the tip of the surface is an umbilic point, so that we have $R=1$ at the tips for all $\tau$ (here we use that the surface is smooth and strictly convex and radially symmetric at the tip).  Hence, $R_{\max}(\tau)$ is achieved on the surface for all $\tau$, and is larger or equal than one.  We actually show that it is also not more than one.

Using the scaling invariance of $R$ and the fact that $U$ satisfies \eqref{eq-u}, we find
\[
R= \frac{-uu_{yy}}{1+u_y^2} = \frac{-UU_{xx}}{1+U_x^2} = 1-UU_t .
\]
We then compute
\begin{multline*}
R_t = \frac{R_{xx}} {1+U_x^2} - \frac{2U_x(1-R)} {U(1+U_x^2)}R_x\\
+ \frac{2U_x^2} {U^2(1+U_x^2)} \bigl(1-R^2\bigr) + (n - 2)\, \frac{2U_x^2} {U^2(1+U_x^2)} \, (1 - R) .
\end{multline*}
Since $R$ is scaling invariant we get in the $(y,\tau)$ variables
\begin{multline}
\label{eq-RRR}
R_\tau = \frac{R_{yy}} {1+u_y^2} -\frac{y}{2}R_y
- \frac{2u_y(1-R)} {u(1+u_y^2)}R_y \\
+ \frac{2u_y^2} {u^2(1+u_y^2)} \bigl(1-R^2\bigr) + (n - 2)\, \frac{2u_y^2} {u^2(1+u_y^2)}\, (1 - R).
\end{multline}
At a maximum of $R$ we have $R_y=0$ and $R_{yy}\leq 0$, so that the maximum of $R(\cdot,\tau)$ on the surface $\bar{M}_{\tau}$, $R_{\max}(\tau)$, satisfies
\[
\frac {d}{d\tau} R_{\max}(\tau) \leq - 2Q (R_{\max}^2-1) - 2(n - 2)Q (R_{\max} - 1),
\]
where $Q$ is as in Lemma~\ref{lem-Q}.

This ODE in particular shows that if $R_{\max}(\tau_0) >1$ for some $\tau_0$, then the same holds for all $\tau \leq \tau_0$.  In this case, if we denote by $\bar y_\tau >0$ any maximum point of $R(\cdot,\tau)$ on $[0,\bar d(\tau)]$, namely
\[
R_{\max} (\tau)=R(\bar y_\tau,\tau),
\]
we may assume $R_{\max}(\tau) > 1$ for all $\tau\le \tau_0$ since otherwise the statement of the Lemma is true.  With this assumption the following holds.

\begin{claim}
\label{claim-help}
If $R(\bar{y}_{\tau},\tau) > 1$ for all $\tau<\tau_0$, then
\[
\liminf_{\tau \to -\infty} Q(\bar y_\tau, \tau) \geq c > 0.
\]
\end{claim}\noindent%
Assuming the claim, and setting $s=-\tau >0$ and $\tilde R_{\max}(s) = R_{\max}(\tau)$, we obtain that
\begin{equation}
\label{eq-ode-R}
\frac {d}{ds} \tilde R_{\max}(s) \geq 
2Q   (\tilde R_{\max}^2-1).
\end{equation}
It follows from \eqref{eq-ode-R} that there exists $s_0 >0$ for which
\[
\frac {d}{ds} \tilde R_{\max}(s) \geq c \, (\tilde R_{\max}^2-1)
\]
for $s\geq s_0$.  This readily implies that $\tilde R_{\max}(s)\nearrow\infty$ as $s\nearrow s_*$ for some finite $s_*$, or equivalently $R_{\max} (\tau)$ blows up at some $\tau_* > -\infty$, which is a contradiction to the fact that our solution is ancient.
\end{proof}
\subsection{Proof of Claim \ref{claim-help}}
We know that $R(\bar y_\tau,\tau) >1$.  We also know that
\[
\lim_{\tau\to-\infty} u(y, \tau) = \sqrt{2(n-1)},
\]
in $C^\infty$ for bounded $y$, because $\bar M_\tau$ converges to the cylinder $\R\times S^{n-1}$.  This implies that
\[
\lim_{\tau\to -\infty} R(y,\tau) = 0
\]
uniformly for $y$ bounded.  We conclude that for all $\tau \leq \tau_0$ there exists at least a point $y_\tau$ such that
\[
0 < y_\tau < \bar y_\tau \quad \text{and} \quad R(y_\tau,\tau) =1.
\]
The convergence to the cylinder also implies that
\[
\lim_{\tau \to -\infty} y_\tau = +\infty.
\]
If we show that
\[
\liminf_{\tau \to -\infty} Q(y_\tau, \tau) \geq c_1 > 0
\]
the claim would follow from Lemma~\ref{lem-Q}.

To this end, it is more convenient to work with the original solution $U(x,t)$ of \eqref{eq-u-original}.  Setting
\[
x_t:= \sqrt{-t} \, y_\tau, \qquad \tau:=-\log(-t)
\]
we need to show that
\[
\liminf_{t \to -\infty} \frac{|t| \, U_x^2}{U^2 \, (1+U_x^2)}(x_t, t) \geq c_1 > 0,
\]
if $R(x_t,t) =1$.  Let us fix such a point $(x_{t_0},t_0)$.  We may assume, without loss of generality that
\[
\frac{ |t_0| \, U_x^2}{U^2 \, (1+U_x^2)}(x_{t_0}, t_0) < \delta^2
\]
with $\delta$ a sufficiently small number (to be chosen below).  Since $U^2/|t| \leq 4(n-1)^2$ for all $t$ sufficiently close to $-\infty$ (this follows from the convergence to the cylinder) it follows from the above inequality that
\[
U_x^2 \leq \frac{4\delta^2}{1-4\delta^2}:=\delta_1^2
\]
that can also be chosen sufficiently small.

We consider the rescaled solution $\tilde U$ to \eqref{eq-u-original} given by
\[
\tilde U(x,t) = \alpha^{-1}_0\, U(x_0+ \alpha_0 \, x, t_0 + \alpha^2_0\, t), \qquad \alpha_0:= U(x_0,t_0).
\]
The convergence to the cylinder implies that $\alpha_0 \ll x_0$ (otherwise $R(x_0,t_0)$ would have been close to zero).  In particular, for all $- 2 \leq x \leq 0$, we have $x_0+ \alpha_0 \, x >0$, hence the concavity $\tilde U_{xx} \leq 0$ implies
\begin{equation}
\label{eq-gb}
0 \leq - \tilde U_x (x,0)  \leq - \tilde U_x (0,0)  \leq \delta_1
\end{equation}
from which the bound
\begin{equation}
\label{eq-sb}
1   \leq  \tilde U(x,0)  \leq 2
\end{equation}
also follows.

We next recall that the solution $z=\tilde U(x,t)$ is a rotationally symmetric $\alpha$-noncollapsed solution to the mean curvature flow (near the chosen point), for some $\alpha > 0$.  Moreover, at this point we have $\lambda_1=\lambda_n$ and $\tilde U(0,0)=1$ together with $\tilde U_x^2(0,0) \leq \delta_1^2$.  Let $X_0=(0,1) \in \R^{n+1}$ denote the corresponding point on the surface $z=\tilde U(x,t)$.  From the curvature estimates in \cite{HK} it follows that there exists a number $\rho \in (0,1)$ depending on $\alpha$, for which we have
\begin{equation}
\label{eq-AAA}
|\nabla A | \leq C(\alpha) \rho^{-1}
\end{equation}
on all points of the surface $z=\tilde U(x,t)$ that intersect the parabolic ball $ P(X_0,\rho):=B_\rho(X_0) \times (t_0-\rho^2, t_0].$ Here $B_\rho(X_0)$ denotes a ball in $\R^3$.  Let $\tilde \lambda_n:= - \tilde U_{xx}/(1+\tilde U_x^2)^{3/2}$.  The bounds \eqref{eq-gb}, \eqref{eq-sb} and \eqref{eq-AAA} imply that
\[
\frac d{dx}\tilde \lambda_n (x,0) \leq C(\alpha), \qquad \qquad \text{for all} \,\, x \in (-\rho/2,0).
\]
Since $\tilde \lambda_n(0,0) \geq 1/2$ (if $\delta_1$ in \eqref{eq-gb} is chosen sufficiently small) we conclude that
\[
\tilde \lambda_n(x,0) \geq 1/4, \qquad \text{for all} \,\, x \in (-\rho_1,0)
\]
for some number $\rho_1=\rho_1(\alpha)$.  This implies the bound
\[
- \tilde U_{xx}(x,0) \geq 1/4, \qquad \text{for all} \,\, x \in (-\rho_1,0)
\]
and contradicts \eqref{eq-gb} if $\delta_1$ is chosen sufficiently small (since $\rho_1 = \rho_1(\alpha)$ is independent of $\delta_1$), hence finishing the proof of the claim.

We have the following two immediate corollaries of Lemma \ref{lem-R}.
\begin{corollary}
\label{cor-mon-l}
For every $\tau < \tau_0$ we have that $(\lambda_i)_y \ge 0$, implying that
\[
\lambda_i(y,\tau) \geq \lambda_i(0,\tau) \geq c >0, \qquad y\in [0,\bd(\tau)],
\]
for $1 \le i \le n-1$.
\end{corollary}
\begin{proof}
Recall that $\lambda_i^{-1} = u\, \sqrt{1+u_y^2}$ implying that
\[
\left(\frac{1}{\lambda_i}\right)_y = \frac{u_y\, (1+u_y^2 + uu_{yy})}{\sqrt{1+u_y^2}} \le 0, \qquad y \in [0,\bd(\tau)]
\]
since $u_y \le 0$ for $y \ge 0$, and since $1 + u_y^2 + uu_{yy} \ge 0$ is being equivalent to $R \le 1$, which we showed to be true in Lemma \ref{lem-R}.  On the other hand, since the $\lim_{\tau\to -\infty} \lambda_i(0,\tau) = {1}/{\sqrt{2\, (n-1)}}$ we immediately obtain the statement of the corollary.
\end{proof}
\begin{corollary}
\label{cor-tip}
For each $\tau$ the rescaled mean curvature $\bar{H}(\cdot,\tau)$ achieves its maximum at the tip $\bd(\tau)>0$.
\end{corollary}
\begin{proof} Let $\lambda_1$, $\lambda_n$ be the two principal curvatures.  Then, at any point $y < \bd(\tau)$, we have
\[
\bar{H}(y,\tau) = (n-1)\, \lambda_1(y,\tau) + \lambda_n(y,\tau) \leq n\, \lambda_1(y,\tau) \leq n\, \lambda_1(\bd(\tau),\tau) = \bar{H}(\bd(\tau),\tau)
\]
where we used the previous corollary and the fact that $\lambda_1= \lambda_n$ at the tip.
\end{proof}
In the corollary that follows, we give a different proof than the one of Proposition \ref{prop-sphere}, of the fact that if the backward limit as $\tau\to -\infty$ is a sphere then the ancient solution is the contracting sphere solution.
\begin{corollary}
If
\[
\lim_{\tau \to -\infty} R_{\min}(\tau)=1,
\]
namely, if the backward limit as $\tau\to -\infty$ is a sphere, then the ancient solution is the family of contracting spheres.
\end{corollary}
\begin{proof}
It follows by \eqref{eq-RRR} that
\[
\frac {d}{d\tau} R_{\min}(\tau) \geq \frac{2u_y^2}{u^2(1+u_y^2)} \, (1-R_{\min}^2(\tau)) + (n - 2)\, \frac{2u_y^2}{u^2(1+u_y^2)}\, (1 - R_{\min}(\tau)).
\]
Hence, if $R_{\min}(\tau_0) \leq 1-\delta$ for some $\tau_0$, then $R_{\min}(\tau) \leq 1-\delta$ for all $\tau \leq \tau_0$ and in particular
\[
R_{\min}(-\infty) \leq 1-\delta.
\]
This implies that if the backwards limit is the sphere which means that $R_{\min}(-\infty) =1$, then $R_{\min} \geq 1$ for all $\tau$.  Since by Lemma \ref{lem-R} we also have $R_{\max} \leq 1$, we conclude that $R \equiv 1$ for all $\tau$, hence the ancient solution is the family of contracting spheres.
\end{proof}

\section{A priori estimates}
\label{sec-apriori}

In this section we show several a priori estimates for our ancient solutions that will be used afterwards.  From now on we adopt the following convention: we will use symbols $C > 0$ and $\tau_0 < 0$ for uniform constants that can change from line to line in computations and by $\delta(\tau)$ a function such that the $\lim_{\tau\to -\infty} \delta(\tau) = 0$ and which can also change from line to line.

\subsection{Pointwise estimates}
The rescaled surfaces $\bar{M}_\tau$ converge to the cylinder with radius $\sqrt{2(n-1)}$ as $\tau\to-\infty$ so we know that for any $\tau_0$ one has
\begin{equation}
\label{eq-u0-trivial-bound}
u(0, \tau) = \sqrt{2(n-1)} + \delta(\tau).
\end{equation}
From the concavity of $u$ it follows that the graph of $u(\cdot, \tau)$ lies above the line connecting $(0, u(0, \tau))$ with the tip $(\bd(\tau), 0)$.  Thus we have
\begin{equation}
\label{eq-u-convexity-lower-bound}
u(y, \tau) \geq u(0, \tau) \frac{\bd(\tau)- y} {\bd(\tau)}
\qquad
(0\leq y\leq \bd(\tau)).
\end{equation}
Since $u(y, \tau)$ is decreasing in $y$ this implies
\[
0\leq u(0, \tau) - u(y, \tau) \leq u(0, \tau) \frac{y} {\bd(\tau)}
\]
and also
\begin{align*}
\left| u(y, \tau) - \sqrt{2(n-1)} \right|
&\leq \left| u(0, \tau) - \sqrt{2(n-1)} \right| + u(0, \tau) \frac{y}{\bd(\tau)}\\
&= \delta(\tau) + u(0, \tau) \frac{y} {\bd(\tau)}\\
&= \delta(\tau) + \sqrt{2(n-1)}\,\frac{y} {\bd(\tau)},
\end{align*}
where we have used $u(0, \tau) = \sqrt{2(n-1)}+\delta(\tau)$ and $y\leq \bd(\tau)$ in the last step.

For future reference we note that if we $\epsilon(\tau)\in(0,1)$ be given, then the above estimate leads to
\begin{equation}
\label{eq-v-small}
|u(y,\tau) - \sqrt{2(n-1)}| \le C\, \epsilon(\tau) + \delta(\tau), \qquad \text{for $y\in [0,\epsilon(\tau)\bd(\tau))$}.
\end{equation}

\subsection{First derivative estimates}
Due to symmetry, with no loss of generality assume $y \ge 0$. The concavity of $u$ also tells us the graph of $u(\cdot, \tau)$ lies on one side of any tangent line $t_y$ (where $t_y$ is the tangent line to the graph at point $y$).  In particular, if we write
\[
y_0(\tau) = -\frac{u}{u_y} + y
\]
for the location where the tangent line $t_y$ intersects the $y$-axis, then the concavity of $u$ implies
\[
\bd(\tau) \leq -\frac{u}{u_y} + y
\]
or equivalently, for $y \in (0, \bd(\tau))$,
\begin{equation}
\label{eq-convexity}
0 \leq - \frac{u_y}{u} \le \frac{1}{\bd(\tau) - y}.
\end{equation}
This means that for $y\in (0, \alpha\, \bd(\tau))$ and any $\alpha < 1$ we have
\[
0 \leq -\frac{u_y}{u} < \frac{1}{(1 - \alpha) \bd(\tau)},
\]
or equivalently,
\begin{equation}
\label{eq-first-der1}
|u_y| \le \frac{u}{\bd(\tau) - y} \le \frac{C}{(1-\alpha)\, \bd(\tau)} =: \frac{C(\alpha)}{\bd(\tau)}.
\end{equation}

\subsection{Higher derivative estimates}
\begin{lemma}
\label{lem-loc-est}
For any $\alpha<1$ there exist constants $C(\alpha) > 0$ and $\tau_0=\tau_0(\alpha)$ so that
\begin{equation}
\label{eq-der-inter}
|u_{yy}| + |u_{yyy}| \le \frac{C(\alpha)}{\bd(\tau)}, 
\qquad 
\text{for all $y\in (0, \alpha\bd(\tau))$ and $\tau<\tau_0$.}
\end{equation}
\end{lemma}
\begin{proof}
To obtain higher order derivative estimates on $u$ we first differentiate the evolution equation \eqref{eq-u} with respect to $y$.  If we write $\ud := u_y$ then we obtain
\[
\frac{\pd \ud}{\pd\tau} = \frac{\ud_{yy}}{1+\ud^2} - \frac{2\ud \ud_y^2}{(1+\ud^2)^2} - \frac y2\, \ud_y + \frac{(n-1)\,\ud}{u^2}.
\]
We will localize the proof of our desired estimate \eqref{eq-der-inter} by introducing the following change of variables.  Given a point $(y_0, \tau_0)$ in space-time with $y_0 \leq \alpha \bd(\tau_0)$, we let
\[
\bar{\ud}(\eta,\tau) := \ud\bigl(y_0e^{\tau/2} + \eta, \tau_0 + \tau\bigr).
\]
If we choose $-\tau_0$ large enough, depending on $\alpha\in(0,1)$, then this function is defined on the rectangle
\[
Q := \{(\eta,\tau)\,\,\,|\,\,\, |\eta| \le 1, -1 \le \tau \le 0\}.
\]
To see this, recall that the diameter $d(t)$ for the unrescaled mean curvature flow is monotonically decreasing.  Since $d(t)$ is related to $\bd(\tau)$ by $\bd(\tau) = e^{\tau/2}d(-e^{-\tau})$, we know that $e^{-\tau/2} \bd(\tau)$ is a decreasing function of $\tau$, and thus
\begin{equation}
\label{eq-comp-diam1}
\bd(\tau_0) \le e^{-\tau/2} \bd(\tau_0 + \tau)
\text{ for $\tau\in [-1,0]$.}
\end{equation}
For any $\alpha<1$, we choose some $\alpha'\in(\alpha, 1)$, e.g.~$\alpha'= (1+\alpha)/2$.  We also choose $\tau_0(\alpha)$ so that
\[
\alpha\bd(\tau') +1 \leq \alpha'\bd(\tau') \quad \text{for all $\tau'<\tau_0$}.
\]
For any $(\eta, \tau) \in Q$ we then get
\[
y_0 e^{\tau/2} + \eta \leq \alpha \bd(\tau_0) e^{\tau/2} + 1 \leq \alpha\bd(\tau_0+\tau) + 1 \leq \alpha'\bd(\tau_0+\tau).
\]
It follows that $\bar{\ud}(\eta,\tau) = \ud(y_0e^{\tau/2}+\eta, \tau_0+\tau)$ is indeed defined on $Q$.

A computation shows that $\bar{\ud}$ satisfies
\[
\frac{\pd\bar{\ud}}{\pd\tau}
= \frac{\bar{\ud}_{\eta\eta}}{1+\bar{\ud}^2}
  - \frac{2\bar{\ud} \bar{\ud}_{\eta}^2}{(1+\ud^2)^2}
  - \frac{\eta}{2}\, \bar{\ud}_{\eta}
  + \frac{(n-1)\,\bar{\ud}}{u^2}.
\]
which we can write as
\begin{equation}
\label{eq-vbar-evolution}
\frac{\pd\bar{\ud}}{\pd\tau}
= a(\eta,\tau,\bar{\ud},\bar{\ud}_{\eta})\, \bar{\ud}_{\eta\eta} + b(\eta,\tau,\bar{\ud},\bar{\ud}_{\eta}),
\end{equation}
where
\[
a(\eta,\tau,\bar{\ud},p) = \frac{1}{1+\bar{\ud}^2}\;,
\qquad
b(\eta,\tau,\bar{\ud},p) =
\frac{(n-1)\, \bar{\ud}}{u^2}
- \frac{\eta}{2}\, p
- \frac{2\bar{\ud} p^2}{(1+p^2)^2}\;.
\]
The estimate \eqref{eq-convexity} combined with $u(0, \tau) = \sqrt{2(n-1)}+\delta(\tau)$ (see \eqref{eq-u0-trivial-bound}) tell us that on the rectangle $Q$ we can bound $\bar{\ud} = u_y(y_0e^{\tau/2} +\eta,\tau_0+\tau)$ by
\[
|\bar{\ud}|
\le \frac{C}{\bd(\tau_0+\tau) - y_0e^{\tau/2}+\eta)}
\le \frac{C}{\bd(\tau_0+\tau) - y_0 e^{\tau/2} -1}.
\]
By \eqref{eq-comp-diam1} we have $e^{\tau/2}\bd(\tau_0) \le \bd(\tau_0+\tau)$, which then implies
\[
|\bar{\ud}|\le \frac{C}{e^{\tau/2}\bd(\tau_0) - y_0 e^{\tau/2} -1}
\le \frac{C}{\bd(\tau_0) - y_0 - e^{-\tau/2}},
\]
for $\tau\in [-1,0]$.  Since $y_0\leq \alpha\bd(\tau_0)$, we have
\[
y_0+e^{-\tau/2} \leq y_0+e^{1/2}\leq \alpha'\bd(\tau_0),
\]
with $\alpha'=(1+\alpha)/2$, assuming again that $-\tau_0$ is large enough.  In the end we get the following estimate for $\ud$ on the rectangle $Q$
\begin{equation}
\label{eq-v-bound-on-Q}
|\bar{\ud}| \leq \frac{C(\alpha)}{\bd(\tau_0)}.
\end{equation}
We apply this bound to the coefficients $a$ and $b$ in the equation \eqref{eq-vbar-evolution} for $\ud$.  For $a$ we get
\[
\kappa \le a(\eta,\tau,\bar{\ud},p) \le 1,
\]
where $\kappa = \kappa(\alpha)$.

The lower bounds \eqref{eq-v-bound-on-Q} for $\bar{\ud}$ and \eqref{eq-u-convexity-lower-bound} for $u$ imply
\[
|b(\eta,\tau,\bar{\ud},p)| \le \frac{C}{\kappa}\, (1+p^2)\, a(\eta,\tau,\bar{\ud},p).
\]
As a consequence of these bounds on the coefficients $a$ and $b$, the classical interior estimates are available for equation \eqref{eq-vbar-evolution} (see \cite{LSU}).  We get
\[
|\bar{\ud}_{\eta}(0,0)| + |\bar{\ud}_{\eta\eta}(0,0)|
\le C_0\sup_Q |\bar{\ud}(\eta,\tau)
\le \frac{C(\alpha)}{\bd(\tau_0)}.
\]
Finally, since $\bar{\ud}_\eta(0,0) = u_{yy}(y_0,\tau_0)$ and $\bar{\ud}_{\eta\eta}(0,0) = u_{yyy}(y_0,\tau_0)$ this completes the proof of Lemma~\ref{lem-loc-est}.
\end{proof}
\subsection{Lower barriers}
\label{sec-lower}
At this point we know nothing about the extrinsic diameter $\bd(\tau)$ beyond the facts that $\bd(\tau)\to\infty$ as $\tau\to-\infty$, and that the growth of $\bd(\tau)$ is bounded by \eqref{eq-rescaled-Hmax-diam-inequality} and \eqref{eq-dbar-growth-bound}.  In this section we will show that the magnitude of $\bd(\tau)$ is determined by how much the solution deviates from the cylinder $r=\sqrt{2(n-1)}$ in the parabolic region $|y|=\cO(1)$.  We also find a lower bound for $u(y, \tau)$ in the region $y\ge M$ in terms of $u(M, \tau)$.  Our proof of these lower bounds relies on a foliation of one end of the interior of the cylinder with radius $\sqrt{2(n-1)}$ whose leaves are ``self-shrinkers,'' i.e., stationary surfaces for the rescaled MCF \eqref{eq-RMCF}, which satisfy
\begin{equation}
\label{eq-self-shrinker}
H+\tfrac12 \vX\cdot\nm = 0.
\end{equation}
For rotationally symmetric surfaces, obtained by rotating the graph of $r=u(y)$ about the $y$-axis, this equation is equivalent with
\begin{equation}
\label{eq-ss}
\frac{u_{yy}}{1+u_y^2} - \frac y2 \, u_y + \frac u2 - \frac{n-1}{u} = 0.
\end{equation}
The solutions to this ODE are geodesics in the upper half plane for the metric
\[
(ds)^2 = u^{n-1} e^{-(u^2+r^2)/2}\, \{(du)^2 + (dr)^2\},
\]
and the ODE can be written as
\begin{equation}
\label{eq-ss-2}
k + \frac y2 \sin\theta + \left(\frac u2 - \frac{n-1}{u}\right) \cos\theta = 0,
\end{equation}
where $k$ is the curvature of the graph of $u$ and $\tan\theta = u_y$.

\begin{figure}[t]
\includegraphics[width=\textwidth]{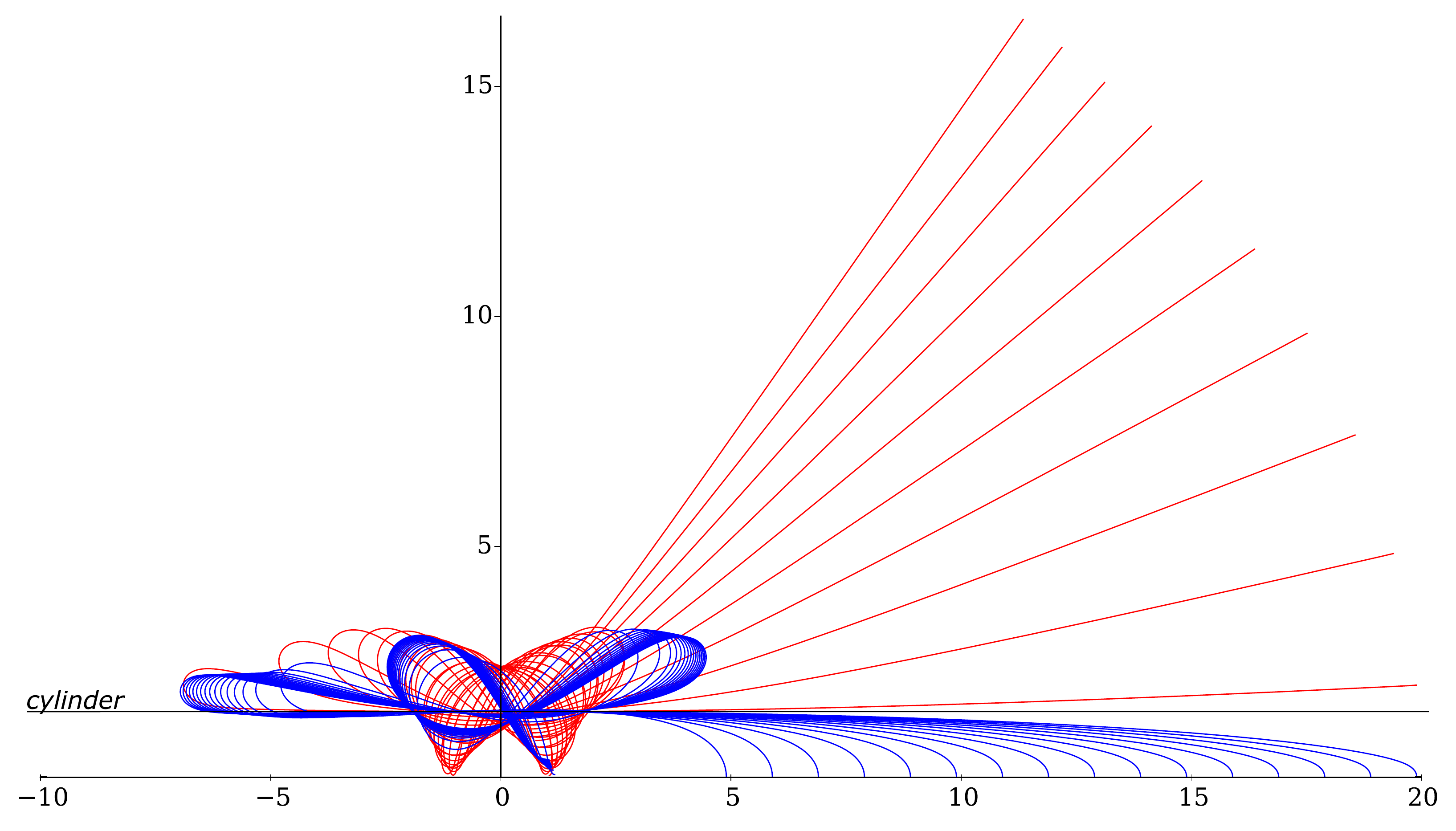}
\caption{Some self-shrinkers $\Sigma_a$ and $\tS_b$, in dimension $n=2$ for various values of $a$ and $b$.  See also \cite{AngDoughnuts,DruganThesis,KleeneMoller}.}
\label{fig-some-shrinkers}
\end{figure}

\subsubsection{Three lemmas about shrinkers}
The following lemmas guarantee the existence of self-shrinker segments and establish their asymptotic behaviour.
\begin{lemma}
\label{lem-u-A}
{\bfseries\upshape(a)} For every $a > 0$ there is a unique solution $u_a$ of \eqref{eq-ss} on the interval $0 \le y \le a$ with
\[
\lim_{y\nearrow a} u_a(y) = 0, \qquad
\lim_{y\nearrow a} u_a'(y) = -\infty.
\]
The function $u_a:[0,a] \to \R^+$ is concave.

{\bfseries\upshape(b)} For every $b>0$ there is a solution $\tu_b:[0,\infty)\to\R$ of \eqref{eq-ss} with
\[
\lim_{y\to\infty} \tu_b'(y) = \lim_{y\to\infty} \frac{\tu_b(y)} {y} = b.
\]
The function $\tu_b:[0,\infty)\to\R$ is convex.
\end{lemma}
We denote the corresponding surfaces by
\begin{equation}
\label{eq-Sigma-a-defined}
\begin{aligned}
\Sigma_a &= \bigl\{ \text{surface of revolution in $\R^{n+1}$ with profile $r=u_a(y)$, $0\le y\le a$} \bigr\}\\
\tS_b &= \bigl\{ \text{surface of revolution in $\R^{n+1}$ with profile $r=\tu_b(y)$, $0\le y <\infty$} \bigr\}
\end{aligned}
\end{equation}
The surfaces $\tS_b$ outside the cylinder were constructed by Kleene and M{\o}ller in \cite{KleeneMoller}.

\begin{lemma}
For large values of $a$ the solution $u_a$ satisfies
\begin{equation}
\label{eq-uA-outer-expansion}
u_a(y) 
= \sqrt{2(n-1)\Bigl(1-\bigl(\frac ya\bigr)^2\Bigr)} 
+ o(1) \qquad (a\to\infty)
\end{equation}
uniformly in $y\geq 0$.
\end{lemma}

\begin{lemma}
\label{lem-uA-inner-expansion}
On any bounded interval $0\leq y\leq M$ one has the following expansion
\begin{equation}
\label{eq-uA-inner-expansion}
u_a(y) = \sqrt{2(n-1)}  \Bigl(1 - \frac{y^2-2} {2a^2} \Bigr) + o(a^{-2}) \qquad (a\to\infty).
\end{equation}
\end{lemma}
We postpone the proofs until section~\ref{sec-constructing-the-foliation}.

\begin{figure}[t]
\includegraphics{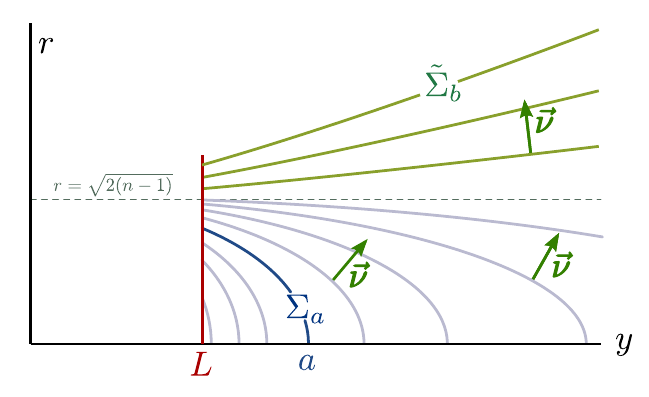}
\caption{The foliation by self-shrinkers $\Sigma_a$ and $\tS_b$ constructed in Lemma~\ref{lem-u-A}.  The unit normals $\nm$ of the foliation provide a calibration for Huisken's functional (see \S\ref{sec-about-the-foliation}).}
\label{fig-foliation-with-normals}
\end{figure}

\subsubsection{Lower bounds for $u(y, \tau)$ and $\bd(\tau)$}
As a corollary of asymptotic behavior of foliations $\Sigma_a$ and the convergence of our solution to a cylinder of radius $\sqrt{2(n-1)}$ we get the following lower bound on $\bd(\tau)$.

\begin{lemma}
\label{lem-lower-bound-for-ancient-solutions}
Let $u(y, \tau)$ be an ancient solution of rescaled mean curvature flow that is defined for $\tau\in(-\infty, \tau_0)$ and $M\leq y\leq \bd(\tau)$, and that satisfies
\begin{equation}
\label{eq-assume-convergence-to-cylinder}
\lim_{\tau \to-\infty} u(y, \tau) = \sqrt{2(n-1)}
\end{equation}
uniformly on any bounded interval $M\leq y\leq M'$.
Suppose also that we are given $\varepsilon>0$ and $\tau_\varepsilon\leq \tau_0$ such that
\begin{equation}
\label{eq-boundary-control}
u(M,\tau) \geq \sqrt{2(n-1)} - \varepsilon \text{ for all } \tau\leq \tau_\varepsilon.
\end{equation}
Then for any $a$ with $u_a(M) \leq \sqrt{2(n-1)} -\varepsilon$ one has
\[
u(y,\tau) \geq u_a(y) \text{ for all } \tau\leq\tau_\varepsilon, \quad M\leq y \leq \bd(\tau).
\]
In particular,
\[
\bd(\tau) \geq a \text{ for all } \tau\leq\tau_\varepsilon.
\]
\end{lemma}
\begin{proof} This follows directly from the maximum principle.  If $a$ is given then our assumption \eqref{eq-assume-convergence-to-cylinder} implies that $u(y,\tau) \to\sqrt{2(n-1)}$ as $\tau\to-\infty$ uniformly for $M\leq y\leq b$, so that $u(y,\tau) \geq u_a(y)$ as $\tau\to-\infty$.

The second assumption \eqref{eq-boundary-control} implies $u(M,\tau) \geq u_a(M)$ for all $\tau\leq\tau_\varepsilon$.  The maximum principle then leads to $u(y,\tau) \geq u_a(y)$ for all $\tau\leq\tau_\varepsilon$.
\end{proof}

By choosing the best $a$ for any given $\varepsilon$ and $\tau_\varepsilon$ and making an assumption about the rate of convergence in $\lim_{\tau\to-\infty} u(M,\tau) = \sqrt{2(n-1)}$, one can get time dependent lower bounds for $\bd(\tau)$.

\begin{corollary}
\label{cor-diam-lower}
Suppose $u$ is an ancient solution of rescaled mean curvature flow that satisfies~\eqref{eq-assume-convergence-to-cylinder} (i.e.~converges to the cylinder in backward time), and for which we have
\begin{equation}
\label{eq-lower-bound-at-M}
u(M,\tau) \geq \sqrt{2(n-1)} - \frac{KM^2}{-\tau}
\end{equation}
for some $K$, $M$.  Then there is a constant $K_1$ such that
\[
u(y,\tau) \geq u_{a(\tau)}(y),
\text{ where } a(\tau) = \sqrt{\frac{-\tau}{2K_1}},
\]
and in particular,
\[
\bd(\tau) \geq \sqrt{\frac{-\tau}{2K_1}}.
\]
\end{corollary}
\begin{proof}
For any given $\tau_1$ we let $\varepsilon = KM^2/(-\tau_1)$ and use \eqref{eq-uA-inner-expansion} to compute the optimal $a$ for which one has
\[
u_a(M) \leq \sqrt{2(n-1)} - \frac{KM^2}{-\tau'}
\]
for all $\tau'\leq\tau$.  Lemma~\ref{lem-lower-bound-for-ancient-solutions} then implies the lower bound for the diameter of an ancient solution.
\end{proof}

\subsection{The inner-outer estimate}

The shrinker foliation also allows us to derive another estimate that will prove to be very useful in section~\ref{sec-parabolic}.  This estimate provides an $L^2$ bound for the difference $v(y,\tau) = u(y,\tau) - \sqrt{2(n-1)}$ in the outer region in terms of the $L^2$ norm of $v$ in the inner region.  It is this estimate that helps us deal with the error terms that arise when we multiply the solution with a cut-off function.  In order to prove the estimate we will rely on the monotonicity of Huisken's functional $\hu$ defined in \eqref{eq-huisken-def}.  If $\Sigma$ denotes the cylinder, then the monotonicity implies that $\hu(\bar{M}_{\tau}) \le \hu(\Sigma)$ for all $\tau$.

In the next few subsections we will prove important estimates that hold for any surface $\Gamma$ with the property that $\hu(\Gamma) \le \hu(\Sigma)$ and that is close to a cylinder in the middle.  Since $\hu(\bar{M}_{\tau}) \le \hu(\Sigma)$ and since our surfaces $\bar{M}_{\tau}$ converge to a cylinder $\Sigma$, uniformly on compact sets, those estimates will hold for hypersurfaces $\bar{M}_{\tau}$ as well.

\subsubsection{The Huisken functional}
For hypersurfaces $\Gamma\subset\R^n$ the Huisken functional is defined by
\begin{equation}
\label{eq-huisken-def}
\hu(\Gamma) = (4\pi)^{-n/2} \int_\Gamma e^{-\phi} d\mu,
\end{equation}
where $\mu$ is $n$-dimensional surface measure on $\Gamma$, and where
\[
\phi(\vX) = \tfrac14 \|\vX\|^2.
\]

\subsubsection{Notation}
\label{sec-notation-rot-symm-surfaces}
We choose coordinates $(\vx, y)$ with $\vx\in\R^n$ and $y\in\R$, and consider surfaces which are rotationally symmetric around the $y$-axis.  The cylinder
\[
\Sigma = \bigl\{ (\vx, y) \in\R^n\times\R: \|\vx\| = \sqrt{2(n-1)} \bigr\}
\]
is stationary for the Huisken functional.

For any $a, b$ with $0\leq a < b\leq \infty$ and any hypersurface $\Gamma$ we define
\[
\Gamma_{ab} = \bigl\{ (\vx, y) : a<|y|<b \bigr\}.
\]

\subsubsection{Statement of the estimates}
We will prove a quantitative version of the following:
\begin{quote}\itshape%
Let $L>0$ be large enough.  If $\Gamma$ is a convex hypersurface with $\hu(\Gamma) \leq \hu (\Sigma)$ for which $\Gamma_{0L}$ ``is close to $\Sigma_{0L}$\!'' then $\Gamma_{L, 2L}$ must also ``be close to $\Sigma_{L, 2L}$.''
\end{quote}
More precisely, we let $L>0$ be given, and assume that $\Gamma$ is a convex hypersurface of revolution for which $\Gamma_{0, 4L}$ can be written as a graph over the cylinder $\Sigma_{0, 4L}$: i.e.~we assume $\Gamma$ is given by
\begin{equation}
\Gamma = \bigl\{ (\vx, y) : \|\vx\| = u(y), \ |y|\leq d\bigr\}
\end{equation}
for some concave function $r=u(y)$.

It will be very convenient to abbreviate
\[
v(y) = u(y) - \sqrt{2(n-1)}.
\]
We will assume that there is some $ \delta > 0$ for which
\begin{equation}
\label{eq-close-to-cylinder-hypothesis}
\sup_{|y|\leq 4L} |v(y)| 
=\sup_{|y|\leq 4L} \left| u(y) -\sqrt{2(n-1)} \right|  \leq \delta.
\end{equation}
Since $u(y)$ is concave, this implies
\begin{equation}
\label{eq-deriv-close-to-cylinder}
\sup_{|y|<3L}|v'(y)| = \sup_{|y|<3L}|u'(y)| \leq \frac{2\delta}{L}.
\end{equation}

\begin{lemma}
\label{lem-inner-outer}
There exist $L_0 > 0$ and $\delta_0>0$ such that for any convex hypersurface $\Gamma$ with profile $r=u(y)$ which satisfies $\hu(\Gamma) \le \hu(\Sigma)$ as well as \eqref{eq-close-to-cylinder-hypothesis} for some $\delta<\delta_0$ and $L>L_0$, one has
\[
\int_0^{2L} v_y^2 e^{-y^2/4} dy \leq C \int_0^L v^2\, e^{-y^2/4} dy,
\]
where $C$ is a constant that does not depend on $L$.
\end{lemma}

The important consequence of the previous Lemma is the following Corollary.

\begin{corollary}
\label{cor-inner-outer}
There is an $L_0 > 0$ and a $\delta_0>0$ such that for any convex hypersurface $\Gamma$ with profile $r=u(y)$ which satisfies $\hu(\Gamma) \le \hu(\Sigma)$ as well as \eqref{eq-close-to-cylinder-hypothesis} for some $\delta<\delta_0$ and $L>L_0$, one has
\[
\int_L^{2L} v^2 e^{-y^2/4} dy \leq \frac{C}{L^2} \int_0^L v^2\, e^{-y^2/4} dy,
\]
where $C$ is a constant that does not depend on $L$.
\end{corollary}

The proofs of Lemma \ref{lem-inner-outer} and Corollary \ref{cor-inner-outer} will occupy us during the following subsections.

\subsubsection{Surfaces with $\hu(\Gamma) \le \hu(\Sigma)$}
\label{sec-inner-outer}

If $\Gamma$ is obtained by revolving the graph of a function $r=u(y)$ around the $y$ axis, then the Huisken functional is given by
\begin{equation}
\label{eq-huisken-def-graphs}
\hu(u) = \int_{-d}^{d} u^{n-1}e^{-u^2/4}\sqrt{1+u_y^2}\;e^{-y^2/4}\,dy,
\end{equation}
where the integral is taken over the domain $(-d, d)$ of $u$.

Our hypothesis that $\Gamma$ has a lower Huisken functional than the cylinder $\Sigma$ implies that
\[
\hu_{0L}(\Gamma) + \hu_{L\infty}(\Gamma) \leq \hu_{0L}(\Sigma) + \hu_{L\infty}(\Sigma),
\]
or,
\begin{equation}
\label{eq-outer-bounds-inner}
\hu_{L\infty}(\Gamma) - \hu_{L\infty}(\Sigma)
\leq
\hu_{0L}(\Sigma) - \hu_{0L}(\Gamma).
\end{equation}

\subsubsection{The RHS of \eqref{eq-outer-bounds-inner}}
The terms on the right in \eqref{eq-outer-bounds-inner} represent how much $\Gamma$ deviates from the cylinder in the bounded region $y\leq L$.  In terms of the profile function $r=u(y)$ they are given by
\[
\hu_{0L}(\Sigma) - \hu_{0L}(\Gamma) = \int_0^L \Bigl\{ \big(2(n-1)/e\big)^{(n-1)/2} - u^{n-1} e^{-u^2/4} \sqrt{1+u_y^2} \Bigr\} e^{-y^2/4}dy.
\]
We observe that $\sqrt{1+u_y^2 } \geq 1$, and also that
\[
u^{n-1}e^{-u^2/4} \leq \big(2(n-1)/e\big)^{(n-1)/2} \text{ for all $u\in\R$},
\]
with equality at $u=\sqrt{2(n-1)}$.  It follows that there is a constant $C>0$ such that if $u(y)=\sqrt{2(n-1)}+v(y)$, then
\[
u^{n-1}e^{-u^2/4} \geq \; \left(\frac{2(n-1)}{e}\right)^{(n-1)/2}\Bigl\{1 - Cv^2 \Bigr\} \text{ when }|v|\leq \delta.
\]
In the region $0\leq y\leq 3L$ we have $|u_y|\leq \delta$, so that there is a constant $C$ such that
\[
\sqrt{1+u_y^2} \geq 1+ Cu_y^2.
\]
Combining these facts we obtain
\begin{align}
\label{eq-inner-bound}
\hu_{0L}(\Sigma) &- \hu_{0L}(\Gamma) \\
&= \int_0^L \Bigl\{ \big(2(n-1)/e\big)^{n-1} - u^{n-1} e^{-u^2/4} \sqrt{1+u_y^2}
\Bigr\} e^{-y^2/4}dy     \notag\\
&\leq\left(\frac{2(n-1)}{e}\right)^{n-1}\int_0^L \Bigl\{ 1 - \bigl(1-Cv^2\bigr) (1+Cu_y^2) \Bigr\}
e^{-y^2/4}dy  \notag\\
&\leq \left(\frac{2(n-1)}{e}\right)^{n-1}\int_0^L \Bigl\{ 1-\bigl(1-Cv^2\bigr) -C\bigl(1-Cv^2\bigr)u_y^2 \Bigr\}
e^{-y^2/4}dy.\notag\\
&\leq \left(\frac{2(n-1)}{e}\right)^{n-1}\int_0^L \Bigl\{ Cv^2 -C\bigl(1-C\delta^2\bigr)u_y^2 \Bigr\} e^{-y^2/4}dy.\notag
\end{align}
Here we have used $|v|\leq \delta$ in the last step.  Let us assume that $\delta$ is so small that $C\delta^2<\frac12$.  Then we can move the term with $u_y^2$ to the left, and we obtain
\[
\hu_{0L}(\Sigma) - \hu_{0L}(\Gamma) + c\int_0^L u_y^2 e^{-y^2/4} dy \leq C\int_0^L v^2 e^{-y^2/4}dy,
\]
where the constants $c, C$ do not depend on $L$.

Combined with \eqref{eq-outer-bounds-inner} this tells us that
\begin{equation}
\label{eq-better-outer-inner-bounds}
\hu_{L\infty}(\Gamma) - \hu_{L\infty}(\Sigma)
+ c\int_0^L v_y^2 e^{-y^2/4} dy
\leq
C\int_0^L v^2 e^{-y^2/4}dy,
\end{equation}
To complete the proof of Lemma~\ref{lem-inner-outer} we must therefore show that
\begin{equation}
\label{eq-outer-uy-bounds}
c\int_L^{2L} v_y^2 e^{-y^2/4} dy
\leq
\hu_{L\infty}(\Gamma) - \hu_{L\infty}(\Sigma)
\end{equation}
holds for some small constant $c$.

\subsubsection{Digression: the minimizing foliation inside the cylinder}
\label{sec-about-the-foliation}
Recall that $\Sigma_a$  and $\tS_b$ are rotationally symmetric self-shrinkers for MCF, where $\Sigma_a$ meets the $y$-axis at $y=a$, while $\tS_b$ is asymptotic to the cone with opening slope $b$.  We have used previously the existence and asymptotic properties of the $\Sigma_a$ to find lower bounds for $\bd(\tau)$.  Here we will use the fact that the $\Sigma_a$ and $\tS_b$ form a foliation of a region inside the cylinder to compare the Huisken functional of different sections of a convex surface.

\begin{lemma}
There exist $\delta>0$ and  $L_0>0$ such that the hypersurfaces $\Sigma_a$ and $\tS_b$ foliate the region
\[
\Omega_0 = \bigl\{ (y, \vx)\in\R\times\R^n : \|\vx\|\leq \sqrt{2(n-1)}+\delta, y\geq L_0\bigr\}.
\]
In particular, $\Omega_0$ is a disjoint union of the caps $\Omega_0\cap\Sigma_a$, and the unit normals $\nm$ to $\Sigma_a$ define a $C^1$ vectorfield on $\Omega_0$.
\end{lemma}
We prove this in section~\ref{sec-normal-variation}.

The unit normals $\nm$ can be written as
\begin{equation}
\label{eq-normal-angle-def}
\nm = \bigl(-\sin\varphi,  \cos \varphi\, \bomega\bigr)
\end{equation}
where $\varphi$ is defined by $\tan\varphi = u_y$ (see Figure~\ref{fig-foliation-with-normal-angle-phi}).  We regard $\varphi$ as a function of $(y, r)$.

\begin{figure}[h]
\includegraphics{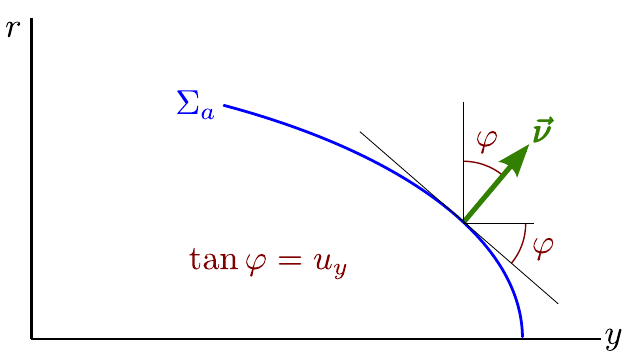}
\caption{The unit normal to the minimizing foliation.  Note that since $u_y<0$, the angle $\varphi$ always satisfies $-\frac{\pi}{2} <\varphi <0$.}
\label{fig-foliation-with-normal-angle-phi}
\end{figure}

\begin{lemma}
The vector field $e^{-\phi}\nm$ is divergence free.
\end{lemma}
\begin{proof}
We have $\nabla\cdot\nm = -H$ for the unit normals to any foliation.  The leaves $\Sigma_a$ all satisfy $H+\frac12 \vX\cdot\nm = 0$ so
\[
\nabla\cdot\bigl( e^{-\phi} \nm\bigr) = -e^{-\phi}(\nabla\phi)\cdot\nm + e^{-\phi}\nabla\cdot\nm= -e^{-\phi}\bigl(H+\tfrac12 \vX\cdot\nm\bigr)=0.
\]
\end{proof}

Near the cylinder $r=\sqrt{2(n-1)}$ the leaves $\Sigma_a$ are almost parallel to the cylinder, so that $\varphi(r, y)$ is small when $r\approx \sqrt{2(n-1)}$.  In our proof of the inner-outer Lemma~\ref{lem-inner-outer} we will need a more precise estimate of $\varphi(y, r)$ near the cylinder.
\begin{lemma}
\label{lem-tan-phi-at-cylinder}
There is a neighborhood of the cylinder $r=\sqrt{2(n-1)}$ on which one has
\[
\tan \varphi = \frac{w} {2ry} (r^2 - 2(n-1)),
\]
where the quantity $w$ satisfies
\begin{equation}
\label{eq-w-squeeze}
2\le w \le 2 + \frac{K}{y^2}
\end{equation}
for some constant $K$.
\end{lemma}
\begin{proof}[Proof of Lemma~\ref{lem-tan-phi-at-cylinder}]
This follows from \eqref{eq-angle-as-fn-of-yu} and the estimate \eqref{eq-w-upper-bound}.  Details are given in section~\ref{sec-constructing-the-foliation}.
\end{proof}

\subsubsection{Proof of \eqref{eq-outer-uy-bounds}}
To estimate the difference between $\hu_{L\infty}(\Sigma)$ and $\hu_{L\infty}(\Gamma)$, we consider the region $\Omega$ contained in the half space $y\geq L$, and bounded by the cylinder $\Sigma_{L\infty}$ and the surface $\Gamma_{L\infty}$.  The boundary of this region is
\[
\pd \Omega = \Sigma_{L\infty} \cup \Gamma_{L\infty} \cup \Delta_L,
\]
where $\Sigma_{L\infty} $ is the section of the cylinder on which $y\geq L$, and $\Delta_L$ is the annulus in the plane $y=L$ between the surface $\Gamma$ and the cylinder.

\begin{figure}[h]
\includegraphics{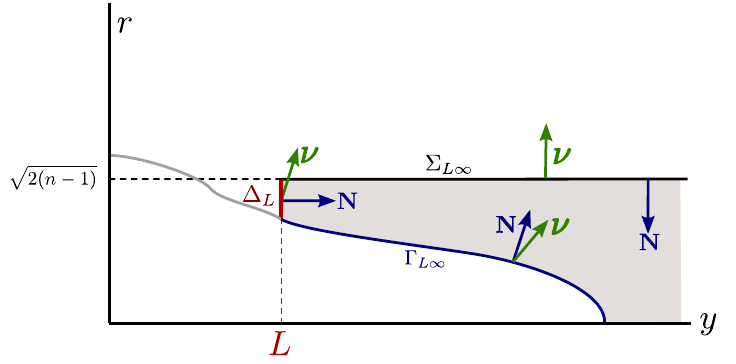}
\includegraphics{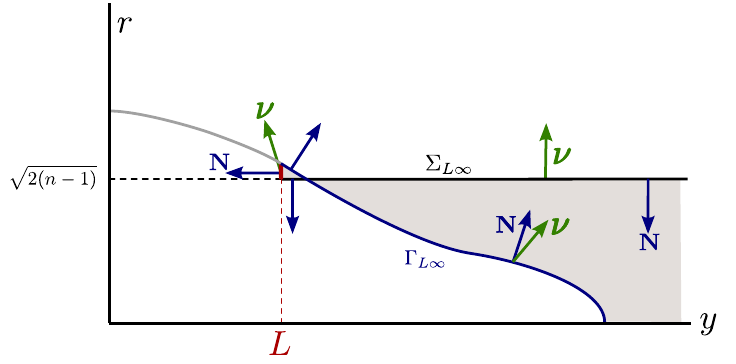}
\caption{The domain $\Omega$, the unit normals to $\pd\Omega$, and the vector field $\nm$.}
\label{fig-domain-Omega}
\end{figure}

Let $\nm$ be the unit normal vector to the minimizing foliation as above and $\vN$ the unit normal to $\pd\Omega$:  in the portion of $\Omega$ contained in the cylinder $r\le \sqrt{2(n-1)}$ we choose $\vN$ to be the inward normal, while outside of the cylinder we let $\vN$ be the outward normal (see Figure~\ref{fig-domain-Omega}).  By the Stokes' theorem we then have
\[
\int_{\Gamma_{L\infty}}(\vN\cdot\nm) e^{-\phi}d\mu + \int_{\Delta_L} (\vN\cdot\nm) e^{-\phi}d\mu +\int_{\Sigma_{L\infty}}(\vN\cdot\nm) e^{-\phi}d\mu =0.
\]
On the cylinder $\Sigma_{L\infty}$ we have $\vN=-\nm$, so $\vN\cdot\nm =-1$.  Therefore we get
\begin{align*}
\hu_{L\infty}(\Gamma)& - \hu_{L\infty}(\Sigma)\\
&= \int_{\Gamma_{L\infty}} e^{-\phi}d\mu - \int_{\Sigma_{L\infty}} e^{-\phi}d\mu \\
&= \int_{\Gamma_{L\infty}} \bigl( 1-(\vN\cdot\nm)\bigr) e^{-\phi}d\mu +\int_{\Gamma_{L\infty}} (\vN\cdot\nm) e^{-\phi}d\mu
+\int_{\Sigma_{L\infty}} (\vN\cdot\nm) e^{-\phi}d\mu\\
&= \int_{\Gamma_{L\infty}} \bigl( 1-(\vN\cdot\nm)\bigr) e^{-\phi}d\mu - \int_{\Delta_L} (\vN\cdot\nm) e^{-\phi} d\mu.
\end{align*}
Rearranging terms again, and using \eqref{eq-outer-bounds-inner} we find
\begin{align*}
\int_{\Gamma_{L\infty}} \bigl( 1-(\vN\cdot\nm)\bigr) e^{-\phi}d\mu &= \hu_{L\infty}(\Gamma) - \hu_{L\infty}(\Sigma)
+ \int_{\Delta_L} (\vN\cdot\nm) e^{-\phi} d\mu \\
&\leq \hu_{0L }(\Sigma) - \hu_{0L}(\Gamma) + \int_{\Delta_L} (\vN\cdot\nm) e^{-\phi} d\mu
\end{align*}
The integrand on the left is nonnegative, so we may restrict the integral to the smaller region $\Gamma_{L,2L} \subset \Gamma_{L\infty}$ and conclude
\begin{align}
\label{eq-basic-inner-outer-bound}
\int_{\Gamma_{L,2L}} \bigl( 1-(\vN\cdot\nm)\bigr) e^{-\phi}d\mu
&\leq \hu_{0L}(\Sigma) - \hu_{0L}(\Gamma) + \int_{\Delta_L} (\vN\cdot\nm) e^{-\phi} d\mu \\
&\leq C\int_0^L v^2 e^{-y^2/4}dy + \int_{\Delta_L} (\vN\cdot\nm) e^{-\phi} d\mu \notag
\end{align}
in view of \eqref{eq-inner-bound}.

We now write out the various quantities in terms of the function $u(y)$, keeping in mind the assumptions \eqref{eq-close-to-cylinder-hypothesis} and \eqref{eq-deriv-close-to-cylinder}, i.e.
\[
|v(y)|\leq \delta \text{ and } |u_y| \leq \frac{2\delta}{L} \text{ for } |y|\leq 3L.
\]
\subsubsection{The integral over $\Gamma_{L,2L}$}
We have
\[
\int_{\Gamma_{L,2L}} \bigl( 1-(\vN\cdot\nm)\bigr) e^{-\phi}d\mu = \int_L^{2L} \bigl( 1- \cos \theta \bigr) e^{-u^2/4}u^{n-1} \sqrt{1+u_y^2} \, e^{-y^2/4} dy.
\]
Here $\theta$ is the angle between $\nm$ and $\vN$.  See Figure~\ref{fig-N-dot-nu}.

We always have $\sqrt{1+u_y^2} \geq 1$.  By assumption we also have $|u-\sqrt{2(n-1)}|<\delta$ on the interval $L<y<2L$.  This leads to a lower bound
\begin{equation}
\int_{\Gamma_{L,2L}} \bigl( 1-(\vN\cdot\nm)\bigr) e^{-\phi}d\mu
\geq
c\int_L^{2L} \bigl( 1- \cos \theta \bigr) \, e^{-y^2/4} dy.
\end{equation}
The constant $c$ does not depend on $L$.
\begin{figure}\centering
\includegraphics{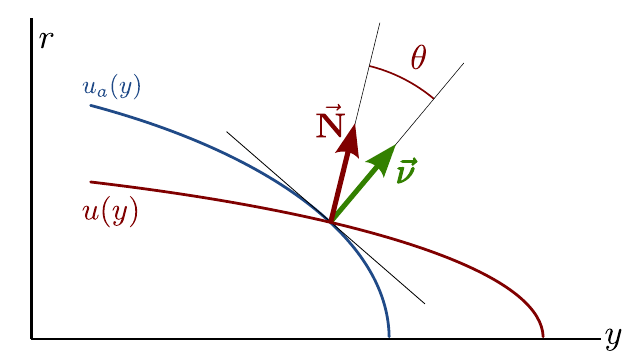}
\caption{The normals $\nm$ to a leaf of the foliation, and $\vN$ to the given hypersurface $\Gamma$.}
\label{fig-N-dot-nu}
\end{figure}
The angle $\theta$ is determined by the angle $\varphi$ defined in \eqref{eq-normal-angle-def} and the slope $u_y$ of the tangent to the graph of $u$.  Then
\[
\theta = \varphi(y, u(y)) - \arctan u_y(y).
\]
Since $|v| + |u_y|=O(\delta)$, we have
\[
|\theta| \geq \frac12 |\tan\varphi(y, u) - u_y|
\]
if we assume $\delta$ is small enough.  Combined with $1-\cos\theta \geq \frac14\theta^2$ for small $\theta$, we get
\begin{equation}
\label{eq-L2Lintegral-bound}
\int_{\Gamma_{L,2L}} \bigl( 1-(\vN\cdot\nm)\bigr) e^{-\phi}d\mu
\geq
\frac 1{16}\int_L^{2L} \bigl\{ u_y(y) - \tan\varphi(y, u(y))\bigr\}^2 \, e^{-y^2/4} dy
\end{equation}

\subsubsection{The term at $y=L$}
The term at $y=L$ in \eqref{eq-basic-inner-outer-bound} is
\begin{align*}
\int_{\Delta_L} (\vN\cdot\nm) e^{-\phi} d\mu
&=\int_{u(L)}^{\sqrt{2(n-1)}} (\vN\cdot\nm) \, e^{-L^2/4}e^{-\eta^2/4}\eta d\eta\\
&\leq \sqrt{2(n-1)} e^{-L^2/4} \int_{u(L)}^{\sqrt{2(n-1)}} (\vN\cdot\nm)\, d\eta.
\end{align*}

The unit normal $\vN$ to $\Delta_L$ is the unit vector parallel to the $y$-axis, so
\[
\bigl| (\vN\cdot\nm) \bigr| = \bigl| \sin \varphi(L, u)\bigr| \leq |\varphi(L, u)| \leq \frac{C}{L} |v|.
\]
It follows that
\begin{equation}
\label{eq-Delta-integral-bound}
\int_{\Delta_L} (\vN\cdot\nm) e^{-\phi} d\mu
\leq \frac CL e^{-L^2/4} |v(L)|^2.
\end{equation}
Combining \eqref{eq-Delta-integral-bound}, \eqref{eq-L2Lintegral-bound}, and \eqref{eq-basic-inner-outer-bound}, we arrive at
\begin{multline}
\label{eq-uy-tanphi-estimate}
\int_L^{2L} \Bigl( u_y(y) - \tan\varphi(y, u(y))\Bigr)^2 \, e^{-y^2/4} dy
\leq \\
\frac CL e^{-L^2/4} |v(L)|^2 + C \int_0^L v^2\, e^{-y^2/4} dy.
\end{multline}

\subsubsection{Removing $\tan\varphi$}
We wish to estimate the integral of $u_y^2$ directly instead of the integral of $(u_y-\tan\varphi)^2$ as in \eqref{eq-uy-tanphi-estimate}.  To do this we begin with
\begin{align*}
\int_L^{2L} u_y^2 e^{-y^2/4}dy
&= \int_L^{2L} \bigl(u_y-\tan\varphi+\tan\varphi\bigr)^2 e^{-y^2/4}dy\\
&\leq 2 \int_L^{2L} \bigl(u_y-\tan\varphi\bigr)^2 e^{-y^2/4}dy
+ 2 \int_L^{2L} \bigl(\tan\varphi\bigr)^2 e^{-y^2/4}dy\\
&\leq 2 \int_L^{2L} \bigl(u_y-\tan\varphi\bigr)^2 e^{-y^2/4}dy + C \int_L^{2L} v^2 e^{-y^2/4}dy
\end{align*}
where we have used $|\tan\varphi(y, u)| \leq \frac CL |v|$, which follows from Lemma~\ref{lem-tan-phi-at-cylinder}.

For any function $v(y)$ one has
\begin{align*}
\int_a^{b} v^2 e^{-y^2/4} dy
&=\int_a^{b} \frac{2v^2}{y} \frac y2 e^{-y^2/4} dy\\
&= \left[-\frac2y v^2 e^{-y^2/4}\right]_a^b
+  \int_a^b \Bigl(\frac4y vv_y - \frac{2} {y^2}v^2\Bigr) e^{-y^2/4}dy\\
&\leq \frac2a v(a)^2 e^{-a^2/4}
+  \int_a^b \frac4y vv_y e^{-y^2/4}dy\\
&\leq \frac2a v(a)^2 e^{-a^2/4}
+  \int_a^b \Bigl(\frac12 v^2 + \frac{8}{y^2}v_y^2\Bigr) e^{-y^2/4}dy\\
&\leq \frac4a v(a)^2 e^{-a^2/4} + \frac{16}{a^2} \int_a^b v_y^2 e^{-y^2/4}dy
\end{align*}
Apply this to $v=u-\sqrt{2(n-1)}$ on the interval $(L, 2L)$:
\begin{align*}
\int_L^{2L} u_y^2 &e^{-y^2/4}dy \\
&\leq 2 \int_L^{2L} \bigl(u_y-\tan\varphi\bigr)^2 e^{-y^2/4}dy
+ C \int_L^{2L} v^2 e^{-y^2/4}dy\\
&\leq 2 \int_L^{2L} \bigl(u_y-\tan\varphi\bigr)^2 e^{-y^2/4}dy + \frac{C} {L} v(L)^2 e^{-L^2/4} + \frac{C} {L^2} \int_L^{2L} u_y^2 e^{-y^2/4}dy
\end{align*}
Here $C$ does not depend on $L$, so if we assume that $L^2>2C$ then we can absorb the last integral in the integral on the left:
\[
\int_L^{2L} u_y^2 e^{-y^2/4}dy \leq 4 \int_L^{2L} \bigl(u_y-\tan\varphi\bigr)^2 e^{-y^2/4}dy + \frac{C} {L}v(L)^2 e^{-L^2/4}.
\]
Combine with \eqref{eq-uy-tanphi-estimate} and we find
\begin{equation}
\label{eq-uy2-estimate}
\int_L^{2L} u_y^2 e^{-y^2/4}dy
\leq  \frac{C} {L} v(L)^2 e^{-L^2/4}
+C\int_0^L v^2 e^{-y^2/4}dy.
\end{equation}

\subsubsection{The final estimate}
This inequality remains true if we replace $L$ by any $\lambda\in(\frac34 L, L)$:
\[
\int_{\lambda}^{2\lambda } u_y(y)^2 \, e^{-y^2/4} dy \leq \frac C{\lambda} e^{-\lambda^2/4} v(\lambda)^2 + C \int_0^\lambda v^2\, e^{-y^2/4} dy.
\]
For $\lambda\in(\frac34 L, L)$ we can reduce the domain of integration on the left to $L\leq y\leq \frac32 L$, and increase the domain of integration on the right to $0\leq y\leq L$.  This leads to
\[
\int_L^{\frac32 L} u_y(y)^2 \, e^{-y^2/4} dy \leq \frac C\lambda e^{-\lambda^2/4} v(\lambda)^2 + C \int_0^L v^2\, e^{-y^2/4} dy.
\]
Integrate both sides of this inequality over $\lambda\in(\frac34 L, L)$:
\[
\frac L4 \int_L^{\frac32 L} u_y^2 \, e^{-y^2/4} dy \leq \int_{\frac34L}^L \frac C \lambda e^{-\lambda^2/4} v(\lambda)^2 d\lambda + \frac{CL}{4} \int_0^L v^2\, e^{-y^2/4} dy.
\]
Clearing the constants and combining the two integrals on the right, we find
\[
\int_L^{\frac32 L} u_y^2 \, e^{-y^2/4} dy \leq C \int_0^L v^2\, e^{-y^2/4} dy.
\]
We proceed now to proving Corollary \ref{cor-inner-outer}.  First we show the following weighted estimate.
\begin{lemma}
\label{lem-weighted-H1-L2}
For any function $u \in C^1([0,\ell])$ we have
\begin{multline}
\label{eq-weighted-H1-L2}
\int_0^\ell   u_y^2 \, e^{-y^2/4} \, dy + \frac 14 \, \int_0^\ell  u^2 \,e^{-y^2/4} \, dy\\
\geq \frac14 \ell e^{-\ell^2/4}u(\ell)^2 + \frac 1{16} \int_0^\ell y^2 \, u^2 \, e^{-y^2/4} \, dy .
\end{multline}
\end{lemma}
\begin{proof}
We begin with
\[
0\leq \bigl( u_y - \tfrac y4 u\bigr)^2 = u_y^2 + \tfrac{y^2}{16} u^2 - \tfrac12 yuu_y
\]
and integrate by parts:
\begin{align*}
\int_0^\ell e^{-y^2/4}\bigl( u_y^2 + \tfrac{y^2}{16}u^2\bigr)dy
&\geq \int_0^\ell e^{-y^2/4} \tfrac12 yuu_y dy \\
&= \bigl[\tfrac14 ye^{-y^2/4} u^2\bigr]_0^{\ell}
- \int_0^\ell \Bigl(\tfrac14 - \tfrac18 y^2\Bigr)e^{-y^2/4} u^2 dy \\
&= \tfrac14\ell e^{-\ell^2/4}u(\ell)^2 +\int_0^\ell \tfrac18y^2u^2 e^{-y^2/4} - \int_0^\ell \tfrac14 u^2 e^{-y^2/4} dy
\end{align*}
Rearranging terms leads to \eqref{eq-weighted-H1-L2}.
\end{proof}
\begin{proof}[Proof of Corollary \ref{cor-inner-outer}]
If we apply Lemma \ref{lem-weighted-H1-L2} to $\ell = 2L$ and to $v = u -\sqrt{2(n-1)}$ we get
\[
\int_0^{2L} y^2 v^2 e^{-y^2/4}\, dy \le \int_0^{2L} v_y^2 e^{-y^2/4}\, dy + \frac14\, \int_0^{2L} v^2 e^{-y^2/4}\, dy.
\]
Applying Lemma \ref{lem-inner-outer} to the first term on the left hand side of the previous estimate leads to
\[
\int_0^{2L} y^2 v^2 e^{-y^2/4}\, dy \le C\, \int_0^L v^2 e^{-y^2/4}\, dy + \frac 14 \int_L^{2L} v^2 e^{-y^2/4}\, dy,
\]
which, since
\[
L^2\int_L^{2L} v^2 e^{-y^2/4}\, dy \le \int_0^{2L} y^2 v^2 e^{-y^2/4}\, dy,
\]
yields the desired estimate
\[
\int_L^{2L} v^2 e^{-y^2/4}\, dy \le \frac{C}{L^2}\, \int_0^L v^2 e^{-y^2/4}\, dy,
\]
for $L$ sufficiently big.
\end{proof}
\section{Asymptotics of the parabolic region}
\label{sec-parabolic}

The goal in this section is to prove part (i) of our Main Theorem~\ref{thm-asymptotics}, as well as Corollary~\ref{cor-exact-diam}.

In the following we derive the asymptotics in the parabolic region $|y|=\cO(1)$ in a few steps.  We first analyze the spectrum of the linear operator $\cL$.  Then we project $\bv$ onto the positive, zero, and negative eigenspaces of $\cL$.  Using the \textit{a priori} estimates from section \ref{sec-apriori} we carefully estimate the error terms \eqref{eq-error-1} and \eqref{eq-error-2} and then, using the ODE arguments developed in \cite{FK} and \cite{MZ}, we are able to prove that as $\tau\to -\infty$ either the projection of $\bv $ onto the zero subspace dominates or the projection of $\bv $ onto the positive subspace of $\cL $ dominates.  We show the latter can not happen.  Once we establish the dominance of the zero eigenspace projection of $\bv$ we employ again our \textit{a priori} estimates and ODE arguments to show the precise asymptotics as stated in part (i) of Theorem~\ref{thm-asymptotics}.

\subsection{Linearization at the cylinder}
We are in the case where $\bar{M}_\tau$ converges to the cylinder $\Sigma$ with radius $\sqrt{2(n-1)}$ as $\tau\to -\infty$, uniformly in compact sets, i.e.
\[
\lim_{\tau\to-\infty}u(y, \tau) = \sqrt{2(n-1)}
\]
uniformly on bounded $y$ intervals.  As a measure for the difference between the solution and the cylinder we introduce $v(y,\tau)$ defined by
\[
u(\cdot,\tau) = \sqrt{2(n-1)} \bigl(1 + v(y, \tau)\bigr)
\]
This function satisfies
\[
\frac{\pd}{\pd\tau}v = \frac{v_{yy}}{1+2(n-1)v_y^2} - \frac{y}{2}v_y + \frac{2+v}{2+2v}v,
\]
which we can rewrite as
\begin{equation}
\label{eq-vv}
\frac{\pd}{\pd\tau} v = \cL[v] + E,
\qquad
|y| \le \bd(\tau),
\end{equation}
where, by definition,
\begin{equation}
\label{eq-linear}
\cL[\psi] = \psi_{yy} - \frac y2\, \psi_y + \psi
\end{equation}
and
\[
E = -\frac{v^2}{2(1+v)} - \frac{2(n-1)v_y^2 v_{yy}}{1+2(n-1)v_y^2}.
\]

Since $v(y,\tau)$ is not defined on all of $\R$ we truncate it smoothly outside the region $|y|\ge 2\ell(\tau)$, where, by definition,
\[
\ell(\tau) := \bd(\tau)^{1/3}.
\]
The choice of the exponent $1/3$ is to some extent arbitrary.  It will be mostly important that $\ell(\tau)\to\infty$, while $\ell(\tau)/\bd(\tau) \to 0$ as $\tau\to-\infty$.

\begin{figure}[t]\centering
\includegraphics[width=0.8\textwidth]{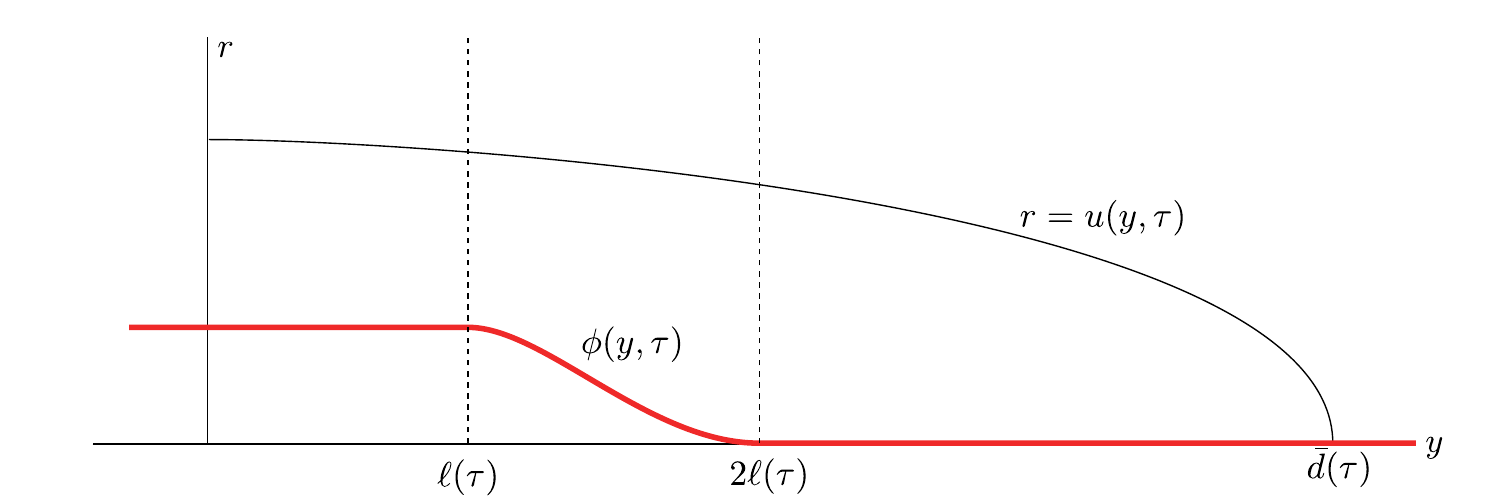}
\caption{The cut-off function $\phi(y, \tau)$ and the intermediate length $\ell(\tau) = \bd(\tau)^{1/3}$.}
\end{figure}
We now choose $\bphi\in C^{\infty}(\R)$ with $\bphi(y)=1$ for $|y| \le 1$ and $\bphi(y)=0$ for $|y| \ge 2$, and define
\[
\phi(y,\tau) := \bphi\left(\frac{y}{ \ell(\tau)}\right)
\text{ and }
\bv (y,\tau) := v(y,\tau) \phi(y,\tau).
\]
Our truncated function $\bv$ is now defined for all $y\in \R$, simply by setting $\bv(y, \tau)=0$ for $|y|\ge 2\ell$.  The equation for $\bv (y,\tau)$ is as follows:
\begin{equation}
\label{eq-bar-v}
\frac{\pd}{\pd\tau} \bv  = \cL[\bv] + \tilde{E}, \qquad y, \tau \in \R,
\end{equation}
where $\tilde{E} = \tilde{E}_1 + \tilde{E}_2 + \tilde{E}_3$,
\begin{equation}
\label{eq-error-1}
\tilde{E}_1 = 
- \frac{v\bv}{2(1+v)} 
- \phi\,\frac{2(n-1)v_y^2 v_{yy}}{1+2(n-1)v_y^2},
\end{equation}
and
\begin{equation}
\label{eq-error-23}
\tilde{E}_2 =
\Bigl(\phi_\tau - \phi_{yy} + \frac y2 \phi_y\Bigr) v,\qquad
\tE_3 = -2 \phi_y v_y.
\end{equation}
The definition $\phi(y, \tau) = \bphi(y/\ell(\tau))$ with $\ell(\tau) = \bd(\tau)^{1/3}$ implies that $\phi$ satisfies
\[
\phi_\tau = - \frac{\bd'(\tau)}{3\bd(\tau)} y\phi_y,
\]
which lets us rewrite $\tE_2$ as
\begin{equation}
\label{eq-error-2}
\tilde{E}_2 =
\left\{-\phi_{yy} + \left(\frac12 - \frac{\bd'}{3\bd}\right) y \phi_y\right\} v.
\end{equation}

\subsection{Estimating the error term $\tE$}
\label{sec-estimating-error}
It will be useful to have the following bounds for the derivatives of the cut-off function $\phi$.
\begin{prop}
\label{prop-phi-bounds}
The derivatives $\phi_y$ and $\phi_{yy}$ are supported in the region $\ell \le y\le 2\ell$, where they satisfy
\[
\ell |\phi_y| + |y\phi_y| + \ell^2 |\phi_{yy}| \leq C_0,
\]
where $C_0 = 3\sup_{\eta}|\bphi'(\eta)|+|\bphi''(\eta)|$.
\end{prop}
These inequalities are a direct consequence of the definition $\phi(y, \tau) = \bphi(y/\ell(\tau))$.  For a function $f$ on $\R$ we introduce the norm weighted $L^2$ norm $\|f\|$ given by
$$\| f \|^2 =\int_{\R} f(y) \, e^{-y^2/4} \, dy.$$

\begin{lemma}
\label{lem-error-v}
For every $\epsilon > 0$ there exists a $\tau_0 < 0$ so that for $\tau\le \tau_0$,
\[
\|\tilde{E}\| \le \epsilon\, \|\bv \|.
\]
\end{lemma}

\begin{proof}
Recall that by \eqref{eq-convexity} and Lemma \ref{lem-loc-est} we have
\[
|v_y| + |v_{yy}| \le \frac{C}{\bd(\tau)}.
\]
Since $\ell(\tau) = \bd(\tau)^{1/3}$ this implies
\[
0 \le v(0,\tau) - v(y,\tau) \le {C}{\ell(\tau)^{-2}}
\text{ for } |y| \le 2 \ell(\tau),
\]
and therefore
\begin{align*}
|v(y,\tau)| &\le |v(y,\tau) - v(0,\tau)| + |v(0,\tau)|\\
&< C\ell^{-2} + |v(0,\tau)|\\
& = \delta(\tau),
\end{align*}
for $|y|\le 2 \ell $, where $\lim_{\tau\to -\infty} \delta(\tau) = 0$.

If we look at \eqref{eq-error-1}, \eqref{eq-error-23} and \eqref{eq-error-2} then we see that
\[
|\tE_1| \le C\bigl(|v\bv| + v_y^2 |v_{yy}|\bigr),
\text{ and }
|\tE_2| \le C|v|,\quad
|\tE_3|\le \frac{C}{\ell}|v_y|,
\]
while we also find that $\tE_2$ and $\tE_3$ vanish outside of the region $\ell \le y \le 2\ell$.

To estimate $\|\tE\|$ we consider the four terms appearing in our pointwise bounds for $\tE_{1}$, $\tE_2$, and $\tE_{3}$ one by one.

The first term in $\tE_1$ is the easiest.  Using $|v|\le \delta(\tau)$ we get
\[
\|v\bv\| \le \delta(\tau)\| \bv\|.
\]
For the next term we use Lemma~\ref{lem-loc-est}, which guarantees $|v_y| + |v_{yy}|\leq C\bd^{-1} = C\ell^{-3}$, to get
\[
|v_y^2v_{yy}| = |v_yv_{yy}| \cdot |v_y| \le C \ell^{-6} |v_y|,
\]
so that
\[
\int_0^{2\ell } |v_y^2v_{yy}|^2 e^{-y^2/4}dy
\le C\ell^{-12}\int_0^{2\ell } |v_y|^2 e^{-y^2/4}dy.
\]
By Lemma~\ref{lem-inner-outer} the $L^2$ norm of $v_y$ on the interval $(0, 2\ell )$ is bounded by the $L^2$ norm of $v$ on the first half of that interval:
\[
\int_0^{2\ell } |v_y^2v_{yy}|^2 e^{-y^2/4}dy
\le C\ell^{-12}\int_0^{\ell } v^2 e^{-y^2/4}dy.
\]
Since $v=\bv$ for $0\le y \le \ell $ we get
\[
\|v_y^2 v_{yy} \| \le \frac{C}{\ell^{6}} \|\bv\|.
\]
Hence $\|\tE_1\|\leq \bigl(\delta(\tau) + C\ell(\tau)^{-6}\bigr)\|\bv\|$.

Turning to $\tE_2$ we recall that $\tE_2$ is supported in $\ell \le y \le 2\ell $, so that we have to bound
\[
\int_{\ell }^{2\ell } v^2 e^{-y^2/4}dy.
\]
By Corollary~\ref{cor-inner-outer} we have
\[
\int_{\ell }^{2\ell } v^2 e^{-y^2/4}dy
\le \frac{C}{\ell^{2}} \int_0^{\ell } v^2 e^{-y^2/4}dy
\]
which implies that $\|\tE_2\| \le C\ell^{-2}\|\bv\|$.

Finally, for $\tE_3$ we use Corollary~\ref{cor-inner-outer} to conclude that
\[
\|\tE_3\|^2 \leq \frac{C}{\ell^2} \int_\ell^{2\ell} v_y^2 e^{-y^2/4}dy
\leq \frac{C}{\ell^2} \int_0^{\ell} v^2 e^{-y^2/4}dy
\leq \frac{C}{\ell^2} \|\bv\|^2,
\]
i.e.~we have $\|\tE_3\|\le C\ell^{-1}\|\bv\|$.

This completes the proof of Lemma~\ref{lem-error-v}.
\end{proof}

\subsection{The linear operator $\cL$}
The operator $\cL$ is self adjoint in the Hilbert space $\hilb := L^2(\mathbb{R}, e^{-y^2/4}dy)$.  We introduce the norm and the inner product on $\hilb$ by
\[
\|f\|^2 = \int_\R f(y)^2 e^{-y^2/4}\, dy,
\qquad \langle f, g \rangle = \int_\R f(y)g(y) e^{-y^2/4}\, dy.
\]
The quadratic form associated with $\cL$ is
\begin{equation}
\label{eq:Q-form}
\langle f, \cL f\rangle
= - \int_\R \bigl\{f(y)^2 - f'(y)^2\rangle\bigr\}
e^{-y^2/4} dy,
\end{equation}
and the quadratic form domain, i.e.~the domain of $\sqrt{2-\cL}$ has norm
\begin{equation}
\label{eq:root-L-norm}
\|\sqrt{2-L}\, f\|^2 = \langle f, (2-\cL)f\rangle
= \int_\R \bigl\{f(y)^2 + f'(y)^2\bigr\} e^{-y^2/4}dy.
\end{equation}

The Hilbert space $\hilb$ has a basis of orthogonal polynomials which are eigenfunctions of $\cL $.  More precisely,
\[
\psi_{2m}(y) = y^{2m} + c_{m-1} y^{2(m-1)} + \dots c_2 y^2 + c_0
\]
with
\[
\cL [\psi_{2m}] = (1 - m)\, \psi_{2m}
\]
and
\[
c_{k-1} = -\frac{2k(2k - 1)}{m - k + 1}\, c_k.
\]
The first few eigenfunctions for the eigenvalues $\lambda_{2m} = 1 - m$ are:
\begin{align*}
\psi_0(y) &= 1, \qquad \lambda_0 = 1, \\
\psi_2(y) &= y^2 - 2, \qquad \lambda_2 = 0, \\
\psi_4(y) &= y^4 - 12 y^2 + 12, \qquad \lambda_4 = -1,
\end{align*}
It easily follows that
\[
\psi_2(y)^2 = \psi_4(y) + 8\psi_2(y) + 8\psi_0
\]
so that $\psi_2\perp \psi_j$ for $j\ne 2$ implies
\begin{equation}
\label{eq-psi-innerprod-relation}
\langle \psi_2, (\psi_2)^2\rangle = 8 \|\psi_2\|^2.
\end{equation}
To obtain an orthonormal basis for $\hilb$ one could consider the functions
\[
\hat\psi_{2m} := \frac{\psi_{2m}}{\|\psi_{2m}\|}.
\]
However, many computations turn out to be algebraically simpler if one uses the eigenfunctions $\psi_{2m}$, which are normalized by requiring their highest order term to be $y^{2m}$.
\subsection{The ODE lemma}
We can decompose the Hilbert space $\hilb$ into positive, negative and neutral eigenspaces
\[
\hilb = \hilb_+ \oplus \hilb_0 \oplus \hilb_-,
\]
where $\hilb_+$ is spanned by $\psi_0$, $\hilb_0$ is spanned by $\psi_2$, and $\hilb_-$ is spanned by the remaining eigenfunctions $\{\psi_4, \psi_6, \dots\}$.  We let $P_\pm$, $P_0$ be the corresponding orthogonal projections, and we define
\[
\bv_\pm(\cdot, \tau) = P_\pm\bigl[\bv(\cdot, \tau)\bigr],\quad
\bv_0(\cdot, \tau) = P_0\bigl[\bv(\cdot, \tau)\bigr]
\]
so that
\begin{equation}
\label{eq-proj}
\bv (y,\tau) = \bv_+(y,\tau) + \bv_0(y,\tau)  + \bv_-(y,\tau).
\end{equation}
Our arguments involve showing that one of these terms will be much larger than the other two as $\tau\to-\infty$, so it will be convenient to abbreviate
\begin{equation}
\label{eq-proj-norms}
V_0(\tau) :=  \|\bv_0(\cdot, \tau)\|,\quad
V_-(\tau) := \|\bv_-(\cdot, \tau)\|, \text{ and }
V_+(\tau) := \|\bv_+(\cdot, \tau)\|
\end{equation}
Our goal is to show the following proposition.
\begin{prop}
As $\tau\to -\infty$, we have
\[
V_-(\tau) + V_+(\tau) = o(V_0(\tau)).
\]
\end{prop}

The key ingredient in proving the Proposition is the ODE Lemma that was proved in \cite{FK} and \cite{MZ}.

\begin{lemma}[ODE Lemma (\cite{FK}, \cite{MZ})] \label{lem-ODE}
Let $X_0(\tau)$, $X_-(\tau)$ and $X_+(\tau)$ be absolutely continuous, real-valued functions that are nonnegative and satisfy
\begin{enumerate}
\item[(i)] $(X_0, X_-, X_+)(\tau) \to 0$ as $\tau\to -\infty$, and $\forall \tau\le \tau^*$, $X_0(\tau) + X_-(\tau) + X_+(\tau) \neq 0$, and
\item[(ii)] $\forall \epsilon > 0$, $\exists \tau_\epsilon\in \mathbb{R}$ such that $\forall \tau \le \tau_\epsilon$,
\begin{equation}
\label{eq-ode1}
\left.
\begin{aligned}
\dot{X}_+ &\ge c_0 X_+ - \epsilon (X_0+X_-)\\
|\dot{X}_0| &\le \epsilon (X_0 + X_- + X_+) \\
\dot{X}_- &\le -c_0 X_- + \epsilon (X_0 + X_+)
\end{aligned}
\right\}
\end{equation}
\end{enumerate}
Then either $X_0 + X_- = o(X_+)$ or $X_- + X_+ = o(X_0)$ as $\tau\to -\infty$.
\end{lemma}

In order to apply Lemma \ref{lem-ODE} to our projections $V_0(\tau)$, $V_-(\tau)$ and $V_+(\tau)$ defined as above, we need to show that $V_0(\tau)$, $V_-(\tau)$, $V_+(\tau)$ satisfy \eqref{eq-ode1}.

\begin{lemma}
\label{lem-error}
For every $\epsilon> 0$, there exists a $\tau_\epsilon$ so that for every $\tau \le \tau_\epsilon$:
\begin{align*}
\dot{V}_+ &\ge (1 - \epsilon) V_+ - \epsilon (V_0 + V_-)\\
|\dot{V}_0| &\le \epsilon (V_0 + V_- + V_+) \\
\dot{V}_- &\le -(1 - \epsilon) V_- + \epsilon (V_0 + V_+)
\end{align*}
where $V_+(\tau) = |\bv_+(\tau)|$, $V_0(\tau) = |\bv_0(\tau)|$ and $V_-(\tau) = \|\bv_-\|$.
\end{lemma}

\begin{proof}
Equation \eqref{eq-bar-v} tells us that $\bv_\tau = \cL\bv + \tE$.  Applying the projection $P_\pm$, $P_0$ we get
\[
\frac{d\bv_+}{d\tau} = \cL\bv_++P_+\tE,\quad
\frac{d\bv_-}{d\tau} = \cL\bv_-+P_-\tE,\quad
\frac{d\bv_0}{d\tau} = \cL\bv_0+P_0\tE.
\]
This allows us to compute the rate at which $V_+ = \|\bv_+\|$ changes by differentiating $\|\bv_+\|^2 = \langle\bv_+, \bv_+\rangle$:
\begin{align*}
\frac{d\|\bv_+\|}{d\tau}
&= \frac{1}{2\|\bv_+\|}\frac{d}{d\tau}\|\bv_+\|^2 \\
&= \frac{1}{\|\bv_+\|}\bigl\langle\bv_+, \frac{d}{d\tau} \bv_+\bigr\rangle \\
&= \frac{ \langle\bv_+, \cL\bv_+\rangle + \langle\tE, \bv_+\rangle }
        {\|\bv_+\|}
\end{align*}
At this point we recall that $1$ is the lowest positive eigenvalue of $\cL|\hilb_+$, so that
\[
\langle f, \cL f\rangle\ge \|f\|^2 \text{ for all }f\in\hilb_+.
\]
Moreover, $\|\tE\|\leq \epsilon\|\bv\|$ implies $\langle\tE, \bv_+\rangle \ge -\epsilon \|\bv_+\|\, \|\bv\|$, so that
\begin{equation}
\label{eq-error-lemma-proof1}
\frac{d\|\bv_+\|}{d\tau}
\ge
\frac{\|\bv_+\|^2 - \epsilon \|\bv\| \|\bv_+\|}{\|\bv_+\|}
=\|\bv_+\| - \epsilon \|\bv\|.
\end{equation}
A similar computation exploiting the fact that $\langle f, \cL f\rangle\le-\|f\|^2$ for all $f\in\hilb_-$ shows that
\begin{equation}
\label{eq-error-lemma-proof2}
\frac{d\|\bv_-\|}{d\tau} \le \|\bv_-\| + \epsilon \|\bv\|.
\end{equation}
and also
\begin{equation}
\label{eq-error-lemma-proof3}
\left| \frac{d\|\bv_0\|}{d\tau} \right| \le \epsilon \|\bv\|.
\end{equation}
Finally we note that $\|\bv\| \le \|\bv_-\| + \|\bv_0\| + \|\bv_+\|$, so that \eqref{eq-error-lemma-proof1}, \eqref{eq-error-lemma-proof2}, and \eqref{eq-error-lemma-proof3} imply the lemma.
\end{proof}

\begin{corollary}
\label{cor-xyz}
If $V_0, V_-, V_+$ are the projections defined as above, then we have either $V_0 + V_- = o(V_+)$ or $V_- + V_+ = o(V_0)$, as $\tau\to -\infty$.
\end{corollary}

\begin{proof}
The proof follows immediately combining Lemma \ref{lem-ODE} and Lemma \ref{lem-error}.
\end{proof}

\begin{lemma}
\label{lem-Vplus-growth-upperbnd}
For every $\epsilon>0$ there is a $\tau_\epsilon\in\R$ such that
\[
\dot{V}_+ \le (1+\epsilon) V_+ + \epsilon (V_0+V_-).
\]
\end{lemma}
\begin{proof}
One can use the same arguments that led to \eqref{eq-error-lemma-proof1}, provided one uses the fact that $\cL|\hilb_+$ is also bounded by $\langle f, \cL f\rangle \le \|f\|^2$, and provided one uses $\langle\tE, \bv_+\rangle \le +\epsilon \|\bv\|\,\|\bv_+\|$.
\end{proof}
The analogous argument for $V_-$ fails because $\cL|\hilb_-$ is unbounded, i.e.~no bound of the form $\forall f\in\hilb_- : \langle f, \cL f\rangle \geq -C\|f\|^2$ holds.

Next we would like to rule out the case $V_0 + V_- = o(V_+)$ and that is the focus of the following section.

\subsection{Dominance of the neutral mode $V_0(\tau)$}

Corollary \ref{cor-xyz} implies that either the $\psi_0$ component of $\bv$ dominates the others,
\[
V_0(\tau) + V_-(\tau) = o(V_+(\tau)) \qquad (\tau\to-\infty),
\]
or else the $\psi_2$ component dominates,
\[
V_+(\tau) + V_-(\tau) = o(V_0(\tau)) \qquad (\tau\to-\infty).
\]
We will show in a somewhat lengthy argument by contradiction that the first alternative, where $V_+$ dominates, cannot occur.  During most of the argument we assume that $V_+$ is in fact the largest of $V_-, V_0, V_+$, and we obtain more precise asymptotics for $\bv$ in this case.  Then we show that given any ancient solution $M_t$ of the unrescaled MCF, one can always choose the blow-up time $T$ in \eqref{eq-type-1-blow-up} so that the resulting parabolic blow-up $\bar{M}_\tau$ leads to a $\bv$ for which the $\psi_2$ is dominant.  Since any solution can be reduced to the case where $V_0$ dominates, we then study the asymptotic behavior of $\bv$ in this case.

Before worrying about which component of $\bv$ is the largest, we first establish that at least one of the components must be the largest.

\begin{lemma}
There is a $\tau_0$ such that $\bv(\tau)\ne 0$ for all $\tau < \tau_0$.
\end{lemma}
\begin{proof}
Suppose that for some $\tau_1\in\R$ one has $\bv(\tau_1)=0$.  Then $u(0, \tau_1) = \sqrt{2(n-1)}$.  Since the surface $M_{\tau_1}$ is convex and symmetric with respect to reflection $y\leftrightarrow -y$, we find that $M_{\tau_1}$ lies inside the cylinder $\Gamma$ with radius $\sqrt{2(n-1)}$.  By the strong maximum principle all later surfaces $M_\tau$ with $\tau>\tau_1$ must lie strictly inside $\Gamma$, and therefore $\bv(\tau)\ne 0$ for all $\tau > \tau_1$.

This argument shows that there cannot be more than one time $\tau_1$ at which $\bv(\tau)$ vanishes.  We therefore certainly know that $\bv(\tau)\ne0$ for all $\tau < \tau_0$, for some suitably chosen $\tau_0$.
\end{proof}

\begin{lemma}
\label{lem-growth-of-dominant-term}
If $V_+$ dominates, i.e. if $V_- + V_0 = o(V_+)$ then for any $\epsilon>0$ there are $\tau_\epsilon$ and $c_\epsilon $, $C_\epsilon$ such that
\begin{equation}
\label{eq-Vplus-exp-growth-rough}
c_\epsilon e^{(1+\epsilon)\tau} \le V_+(\tau) \le C_\epsilon e^{(1-\epsilon)\tau}
\end{equation}
for all $\tau<\tau_\epsilon$.  On the other hand, if $V_0$ dominates, i.e. if $V_-+V_+ = o(V_0)$ then
\begin{equation}
\label{eq-V0-slow-growth}
V_0(\tau) \ge C_\epsilon e^{\epsilon\tau}
\end{equation}
for all $\tau<\tau_\epsilon$.
\end{lemma}

\begin{proof}
If $V_+$ dominates, then Lemmas~\ref{lem-error} and \ref{lem-Vplus-growth-upperbnd} imply that for any $\epsilon>0$ we can find a $\tau_\epsilon$ such that
\[
(1- 2\epsilon) V_+ \le \frac{dV_+}{d\tau} \le (1+2\epsilon)V_+
\]
for all $\tau<\tau_\epsilon$.  Integrating this we find \eqref{eq-Vplus-exp-growth-rough}.

Similarly, if $V_0$ dominates instead of $V_+$, then Lemma~\ref{lem-error} implies
\[
\left|\frac{dV_0} {d\tau}\right| \le 2 \epsilon V_0,
\]
which leads to \eqref{eq-V0-slow-growth}.
\end{proof}

We now improve our estimate \eqref{eq-Vplus-exp-growth-rough} of the growth of $V_+$ assuming it is the largest term.  While we will, in the end, prove that this situation does not occur, the following lemma is essential to our proof that this is so.
\begin{lemma}
\label{lem-Vplus-wins-asymptotics}
If $V_+$ dominates, then the limit
\[
\lim_{\tau\to-\infty} e^{-\tau} V_+(\tau) = K.
\]
exists.  Moreover, $K\ne 0$, and on any fixed compact interval $|y|\le L$ one has
\begin{equation}
\label{eq-Vplus-wins-asymptotics}
\lim_{\tau\to-\infty} e^{-\tau} v(y,\tau) =\lim_{\tau\to-\infty} e^{-\tau} \bv(y,\tau) = K,
\end{equation}
uniformly.

\end{lemma}
\begin{proof}
By definition \eqref{eq-proj-norms} we have $V_+(\tau) = \bigl|\langle \psi_0, \bv(\cdot, \tau)\rangle\bigr|$.
Since we may assume that $V_+(\tau)\ne 0$ for all $\tau<\tau_0$, we have
\[
\text{either }V_+(\tau) = \langle \psi_0, \bv(\cdot, \tau) \rangle
\text{ or }V_+(\tau) = -\langle \psi_0, \bv(\cdot, \tau) \rangle
\]
for all $\tau<\tau_0$.  We will assume $V_+(\tau) = \langle \psi_0, \bv(\cdot, \tau) \rangle$ and leave the other case to the reader.
We consider the evolution of $ \langle \psi_0, \bv \rangle$.
\[
\frac{d}{d\tau} \langle \psi_0, \bv \rangle
=
\bigl\langle \psi_0, \cL\bv + \tE \bigr\rangle
=
\langle \cL\psi_0, \bv \rangle + \langle \psi_0, \tE \rangle
=
\langle \psi_0, \bv \rangle + \langle \psi_0, \tE \rangle.
\]
To prove the exponential behavior we must show that the error term $\langle \psi_0, \tE\rangle$ is small.  Using $\|\tE\|\leq \epsilon\|\bv\|$ only gives us the estimate \eqref{eq-Vplus-exp-growth-rough}, so we will have to find better bounds on $\langle\psi_0, \bv\rangle$.

We already know that $\|\bv(\tau)\| \le C_\epsilon e^{(1-\epsilon)\tau}$.  Lemma~\ref{lem-inner-outer} implies that for any fixed $M>0$ we have
\[
\int_0^M \bigl\{ v^2 + v_y^2\bigr\} e^{-y^2/4}dy
\le C e^{2(1-\epsilon)\tau}.
\]
Lemma~\ref{lem-weighted-H1-L2} then tells us that
\[
\sup_{0<y<M} |v(y, \tau)| \le C e^{(1-\epsilon)\tau},
\]
where $C$ is a generic constant which depends on $M$.  This implies
\[
u(M, \tau) \ge \sqrt{2(n-1)} \Bigl(1 - \frac{M^2}{Ce^{(1-\epsilon)\tau}}\Bigr),
\]
so that Lemma~\ref{lem-lower-bound-for-ancient-solutions} then provides us with a lower bound for $u(y, \tau)$ and more importantly, for $\bd(\tau)$.  We get
\begin{equation}
\label{eq-Vpluswins-dbar-lowerbound}
\bd(\tau) \ge c \, e^{(1-\epsilon)\tau/2}
\end{equation}
for $\tau < \tau_\epsilon$ and for some small $c>0$.

$\langle \psi_0, \tE_1 \rangle$ has two terms, the first being
\[
\left|\Bigl\langle \psi_0, \frac{\phi v^2}{2(1+v)}\Bigr\rangle\right|
\le C\int_0^M v^2 e^{-y^2/4} dy
\le C \|\bv\|^2
\le C V_+^2.
\]
The other term in $\langle \psi_0, \tE_1 \rangle$ is
\[
\left|
\Bigl\langle \psi_0, \frac{\phi v_y^2 v_{yy}}{1+2(n-1)v_y^2} \Bigr\rangle
\right|
\le \frac{C}{\bd(\tau)} \int_0^M v_y^2 e^{-y^2/4} dy
\le \frac{C}{\bd(\tau)} \|\bv\|^2
\le CV_+^2.
\]
For $\langle \psi_0, \tE_2 \rangle$ we have
\[
|\langle \psi_0, \tE_2 \rangle|
= \left|\Bigl\langle 
  \psi_0, 
  \bigl(-\phi_{yy}+(\tfrac12 - \tfrac{\bd'}{3\bd}) y\phi_y\bigr) \bv
\Bigr\rangle\right|
\le C\left\|-\phi_{yy}+(\tfrac12 - \tfrac{\bd'}{3\bd}) y\phi_y\right\|\,
\|\bv\|
\]
The function $-\phi_{yy}+(\tfrac12 - \tfrac{\bd'}{3\bd}) y\phi_y$ is uniformly bounded, and it is supported in the interval $\ell\le y\le 2\ell$, so
\[
\left\|-\phi_{yy}+(\tfrac12 - \tfrac{\bd'}{3\bd}) y\phi_y\right\|^2
\le 
C\int_\ell^{2\ell} e^{-y^2/4} dy
\le C e^{-\ell^2/4}.
\]
Thus we get 
\[
|\langle \psi_0, \tE_2 \rangle|
\le Ce^{-\ell^2/8} \|\bv\|
\le Ce^{-\ell^2/8} V_+.
\]
Finally, for $\langle \psi_0, \tE_3 \rangle$ we get
\[
\left|\langle \psi_0, \tE_3 \rangle\right|
= |\langle \psi_0, \phi_y v_y \rangle|
\le C\|\phi_y\| \, \|v_y\|_{L^2(\ell, 2\ell)}.
\]
We again use Lemma~\ref{lem-inner-outer} to estimate $\|v_y\|_{L^2(\ell, 2\ell)}\leq C \|\bv\|$, and we again note that $\phi_y$ is supported in $\ell \le y\le 2\ell$, so that $\|\phi_y\|\le Ce^{-\ell^2/8}$.  Combined, we obtain
\[
\left|\langle \psi_0, \tE_3 \rangle\right|
\le Ce^{-\ell^2/8} \|\bv\| \le Ce^{-\ell^2/8} V_+.
\]
Adding the three estimates for $\langle\psi_0, \tE_i\rangle$ ($i=1,2,3$) we find
\[
\left|\bigl\langle \psi_0, \tE \bigr\rangle\right|
\le C \bigl( V_+^2 + e^{-\ell^2/8}V_+\bigr).
\]
Recall that $\ell = \bd^{1/3}$, so that by \eqref{eq-Vpluswins-dbar-lowerbound} we get $e^{-\ell^2/8} \le C e^{-e^{(1-\epsilon)\tau/3}/8} \le C e^{\tau}$.  This, combined with our rough exponential bound $V_+\le C e^{(1-\epsilon)\tau}$ leads us to
\[
\left|\bigl\langle \psi_0, \tE \bigr\rangle\right| \le C e^{(1-\epsilon)\tau}V_+.
\]
Thus we have
\[
\left|\frac{dV_+}{d\tau} - V_+\right| \le Ce^{(1-\epsilon)\tau} V_+,
\]
and
\[
\frac{d\ln V_+}{d\tau} = 1 + \cO\bigl(e^{(1-\varepsilon)\tau}\bigr).
\]
Integration shows that $e^{-\tau} V_+(\tau)$ does indeed converge to some constant $K$, and that
\begin{equation}
V_+(\tau) = \bigl(K + \cO(e^{(1-\epsilon)\tau})\bigr) e^{\tau}. 
\label{eq-Vplus-asymptotics-with-error}
\end{equation}

We now prove \eqref{eq-Vplus-wins-asymptotics}.  Since $V_0+V_- = o(V_+)$ it follows from convergence of $e^{-\tau}\langle\psi_0, \bv\rangle$ that $e^{-\tau}\bv$ converges in $\hilb$, and therefore that $e^{-\tau}\bv|_{[0, L]}$ converges in $L^2([0,L])$.  Lemma~\ref{lem-inner-outer} tells us that $e^{-\tau}\bv|_{[0,L]}$ is bounded in $H^1([0,L])$, so interpolation between $L^2$ and $H^1$ implies that $e^{-\tau}\bv|_{[0,L]}$ converges uniformly, as claimed.

We complete the proof of Lemma~\ref{lem-Vplus-wins-asymptotics} by observing that $K$ cannot vanish, for if if it did, then \eqref{eq-Vplus-asymptotics-with-error} would imply $V_+(\tau) = \cO( e^{(2-\epsilon)\tau})$, which contradicts the lower bound in \eqref{eq-Vplus-exp-growth-rough}.
\end{proof}

\begin{lemma}
\label{lem-V0-wins}
The neutral mode $V_0$ is the largest, namely we have $V_-+V_+ = o(V_0)$ for $\tau\to-\infty$.
\end{lemma}
\begin{proof}
We go back to our original definition \eqref{eq-type-1-blow-up} of the parabolic blow-up of a given ancient solution $M_t$ to MCF and consider the effect of a change in the blow-up time $T$ on the blow-up $\bar{M}_\tau$ (and thus $u(y, \tau)$ and $\bv(y, \tau)$).

Assume that $U(x,t)$ is a solution to the unrescaled MCF \eqref{eq-u-original}.  For any choice of blow-up time $T$ define $u(y, \tau)$ according to \eqref{eq-cv1}, i.e.
\[
u(y,\tau) = \frac 1{ \sqrt{T-t}} \, U(x,t), \qquad y = \frac {x}{\sqrt{T-t}}, \,\, \tau = -\log (T-t),
\]
so that $u(y,\tau)$ satisfies \eqref{eq-u}.

If we assume that the solution is one in which $V_+$ dominates, then \eqref{eq-Vplus-wins-asymptotics} implies that
\[
u(y,\tau) = \sqrt{ 2(n-1) (1+ K\, e^{\tau})} + o(e^\tau), \qquad \text{as} \,\, \tau \to -\infty
\]
for some $K >0$, uniformly on bounded intervals $|y|\le L$.

In terms of the original solution this is equivalent with
\begin{align*}
U(x,t)
&= \sqrt{T-t} \, \left \{ \sqrt{ 2(n-1) (1+ K\, (T-t)^{-1})} + o((T-t)^{-1}) \right \}\\
&= \sqrt{ 2(n-1) \, ( T+K-t) } + o\bigl((T-t)^{-1/2} \bigr) \qquad (t\to-\infty)
\end{align*}
uniformly for $|x|\le L\sqrt{T-t}$.

If we had chosen $T+K$ instead of $T$ as our blow-up time, then the rescaled profile would have been
\[
\hat u(y,\tau) = \frac 1{\sqrt{T+K-t}} \, U(x,t), \qquad y = \frac {x}{\sqrt{T+K-t}}, \,\, \tau = -\log (T+K-t),
\]
where $\hat u$ still satisfies \eqref{eq-u}.  The asymptotic behavior of $\hat{u}$ as $\tau\to-\infty$ is given by
\begin{equation}
\label{eq-proof-v0wins-1}
\hat u (y,\tau) = \sqrt{2(n-1)} + o(e^{\tau}), \qquad (\tau \to -\infty),
\end{equation}
uniformly on bounded intervals $|y|\le L$.  If we define $\hat v$ by $\hat u = \sqrt{2(n-1)} (1+\hat v)$, then we also have the three components $\hat V_0$, $\hat V_+$, and $\hat V_-$ defined as in \eqref{eq-proj-norms}.  By construction we have $\hat V_+ = o(e^\tau)$.  It then follows that $\hat V_+$ cannot be the largest, because then \eqref{eq-Vplus-wins-asymptotics} would hold with $K = 0$, in contradiction with Lemma~\ref{lem-Vplus-wins-asymptotics}.  Thus we have found that $\hat V_0$ is the largest, i.e.~$\hat V_+ + \hat V_- = o(\hat V_0)$ as $\tau\to-\infty$.

Lemma~\ref{lem-growth-of-dominant-term} and in particular \eqref{eq-V0-slow-growth} applied to $\hat u$ give us that $\hat V_0 \ge C_\epsilon e^{\epsilon\tau}$ for any small $\epsilon>0$.  Using $\hat V_+ + \hat V_-=o(\hat V_0)$ we conclude that
\[
\lim_{\tau\to-\infty} \frac{\hat v}{V_0} = -\psi_2(y) = - \frac{y^2-2}{\|y^2-2\|}
\]
in the $\hilb$ norm.  Hence
\[
\lim_{\tau\to-\infty} \Bigl\langle \chi_{[-1,1]}, \frac{\hat v}{V_0}\Bigr\rangle =\langle \chi_{[-1,1]}, -\psi_2\rangle \ne 0.
\]
On the other hand we have shown that $\hat v = o(e^{\tau})$ uniformly on bounded intervals, while $V_0^{-1} \le C e^{-\epsilon\tau}$.  This would imply
\[
\left\| \frac{\hat v}{V_0} \right\| = o\bigl(e^{(1-\epsilon)\tau}\bigr),
\]
which then leads to $\lim_{\tau\to-\infty} \langle \frac{\hat v}{V_0}, \psi_2\rangle = 0$. This final contradiction completes the proof.

\end{proof}

\subsection{Asymptotics of the dominating term $\bv_0$}
We have shown that the $\bv_0$ term in the expansion of $\bv$ is dominant for $\tau\to-\infty$.  If we write
\[
\bv_0(y,\tau) = \alpha(\tau)\, \psi_2(y)
\]
where $\psi_2(y) = y^2 - 2$ and
\begin{equation}
\label{eq-alpha-100}
\alpha(\tau) = \frac{ \langle \bv , \psi_2\rangle}{\|\psi_2\|^2}
\end{equation}
then
\begin{equation}
\label{eq-useful}
\bv (y,\tau) = \alpha(\tau) \psi_2(y) + o(\alpha(\tau)).
\end{equation}
Here $o(\alpha)$ is an $\hilb$-valued function of $\tau$ whose norm satisfies
\[
\lim_{\tau\to-\infty} \frac{\|o(\alpha(\tau))\|} {\alpha(\tau)} = 0.
\]
Note that
\[
V_0(\tau) = |\alpha(\tau)|\cdot \|\psi_2\|.
\]
Our main goal in this section is to prove that $\alpha$ asymptotically satisfies a simple differential equation, from which its asymptotic growth at $\tau\to-\infty$ follows directly, Namely:
\begin{lemma}
\label{lem-alpha-asymptotic-diffeq}
For $\tau\to-\infty$ one has
\[
\frac{d\alpha} {d\tau} = 4\alpha^2(\tau) + o\bigl(\alpha(\tau)^2\bigr)
\]
and
\[
\alpha(\tau) = -\frac{1+o(1)}{4\tau}.
\]
\end{lemma}
The proof, which occupies the rest of this section, begins with differentiating \eqref{eq-alpha-100} with respect to time.  Thus to find $\alpha'(\tau)$ we must compute $\frac{d} {d\tau}\langle \psi_2, \bv \rangle$.  Using the evolution equation \eqref{eq-bar-v} for $\bv$ we find
\begin{equation}
\label{eq-V0wins-alpha-diffeq}
\frac{d\alpha} {d\tau} = 
\|\psi_2\|^{-2}\frac{d} {d\tau} \langle \psi_2, \tE \rangle.
\end{equation}
where $\tE=\tE_1+\tE_2+\tE_3$ is as in \eqref{eq-error-1}, \eqref{eq-error-23}.  We will show that of all the terms that contribute to $\langle\psi_2, \tE\rangle$ the first term from $\tE_1$ is the largest, while the other all are of order $o(\alpha^2)$.  We now begin with the estimates we need to prove this.
\begin{lemma}
\label{lem-V0wins-v-estimates}
For $\tau\le\tau_0$ we have
\begin{subequations}
\begin{gather}
\label{eq-V0wins-v-estimates-a}
\|\bv\| + \|y\bv\| + \|\bv_y\| \leq C|\alpha|,\\
\label{eq-V0wins-v-estimates-b}
\int_{0}^{2\ell} \bigl\{ v^2 + y^2 v^2 + v_y^2\bigr\} e^{-y^2/4} dy \le C\alpha^2.
\end{gather}
\end{subequations}
\end{lemma}
\begin{proof}
Since $V_0$ is dominant we have $\|\bv\| \leq (1+o(1)) V_0(\tau) \le C|\alpha|$.

This implies
\[
\int_0^\ell v(y, \tau)^2 e^{-y^2/4} dy \le C\alpha^2
\]
and by the inner-outer Lemma~\ref{lem-inner-outer} and Corollary~\ref{cor-inner-outer} we get
\[
\int_0^{2\ell} \bigl\{ v^2+v_y^2\bigr\} e^{-y^2/4} dy \le C\alpha^2.
\]
In view of Lemma~\ref{lem-weighted-H1-L2} we also get
\[
\int_0^{2\ell} y^2v^2 e^{-y^2/4} dy\le C\alpha^2.
\]
Together these estimates imply \eqref{eq-V0wins-v-estimates-b}.  The bounds
\eqref{eq-V0wins-v-estimates-a} on $\bv$ then follow from the definition $\bv=\phi v$,
combined with the boundedness of the derivative $\phi_y$.
\end{proof}

So far we know that $\|\bv - \alpha(\tau)\psi_2\|=o(\alpha)$, but the same is true in a stronger norm.
\begin{lemma}
\label{lem-root-L-norm-of-small-modes}
\[
\left\|
\sqrt{2-\cL} \; \bigl(\bv - \alpha(\tau)\psi_2\bigr)
\right\|
=o(\alpha(\tau))\qquad (\tau\to-\infty).
\]
\end{lemma}
\begin{proof}
We know that $\bv$ satisfies the linear inhomogeneous equation \eqref{eq-bar-v}, i.e.~$\bv_\tau = \cL\bv + \tE$, and we also know that $\|\tE\| \le \epsilon\|\bv\|$ for all $\epsilon>0$ and $\tau\le \tau_\epsilon$.  If let $P$ be the projection
\[
Pf = f - \frac{\langle\psi_2, f\rangle}{\|\psi_2\|^2} \psi_2,
\]
and abbreviate
\[
\bw = P\bv = \bv - \alpha\psi_2,
\]
then $\bw(\tau)$ satisfies
\[
\bw_\tau = \cL \bw + P\tE.
\]
At any given $\tau$ the variation of constants formula says
\[
\bw(\tau) = e^{\cL}\bw(\tau-1) + 
\int_{\tau-1}^\tau e^{(\tau-\tau')\cL} P\tE(\tau') \,d\tau'.
\]
Apply $\sqrt{2-\cL}$ to both sides and, using $\|\sqrt{2-\cL}\; e^{\theta\cL}\|\le C \theta^{-1/2}$ for $0<\theta\le1$, we compute the $\hilb$ norm to get
\begin{align*}
\|\sqrt{2-\cL} \; \bw(\tau)\|
&\le C\|\bw(\tau-1)\|
+ \int_{\tau-1}^\tau  \frac{C}{\sqrt{\tau-\tau'}} \|\tE(\tau')\| d\tau'.\\
&\le C\|\bw(\tau-1)\|
+ C \sup_{[\tau-1,\tau]} \|\tE(\tau')\| \\
&\le C\|\bw(\tau-1)\|
+ C\epsilon \sup_{[\tau-1,\tau]} |\alpha(\tau')| 
\end{align*}
for all $\tau\le \tau_\epsilon$.  Recall that $\bw=\bv - \alpha\psi_2 = o(\alpha)$, so $\|\bw(\tau-1)\| \le C\epsilon |\alpha(\tau-1)|$ for $\tau\le \tau_\epsilon$, and hence we have
\[
\|\sqrt{2-\cL} \; \bw(\tau)\| \le 
C\epsilon \sup_{[\tau-1,\tau]} |\alpha(\tau')| \text{ for }\tau\le\tau_\epsilon.
\]
Finally we observe that since $V_0(\tau)$ dominates the other two norms, it follows from Lemma~\ref{lem-error} that $|\dot V_0| \le \epsilon V_0$ for $\tau\le\tau_\epsilon$, and thus $V_0(\tau') \le e^{\epsilon}V_0(\tau) $ for $\tau'\in[\tau-1 ,\tau]$.  Since $\|V_0(\tau)\| = \|\psi_2\|\cdot|\alpha(\tau)|$, we get
\[
\sup_{[\tau-1, \tau]}|\alpha(\tau')| \leq C \alpha(\tau),
\]
and therefore also
\[
\|\sqrt{2-\cL} \; \bw(\tau)\| \le 
C\epsilon |\alpha(\tau)| \text{ for }\tau\le\tau_\epsilon,
\]
as claimed.
\end{proof}

\begin{corollary}
\label{cor-v-over-alpha-converges}
On any finite interval $|y|\le L$ we have
\[
\lim_{\tau\to-\infty} \frac{v(y, \tau)} {\alpha(\tau)} = y^2-2
\]
uniformly.
\end{corollary}
Since $u(y, \tau) = \sqrt{2(n-1)}\bigl(1+v(y, \tau)\bigr)$ is a concave function this implies that $\alpha(\tau)<0$ for all $\tau<\tau_0$ for some $\tau_0$.
\begin{proof}
For any function $f$ the norm $\|\sqrt{2-\cL}\; f\|$ bounds the $H^1$ norm on any compact interval $|y|\le L$, and therefore one has
\[
\sup_{|y|\le L}|f(y)| \le C_L \|\sqrt{2-\cL}\; f\|.
\]
This, together with the previous Lemma~\ref{lem-root-L-norm-of-small-modes}, implies uniform convergence of $\bv/\alpha$ to $\psi_2$.  For any $L$ there is a $\tau_L$ such that $\bv$ and $v$ coincide on the interval $|y|\le L$ if $\tau\le\tau_L$, so $v/\alpha$ also converges to $\psi_2$ on $[-L, L]$.
\end{proof}
\begin{lemma}
\label{lem-V0wins-dbar-lower-bound}
There is a constant $c>0$ such that
\[
\bd(\tau) \ge c |\alpha(\tau)|^{-1/2} \quad \text{ and }
\quad \ell(\tau) \ge c |\alpha(\tau)|^{-1/6}.
\]
\end{lemma}
\begin{proof}
In our setting we know that $\bv(\tau) = \alpha\psi_2 + o(\alpha)$, and we can repeat the argument that led to (\ref{eq-Vpluswins-dbar-lowerbound}), with $\bd\ge c/\sqrt{\alpha}$ as immediate conclusion.  The second lower bound follows from the definition $\ell = \bd^{1/3}$.
\end{proof}

We can now begin with estimating how the various terms in $\tE$ contribute to $\alpha'(\tau)$.  We begin with $\tE_1$.
\begin{lemma}
\label{lem-E1-bound}
\begin{equation}
\label{eq-E1-bound}
\frac{\langle\psi_2, \tE_1\rangle}{\|\psi_2\|^2} = 4 \alpha^2 + o(\alpha^2).
\end{equation}
\end{lemma}
\begin{proof}
We can write $\tE_1 = \tE_{1a} + \tE_{1b}$, where
\[
\tE_{1a} = \left\langle\psi_2, \frac{v\bv}{2(1+v)}\right\rangle \quad 
\text{ and }
\quad \tE_{1b}
= \left\langle\psi_2, \frac{2(n-1)\phi v_y^2v_{yy}}{1+2(n-1)v_y^2}\right\rangle.
\]
To estimate $\tilde E_{1b}$ we recall that for $|y|\le 2\ell$ one has $|v_{yy}| \le C\bd^{-1} = C\ell^{-3}$ by Lemma~\ref{lem-loc-est}.  Also, since $\psi_2(y) = y^2-2$ we have $|\psi_2| \le C\ell^2$ when $|y|\le2\ell$.  Thus
\begin{align*}
|\tE_{1b}|
&\le 2(n-1) \int_{|y|\le2\ell} |\psi_2v_{yy}|\, v_y^2 e^{-y^2/4}dy \\
&\le C\ell^{-1}\int_{|y|\le2\ell} v_y^2 e^{-y^2/4} dy\\
&\le C\ell^{-1} \alpha^2.
\end{align*}
By Lemma~\ref{lem-V0wins-dbar-lower-bound} we get $\ell^{-1} \le C|\alpha|^{1/6}$, so
\[
|\tE_{1b}| \le C |\alpha|^{2+1/6} = o(\alpha^2).
\]
To estimate the other term, which has $\tE_{1a}$, we split $\tE_{1a}$ into three parts:
\begin{equation}
\label{eq-E-1a-split}
\frac{v\bv}{2(1+v)} 
= \frac12 v\bv - \frac{v^2\bv}{2(1+v)}
=\frac12 \bv^2 +\frac12 (v-\bv)\bv - \frac{v^2\bv}{2(1+v)}.
\end{equation}
Since $\bv$ is supported on $|y|\le 2\ell$ and since $|\bv|\le v$ there, the contribution of the third term can be bounded by
\[
\left|\left\langle\psi_2, \frac{v^2\bv}{2(1+v)} \right\rangle\right|
\le \sup_{[0, 2\ell]}|v|\cdot\int_0^{2\ell} |\psi_2(y)| v^2 e^{-y^2/4}dy.
\]
We estimate $v$ on the interval $[0, 2\ell]$ by noting that $v(0, \tau) =\cO(|\alpha|)$ by Corollary~\ref{cor-v-over-alpha-converges}, and $|v_y|\le C\bd^{-1} = C\ell^{-3}$ by \eqref{eq-first-der1}.
\[
\sup_{[0, 2\ell]}|v| \le C|\alpha| + C\ell^{-3}\cdot2\ell 
= C\bigl(|\alpha| + \ell^{-2}\bigr)
\le C|\alpha|^{1/3},
\]
by Lemma~\ref{lem-V0wins-dbar-lower-bound}.  Thus we find
\[
\left|\left\langle\psi_2, \frac{v^2\bv}{2(1+v)} \right\rangle\right|
\le C|\alpha|^{1/3} \int_0^{2\ell} \bigl\{ v^2 + y^2v^2\bigr\}\, dy
\le C|\alpha|^{2+1/3}
=o(\alpha^2),
\]
where we have used Lemma \ref{lem-V0wins-v-estimates}.

We go on with the middle term in \eqref{eq-E-1a-split}.  Since $\bv = \phi v$ we have $(v-\bv)\bv = (1-\phi)\phi v^2$, which is supported in the interval $[\ell, 2\ell]$.  Thus we have
\[
\left|\bigl\langle\psi_2, \tfrac12(v-\bv)\bv\bigr\rangle\right|
=\tfrac12\left|\bigl\langle\psi_2, (1-\phi)\phi v^2\bigr\rangle\right|.
\]
Because of the Gaussian weight in the inner product, this term is very small.  We crudely bound $|\psi_2| \le Cy^2$, $|v|\le C$, and find
\begin{align*}
\tfrac12\left|\bigl\langle\psi_2, (1-\phi)\phi v^2\bigr\rangle\right|
&\le C \int_\ell^{2\ell} y^2 e^{-y^2/4}dy\\
&\le C\ell e^{-\ell^2/4}\\
&\le C|\alpha|^{-1/6}e^{-|\alpha|^{-1/3}/4}\\
& = o(\alpha^2),
\end{align*}
where we have again used Lemma~\ref{lem-V0wins-dbar-lower-bound}.

We are left with the first term in \eqref{eq-E-1a-split}.  We substitute $\bv = \alpha\psi_2 + \bw$ and expand to get
\[
\bigl\langle\psi_2, \bv^2\bigr\rangle
= \alpha^2 \langle\psi_2, \psi_2^2\rangle + 2\alpha\langle\psi_2, \bv\bw\rangle
+ \langle\psi_2, \bw^2\rangle.
\]
We know that $\|\bv\| + \|\bv_y\| + \|y\bv\| = \cO(|\alpha|)$, and Lemma~\ref{lem-root-L-norm-of-small-modes} says that we have $\sqrt{2-\cL}\cdot\bw = o(|\alpha|)$, so we also have $\|\bw\| + \|\bw_y\| + \|y\bw\| = o(|\alpha|)$.  Keeping in mind that $\psi_2(y) = y^2-2$, so that $|\psi_2(y)|\le C(1+|y|)^2$, we get
\[
\left|\langle\psi_2, \bw^2\rangle\right| \le C\|(1+|y|)\bw\|^2 = o(\alpha^2),
\]
and
\[
\left| \langle \psi_2, \bv\bw \rangle \right|
\le \|(1+|y|)\bv\|\; \|(1+|y|)\bw\|
=o(\alpha^2).
\]
Finally, by \eqref{eq-psi-innerprod-relation} we have $\langle \psi_2, \psi_2^2 \rangle = 8\|\psi_2\|^2$, so that $\langle \psi_2, \bv^2 \rangle= 8\|\psi_2\|^2 \alpha^2 + o(\alpha^2)$.  Adding this and the estimates of the other terms in \eqref{eq-E-1a-split} leads to the asymptotic relation in \eqref{eq-E1-bound}.
\end{proof}

In Lemma~\ref{lem-error-v} we estimated the error terms $\tE_2$ and $\tE_3$.  At this point we have better estimates for $\ell$ which allow us to improve the old estimates.
\begin{lemma}
\label{lem-E2-bound}
\[
\|\tE_2\| + \|\tE_3\| = o(\alpha^2).
\]
\end{lemma}
\begin{proof}
In the proof of Lemma~\ref{lem-error-v} we found that
\[
\|\tE_2\|^2\le \int_\ell^{2\ell} v^2 e^{-y^{2}/4} dy.
\]
Using the very rough bound $|v|\le C$ together with $\ell\ge c|\alpha|^{-1/6}$ we get
\[
\|\tE_2\|^2\le C\int_\ell^{2\ell} e^{-y^{2}/4} dy
\le C e^{-\ell^2/4} \le e^{-c|\alpha|^{-1/3}} = o(\alpha^2).
\]
For $\tE_3$ we have a similar argument.  From \eqref{eq-error-23} and the fact that both $\phi_y$ and $v_y$ are uniformly bounded we get $|\tE_3|\le C$, while $\tE_3$ also is supported in $[\ell, 2\ell]$.  The same computation as above then shows that $\|\tE_3\| = o(\alpha^2)$.
\end{proof}
\subsection*{Completion of the proof of Lemma~\ref{lem-alpha-asymptotic-diffeq}}
We began the proof of Lemma~\ref{lem-alpha-asymptotic-diffeq} by writing the derivative $\alpha'(\tau)$ as in \eqref{eq-V0wins-alpha-diffeq}.  We can now use Lemmas~\ref{lem-E2-bound} and \ref{lem-E1-bound} to expand $\tE$ in \eqref{eq-V0wins-alpha-diffeq}, which quickly leads to the claimed result, i.e.~the \textsc{ODE} ~$\alpha'= 4\alpha^2 + o(\alpha^2)$.

Integration of this differential equation directly gives $\alpha=-(1+o(1))/(4\tau)$, as claimed in Lemma~\ref{lem-alpha-asymptotic-diffeq}.

A direct consequence of Lemma~\ref{lem-alpha-asymptotic-diffeq} is the following lower bound for the extrinsic diameter $\bd(\tau)$:
\begin{equation}
\label{eq-dbar-precise-lower-bound}
\bd(\tau) \ge c\, \sqrt{-\tau}.
\end{equation}

\section{Intermediate region}
\label{sec-inter}

From section \ref{sec-parabolic}, for every finite $M > 0$ we have
\begin{equation}
\label{eq-asymp-100}
u(y,\tau) = \sqrt{2(n-1)} \Bigl\{ 1 + \frac{y^2 - 2}{4\tau}\Bigr\}
+ o(|\tau|^{-1}), \qquad |y| \le M,
\end{equation}
as $\tau\to -\infty$.

We introduce the coordinate $z = \frac{y}{\sqrt{|\tau|}}$ and consider $\bar{u}(z,\tau) = u(y,\tau)$.  It easily follows that
\[
\frac{\pd\bar{u}}{\pd\tau} = \frac{\bar{u}_{zz}}{|\tau| + \bar{u}_z^2} - \frac z2 \left(1 - \frac{1}{\tau}\right)\,\bar{u}_z + \frac{\bar{u}}{2} - \frac{n-1}{\bar{u}}.
\]

\begin{lemma}
\label{lem-inter}
With the notation as above we claim
\[
\lim_{\tau\to -\infty} \bar{u}(z,\tau) =\sqrt{n-1}\sqrt{2 - z^2},
\]
and the convergence is uniform in $z$, away from $z = \sqrt{2}$.
\end{lemma}

\begin{proof}
We prove the proposition by constructing the appropriate upper and lower barriers around our solution which will force it to converge to the right limit as $\tau\to -\infty$.

\textit{Construction of lower barriers. } The construction in \S\ref{sec-lower} together with the precise asymptotics in the parabolic region \eqref{eq-asymp-100} yield the desired lower barriers.  To be more precise, take $L > 0$ big enough.  By \eqref{eq-asymp-100} we have
\begin{equation}
\label{eq-one-bound}
\frac{u(L,\tau)}{\sqrt{2(n-1)}}
= 1  + \frac{L^2-2}{4\tau} + o(|\tau|^{-1})
\ge 1  - \frac{L^2-2}{4|\tau|\, (1 - \delta(\tau))},
\end{equation}
for some nonnegative function $\delta(\tau)$ with $\lim_{\tau\to -\infty} \delta(\tau) = 0$.  We may assume that $|\tau|(1-\delta(\tau))$ is monotone in $\tau$.

Let $u_b(y)$ be one of the stationary solutions constructed in Lemma \ref{lem-u-A}.  By the expansion \eqref{eq-uA-inner-expansion} we have
\begin{equation}
\label{eq-uAL}
\frac{u_b(L)}{\sqrt{2(n-1)}} = 1 - \frac{L^2-2}{2 b^2} + o(b^{-2})
\leq 1 - \frac{L^2-2}{(2+\epsilon(b))b^2},
\end{equation}
for some nonnegative function $\epsilon(b)$ with $\lim_{b\to\infty} \epsilon(b) = 0$.

Let $\tau < \tau_0$ be arbitrary.  Choose $b(\tau)$ so that
\[
b^2 (2 + \epsilon(b)) = 4|\tau|(1 - \delta(\tau)).
\]
Such a choice can always be made, and one has
\begin{equation}
\label{eq-b-asymptotics}
b(\tau) = (1-\delta_1(\tau))\,\sqrt{2|\tau|}
\end{equation}
for some function $\delta_1(\tau)$ with $\lim_{\tau\to-\infty} \delta_1(\tau) = 0$.

By our choice of $\delta_1$ and $b(\tau_1)$ and by \eqref{eq-uAL} we have
\begin{equation}
\label{eq-other-bound}
\frac{u_{b(\tau)}(L)}{\sqrt{2(n-1)}}
\le - \frac{L^2-2}{4|\tau|\, (1 - \delta_1(\tau'))}
\text{ for all } \tau' \le \tau.
\end{equation}
By \eqref{eq-one-bound} and \eqref{eq-other-bound} we have
\[
u(L,\tau') \ge u_{b(\tau)}(L), \text{ for all } \tau'\le \tau.
\]
Therefore by Lemma~\ref{lem-lower-bound-for-ancient-solutions} we have,
\[
u(y,\tau) \ge u_{b(\tau)}(y), \text{ for all } y \ge L.
\]
Combining this with \eqref{eq-uA-outer-expansion} we get
\[
u(y,\tau) \ge \sqrt{2(n-1) \Bigl(1 - \frac{y^2-2} {b(\tau)^2} \Bigr) - o(1)}
\]
which, in view of \eqref{eq-b-asymptotics}, implies
\[
u(y,\tau)
\ge \sqrt{2(n-1) \Bigl(1 - \frac{y^2}{2|\tau|}\Bigr)
          +\delta_2(\tau)\frac{y^2}{|\tau|} - o(1)}.
\]
This estimate is meaningless unless $y^=\cO(|\tau|)$, so we may absorb the term $\delta_2(\tau)y^2/|\tau|$ in the $o(1)$ term.  In terms of the $z$ variable we then get
\begin{equation}
\label{eq-intermediate-lowerbound}
\bu(z, \tau) \ge \sqrt{n-1} \sqrt{2-z^2 + o(1)}.
\end{equation}
\textit{Construction of upper barriers. } The solution $\bu (\cdot,\tau)$ is concave and therefore $\bu _{zz} \le 0$, yielding
\[
\frac{\pd}{\pd\tau} \bu \le -\frac z2 \left(1 - \frac{1}{\tau}\right)\, \bu _z + \frac{\bu }{2} - \frac{n-1}{\bu }.
\]
Define $\bv := \bu ^2 - 2(n-1)$.  Then,
\[
\frac{\pd}{\pd\tau}\bv \le -\frac z2\, \left(1 - \frac{1}{\tau}\right)\, \bv_z + \bv .
\]
We see $\bv(z,\tau)$ is a subsolution to the first order partial differential equation
\[
\frac{\pd}{\pd\tau} w = -\frac z2 \left(1 - \frac{1}{\tau}\right)\, w_z + w,
\]
which we can write as
\begin{equation}
\label{eq-method-char}
\frac{d}{d\tau} w(z(\tau),\tau) = w(z(\tau),\tau),
\end{equation}
where
\begin{equation}
\label{eq-characteristics}
\frac{d}{d\tau} z =  \frac z2\, \left(1 - \frac{1}{\tau}\right)
\end{equation}
is the characteristic equation for \eqref{eq-method-char}.  See Figure~\ref{fig-characteristics}.

\begin{figure}\centering
\includegraphics[scale=0.8]{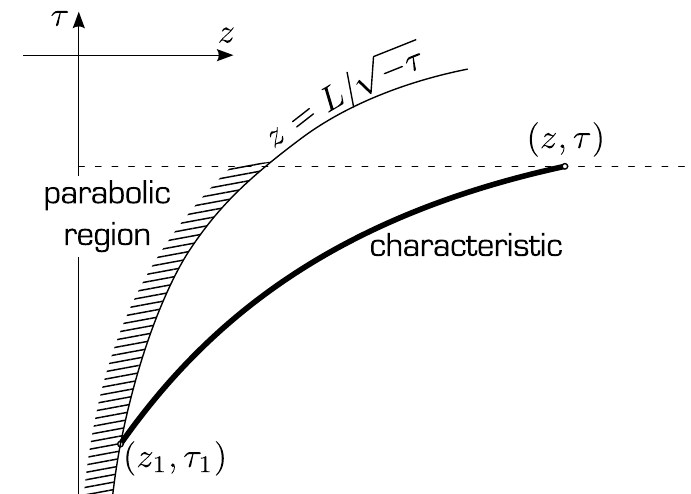}
\caption{To estimate $\bv$ at $(z, \tau)$ we follow the characteristic through $(z,\tau)$ back to the boundary of the parabolic region, where $y=L$, $z=L/\sqrt{-\tau}$.}
\label{fig-characteristics}
\end{figure}

Assume the curve $(z(\tau),\tau)$ connects $(z,\tau)$ and $(z_1,\tau_1)$ with $z_1 = L/\sqrt{-\tau_1}$, for $L > 0$ big.  Integrate \eqref{eq-characteristics} from $\tau$ to $\tau_1$ to get
\[
\tau - \tau_1 = \log\left(\frac{z^2|\tau|}{L^2}\right).
\]
At the point $z_1 = \frac{L}{\sqrt{|\tau_1|}}$ we can use \eqref{eq-asymp-100} to compute $\bv$:
\begin{align*}
\bv (z_1,\tau_1) & = \bu(z_1, \tau_1)^2 - 2(n-1)\\
&= 2(n-1)\Bigl(1 + \frac{L^2-2}{2\tau_1} + o(|\tau_1|^{-1})\Bigr) - 2(n-1)\\
&= -(n-1) \frac{L^2-2}{|\tau_1|}(1+\epsilon(\tau_1)),
\end{align*}
where $\epsilon(\tau)$ is yet another function with $\lim_{\tau\to-\infty} \epsilon(\tau)=0$.

On the other hand, if we integrate \eqref{eq-method-char} from $\tau$ to $\tau_1$ we get
\[
w(z,\tau) = e^{\tau - \tau_1}\, w(z_1,\tau_1),
\]
and we can start $w$ with the initial condition $w(z_1,\tau_1) = \bv(z_1, \tau_1)$, so that
\[
w(z,\tau) = -(n-1)\frac{z^2 |\tau|}{|\tau_1|}\,\frac{L^2-2}{L^2} \, (1 + \epsilon(\tau_1)),
\]
with $\tau_1 = \tau + \log \frac{L^2}{z^2 |\tau|}$.  Therefore,
\[
w(z,\tau)
= -\frac{(n-1)z^2|\tau|}{\left|\tau + \log\frac{L^2}{z^2|\tau|}\right|}
\frac{L^2-2}{L^2}\,
(1 + \epsilon_1(\tau)),
\]
with $\lim_{\tau\to -\infty} \epsilon_1(\tau) = 0$.
Since our point $(z, \tau)$ lies in the region $y\ge L$, we have $z\ge L/\sqrt{|\tau|}$.  We also have $z=\cO(1)$, so that we can bound the logarithm in the denominator by
\[
\left| \log\frac{L^2}{z^2|\tau|} \right| \le C\log|\tau|.
\]
Thus we get
\[
w(z, \tau)
= -(n-1)z^2 \frac{L^2-2}{L^2}
        \frac{1+\epsilon_1(\tau)}
             {1+ \cO\bigl(\frac{\log|\tau|}{|\tau|}\bigr)}
= -(n-1)z^2 \frac{L^2-2}{L^2} \bigl(1+\epsilon_2(\tau)\bigr)
\]
Since $\bv (z_1,\tau_1) = w(z_1,\tau_1)$, by the maximum principle applied to \eqref{eq-method-char}, along characteristics $(z(\tau),\tau)$ connecting $(z_1,\tau_1)$ and $(z,\tau)$ we have
\[
\bv (z,\tau) \le w(z,\tau).
\]
This implies that for all $z\ge L/\sqrt{|\tau|}$ one has
\[
\bu (z,\tau) \le \sqrt{n-1}\sqrt{2 - \frac{L^2-2}{L^2}z^2} + \epsilon_3(\tau),
\]
where again $\lim_{\tau\to -\infty} \epsilon_3(\tau) = 0$.  Hence for all $z\in (0, \sqrt{2})$
\[
\limsup_{\tau\to-\infty} \bu(z, \tau) \le \sqrt{n-1}\sqrt{2 - \frac{L^2-2}{L^2}z^2}.
\]
Since this holds for all $L>0$, we may conclude that
\begin{equation}\label{eq-intermediate-upperbound}
\limsup_{\tau\to-\infty} \bu(z, \tau) \le \sqrt{n-1}\sqrt{2 - z^2}.
\end{equation}
Finally, \eqref{eq-intermediate-lowerbound} and \eqref{eq-intermediate-upperbound} together imply Lemma~\ref{lem-inter}.
\end{proof}

\begin{corollary}
\label{cor-asymp-diam}
$\bd(\tau)= \sqrt{2|\tau|} (1 + o(1))$ for $\tau\to-\infty$.
\end{corollary}
\begin{proof}
The proof of the statement immediately follows from \eqref{eq-intermediate-lowerbound} and \eqref{eq-intermediate-upperbound} if we recall that $z = y/\sqrt{|\tau|}$.
\end{proof}

\section{Tip region}
\label{sec-tip}

In this section we give a more precise description of the surface in the tip region.  First we complete our proof of Theorem~\ref{thm-comparable-geometry}, which gives us a good estimate for the size of the curvature at the tip.  With that estimate in hand we then discuss the type-II blow-up at the tip.

\subsection{Proof of Theorem~\ref{thm-comparable-geometry}}
In Corollary~\ref{cor-Hmax-less-than-diam} we already showed that
\[
\bhm(\tau) \le \bd(\tau)
\]
at all times $\tau$.  We will now show that there exists a uniform constant $C$ so that
\[
\bd(\tau) \le C \bhm(\tau)
\]
for all $\tau \leq \tau_0$ and for some $\tau_0<0$.

Recall that
\[
\bd'(\tau) - \frac{\bd(\tau)}{2} = -\bhm(\tau).
\]
Integrating this from $\tau$ to $\tau_0$, using $\bhm(\tau) \le \bd(\tau) \le \bar{c}
\sqrt{|\tau|}$ and $\bd(\tau) \geq c \, \sqrt{|\tau|}$, and also choosing $A$ sufficiently big in the last step we find
\begin{align*}
\bd(\tau)
&= e^{\tau/2}\, \left(C + \int_{\tau}^{\tau_0} \bhm(\sigma) e^{-\sigma/2}\, d\sigma\right) \\
&= e^{\tau/2}\, \left(C + \int_{\tau + A}^{\tau_0} \bhm(\sigma) e^{-\sigma/2}\, d\sigma
+ \int_{\tau}^{\tau + A} \bhm(\sigma) e^{-\sigma/2}\, d\sigma\right)\\
&\le C \,e^{\tau/2}\, \left(1 + \sqrt{|\tau|}\, e^{-\tau/2} e^{-A/2}
+ \int_{\tau}^{\tau+A} \bhm(\sigma) e^{-\sigma/2}\, d\sigma\right) \\
&\le C \,e^{\tau/2} + C\sqrt{|\tau|} \, e^{-A/2}
+ Ce^{\tau/2}\, \int_{\tau}^{\tau+A} \bhm(\sigma) e^{-\sigma/2}\, d\sigma \\
&\le \frac{\bd(\tau)}{2} + C e^{\tau/2}\, \int_{\tau}^{\tau+A} \bhm(\sigma) e^{-\sigma/2}\, d\sigma,
\end{align*}
and thus
\[
\bd(\tau) \le 2 C \, e^{\tau/2}\, \int_{\tau}^{\tau+A} \bhm(\sigma) e^{-\sigma/2}\, d\sigma.
\]
One consequence of the Harnack inequality for the ancient mean curvature flow (\cite{Ha}) is that $H_t \ge 0$ (for the unrescaled flow).  Since $H = e^{\tau/2}\bhm$, this implies that
\[
\bhm(\sigma) \, e^{\sigma/2} \le \bhm(\tau+A) \, e^{(\tau+A)/2} \text{ for } \sigma \in [\tau, \tau+A].
\]
Hence,
\begin{align*}
\bd(\tau) &\le C e^{\tau/2}\, \int_{\tau}^{\tau + A} (\bhm(\sigma) \, e^{\sigma/2})\, e^{-\sigma}\, d\sigma \\
&\le C e^{\tau/2}\, \bigl(\bhm(\tau+A)\, e^{(\tau + A)/2}\bigr)
        \int_{\tau}^{\tau+A} e^{-\sigma}\, d\sigma \\
&\le C e^{A/2} \bhm(\tau+A)
\end{align*}
Using \eqref{eq-dbar-growth-bound} we have $\bd(\tau+A) \le e^{A/2} \bd(\tau)$, for all $\tau \le \tau_0$.  This implies
\[
\bd(\tau + A) \le C \,\bhm(\tau + A) \quad \text{ for all } \tau\le \tau_0 - A,
\]
or equivalently,
\[
\bd(\tau) \le C \bhm(\tau) \quad \text{ for all } \tau \le \tau_0
\]
since $A$ only depends on universal constants.  This completes the proof of the estimate
\[
c\, \bd(\tau) \le \bar{H}_{\max}(\tau) \le \bd(\tau).
\]
Furthermore,
\[
\bar{A}(\tau)
= 2\int_0^{\bd(\tau)} u^{n-1}\, \sqrt{1 + u_y^2}\, dy
\le C \int_0^{\bd(\tau)} \frac{1}{\lambda_1(y,\tau)}\, dy
\le C \bd(\tau),
\]
where in the last inequality we used Corollary \ref{cor-mon-l} and the fact that we
have a uniform convergence to the cylinder on compact sets, so that
$\lim_{\tau\to-\infty} \lambda_1(0,\tau) = \frac{1}{\sqrt{2(n-1)}}$.  To complete the proof of Theorem \ref{thm-comparable-geometry} we need to show that
\begin{equation}\label{eq:Ad2}
\bd (\tau) \le C\, \bA (\tau), \qquad \tau \le \tau_0.
\end{equation}
for some uniform constant $ C > 0$.  To this end, we have
\[
\begin{split}
\bar{A}(\tau)
&= 2\int_0^{\bd (\tau)} u^{n-1} \, \sqrt{1 + u_y^2}\, dy\\
&\ge 2 \int_{\bd (\tau)/3}^{\bd (\tau)/2} u^{n-1}  \, \sqrt{1 + u_y^2}\, dy \\
&= 2 \sqrt{|\tau|}\, \int_{\frac{\bd (\tau)}{3\sqrt{|\tau|}}}^{\frac{\bd
(\tau)}{2\sqrt{|\tau|}}} \bar{u}^{n-1} \sqrt{1 + \frac{\bar{u}_z^2}{|\tau|}}\, dz\\
&\ge \sqrt{|\tau|}\, \int_{\frac{\bd (\tau)}{3\sqrt{|\tau|}}}^{\frac{\bd (\tau)}{2\sqrt{|\tau|}}} \bar{u}^{n-1} \, dz.
\end{split}
\]
By Lemma \ref{lem-inter} and Corollary \ref{cor-asymp-diam} the last integral in the previous estimate is greater or equal than $c \, \sqrt{|\tau|}$ for a uniform constant $c > 0$, implying
\begin{equation}\label{eq:gelp-Ad}
\bar{A}(\tau) \ge c\, \sqrt{|\tau|} \ge c_1\, \bd (\tau),\qquad \tau\le \tau_0.
\end{equation}
This completes the proof of Theorem~\ref{thm-comparable-geometry}.

\subsection{Asymptotic expansion of the curvature at the tip}
We can improve the bounds we have just derived and describe the limiting behavior of the maximal curvature as $\tau\to-\infty$.

\begin{prop}
\label{prop-Hmax-limit}
The following limits hold
\begin{equation}
\label{eq-H-asymp}
\lim_{t\to-\infty}\frac{H_{\max}(t)}{\sqrt{|t|\, |\log(-t)|}} = \frac{1}{\sqrt{2}}
\qquad \text{and} \qquad \lim_{\tau\to-\infty} \frac{\bhm(\tau)}{\sqrt{|\tau|}} = \frac1{\sqrt{2}}.
\end{equation}
\end{prop}

\begin{proof}
The second limit implies the first one by simple rescaling, so we will only prove the second one here.

Recall again that from Harnack inequality for the mean curvature flow we have
\[
\frac{d}{dt} H_{\max} \ge 0
\]
or equivalently, for the rescaled flow, we have
\begin{equation}
\label{eq-harnack100}
\frac{d}{d\tau} \left(e^{\tau/2}\, \bhm(\tau)\right) \ge 0.
\end{equation}
We also have
\[
\frac{\bd'(\tau)}{\bd(\tau)} =\frac 12 - \frac{\bhm(\tau)}{\bd(\tau)},
\]
with $\bd(\tau) = \sqrt{2|\tau|}\, (1+\delta(\tau))$.  Let $\epsilon > 0$ be a small number.  Integrate previous identity from $\tau$ to $\tau + \epsilon$ to get
\begin{equation}
\label{eq-help1000}
I := \int_{\tau}^{\tau + \epsilon} \frac{\bhm(s)}{\bd(s)}\, ds = \frac{\epsilon}{2} - \log \frac{\bd(\tau+\epsilon)}{\bd(\tau)}.
\end{equation}
Using \eqref{eq-harnack100} we get
\begin{equation}
\label{eq-I-one}
\frac{1}{\epsilon}\, I
= \frac{1}{\epsilon}\,
        \int_{\tau}^{\tau+\epsilon}
                \frac{\bhm(s) e^{s/2}}{\sqrt{2|s|} (1+\delta(s))}\, e^{-s/2}\, ds
\le  \frac{\bhm(\tau+\epsilon)\, e^{\epsilon/2}}{\sqrt{2|\tau + \epsilon|}}
        (1+\delta(\tau)),
\end{equation}
and similarly,
\begin{equation}
\label{eq-I-two}
\frac{1}{\epsilon}\, I
\ge  \frac{\bhm(\tau)}{\sqrt{2|\tau|}} (1+\delta(\tau)).
\end{equation}
On the other hand, by Corollary \ref{cor-asymp-diam} and \eqref{eq-help1000} we get
\[
\frac{1}{\epsilon} I
= \frac 12
        - \frac{1}{\epsilon}
        \log\frac{\sqrt{|\tau + \epsilon|} (1 + \delta(\tau))}
                 {\sqrt{|\tau|}}.
\]
Combining this with \eqref{eq-I-two} yields
\[
\frac{\bhm(\tau)}{\sqrt{2|\tau|}}
\le e^{\epsilon/2}\,(1 + \delta(\tau))\,
\left(\frac 12 - \frac{1}{2\epsilon}\, \log\left(1 + \frac{\epsilon}{\tau}\right)
        + \frac{1}{\epsilon}\, \log \frac{1+\delta(\tau+\epsilon)}{1+\delta(\tau)}
\right),
\]
which implies
\[
\frac{\bhm(\tau)}{\sqrt{2|\tau|}} < \frac 12 + \sigma,
\]
for $\sigma > 0$ arbitrarily small and $\epsilon = \epsilon(\sigma)$ and $\tau < \tau(\epsilon,\sigma)$ chosen so that the estimate holds.  Similarly, using \eqref{eq-I-one} we get
\[
\frac{\bhm(\tau)}{\sqrt{2|\tau|}} > \frac 12 - \sigma,
\]
for $\epsilon = \epsilon(\sigma)$ and $\tau \le \tau(\epsilon,\sigma)$ sufficiently small.  Finally, \eqref{eq-H-asymp} follows as claimed.

\subsection{Proof of Theorem~\ref{thm-asymptotics}, (iii)}
Let $\lambda_s := H_{\max}(s)$ and let $\tilde{M}_t^s = \lambda_s \, (M_{s+\lambda_s^{-2} t} - p_s)$, where $p_s$ is the tip of $M_s$ as in the statement of proposition.  By Corollary \ref{cor-tip}, $H_{\max}(s) = H(p_s,s)$.  Take any sequence $s_i\to -\infty$ and denote by $\tilde{M}_t^i := \tilde{M}_t^{s_i}$ and $\lambda_i := \lambda_{s_i}$ for simplicity.  Then $\bar{H}_i (0,0) = 1$ and
\[
\tilde{H}_i(p,t) = \frac{H(p,s_i + \lambda_i^{-2} t)}{H_{\max}(s_i)} \le \frac{H_{\max}(s_i+\lambda_i^{-2} t)}{H_{\max}(s_i)}.
\]
Using \eqref{eq-H-asymp} we get
\begin{equation}
\label{eq-precise-bound}
\tilde{H}_i(p,t) \le \frac{\sqrt{|s_i + \lambda_i^{-2} t|\, \log |s_i + \lambda_i^{-2} t|}}{\sqrt{|s_i|\log|s_i|}}\, (1 + \epsilon(s_i)),
\end{equation}
where $\lim_{i\to\infty} \epsilon(s_i) = 0$.  For any finite interval in $t$ the above quantity is uniformly bounded as $i\to\infty$ so there exists a smooth limiting flow
\[
\tilde{M}_t^{\infty} = \lim_{i\to\infty} \tilde{M}_t^i,
\]
which is an eternal solution to the mean curvature flow.  It has the property that $\bar{H}_{\infty}(0,0) = 1$.  Furthermore, by \eqref{eq-precise-bound} we have $\tilde{H}_{\infty}(p,t) \le 1$ on $\tilde{M}_{\infty}^t$.  Since $\tilde{M}_{\infty}^t$ arises as a smooth limit of compact solutions, it satisfies Hamilton's Harnack inequality \cite{Ha}, and in particular, $\frac{\pd}{\pd t} \tilde{H}_{\infty}^t \ge 0$.  Together with $\tilde{H}_{\infty}(0,0) = 1$ and $\tilde{H}_{\infty}(p,t) \le 1$ on $\tilde{M}_{\infty}^t$, we have $\frac{\pd}{\pd t} \tilde{H}_{\infty} = 0$ at the origin, for all $t > 0$.  By symmetry we have $\nabla\tilde{H}_{\infty} = 0$ at the origin.  Thus, by the same proof of the equality case of Hamilton's Harnack inequality, which was in fact the observation made in \cite{HO}, $\tilde{M}_{\infty}^t$ must be a translating soliton, and so it must be the Bowl soliton.  Finally, due to the uniqueness of the Bowl with the mean curvature being one at the origin, the subsequential limit is actually a full limit.
\end{proof}

\begin{remark}
In \cite{HO} the statement of part (iii) in Theorem~\ref{thm-asymptotics}, was proved to be true for the ancient oval solution the authors were constructing.  Note that they proved a rescaled limit of ancient oval solutions along a carefully chosen sequence of times $s_i\to -\infty$ must be the Bowl.  In our case, knowing more precise asymptotics which we prove to hold for any rotationally symmetric non-collapsed solution to the mean curvature flow, allows us to show that the rescaled limit along any sequence of times $s_i\to\infty$ will be the unique Bowl solution.
\end{remark}

\section{Constructing the minimizing foliation}
\label{sec-constructing-the-foliation}

\subsection{The foliation}
In this section we construct a foliation of the truncated cone in $\R^{n+1}$ defined by $y\ge y_0$ and $u\le b_0y$ (for suitable constants $y_0$ and $b_0$) whose leaves are ``self shrinkers,'' i.e.~which satisfy \eqref{eq-self-shrinker}.
The normals to these surfaces provide a calibration that allowed us to prove the inner-outer lemma, and they also provided the barriers that we used to deduce convergence in the intermediate region from convergence in the parabolic region.

The foliation we consider consists of rotationally symmetric surfaces, with the $y$-axis as axis of rotation.  Such surfaces are obtained by revolving a curve $\gamma\subset [0,\infty)\times\R$ around the $y$-axis.  We use the same notation as in \S\ref{sec-notation-rot-symm-surfaces}.

\begin{theorem}
\label{thm-about-the-self-shrinker-foliation}
Let $b_0>0$ be given.  There is a constant $y_0>0$ for which the following holds:

{\bfseries\upshape(a)} For each $a\geq y_0$ there is a unique rotationally symmetric embedding $\Sigma_a$ of a disc with the following properties
\begin{enumerate}
\item $\Sigma_a$ is a self shrinker, i.e.~it satisfies \eqref{eq-self-shrinker}
\item $\Sigma_a$ meets the $y$-axis at $y=a$
\item $\Sigma_a$ is contained in the half cylinder $r\leq \sqrt{2(n-1)}, y\geq y_0$
\item the boundary $\pd\Sigma_a$ is contained in the disc $r\leq\sqrt{2(n-1)}, y=y_0$.
\end{enumerate}

{\bfseries\upshape(b)} For each $b\in(0, b_0)$ there exists an embedding $\tS_b$ of a cylinder $[y_0, \infty)\times S^{n-1}$, obtained by rotating the graph of a function $r=\tu_b(y)$ around the $y$ axis, that satisfies
\begin{enumerate}
\item $\tS_b$ is a self shrinker, i.e.~it satisfies \eqref{eq-self-shrinker} and $\tu_b$ is a solution of \eqref{eq-ss}; 
\item $\tS_b$ is asymptotic to the cone with opening slope $b$, in fact, 
$\tu_b(y) = by + \cO\bigl(y^{-1}\bigr)$ as $(y\to\infty)$.

\item the boundaries $\pd\Sigma_a$ and $\pd\tS_b$ are contained in the disc $r\leq b_0y_0$, $y=y_0$.
\end{enumerate}

{\bfseries\upshape(c)} The family of disks $\Sigma_a$ with $a> y_0$, the family of cylinders $\tS_b$ with $0<b<b_0$, and the cylinder $\Gamma$ together form a foliation of the region $r<b_0y, y > y_0$.  The unit normals $\nm$ to $\Sigma_a$, $\tS_b$, and $\Gamma$ define a continuous vector field on this region that is everywhere smooth except possibly on the cylinder $\Gamma$.
\end{theorem}
The self shrinkers $\tS_b$ with conical ends are exactly the ``trumpets'' that were constructed by Kleene and M{\o}ller in Theorem 3 of \cite{KleeneMoller}.  For the surfaces $\tS_b$ we therefore only have to verify that they form a foliation, and that their normals satisfy the same estimate from Lemma~\ref{lem-tan-phi-at-cylinder}. 

It will become clear from the construction that these surfaces $\Sigma_a$ and $\tS_b$ can be extended uniquely as immersions that still satisfy \eqref{eq-self-shrinker}.  However, in the region $y\leq y_0$ they may intersect each other or even themselves.  See Figure~\ref{fig-some-shrinkers}.

It will also be important to have an asymptotic description of the surfaces $\Sigma_a$, $\tS_b$ for large values of $a$, or small values of $b$, respectively.

\begin{theorem}
For each $a\geq y_0$ the surface $\Sigma_a$ is obtained by rotating the graph of a function $r=u(y;a)$ about the $y$-axis. This function satisfies
\begin{equation}
u(y, a) \geq \sqrt{2(n-1)}\,\sqrt{1-y^2/a^2} \text{ for } 0\leq y\leq a
\label{eq-u-lower-bound}
\end{equation}
and
\begin{equation}
u(y, a) \leq
\sqrt{2(n-1)} \sqrt{1 - \Bigl(1-\cO\bigl(\frac{\ln a}{a^2}\bigr) \Bigr) \frac{y^2-10}{a^2}}
\end{equation}
on the interval $5\leq y \leq a-C_0/a$ for some constant $C_0$.
\end{theorem}
The constant $C_0=\Psi(M)+\cO(a^{-2})$ is determined in Proposition~\ref{prop-u-upper-bound}.

\subsection{Construction of $\Sigma_a$}
A surface obtained by rotating a curve $\gamma\subset\R\times[0,\infty)$ about the $y$-axis satisfies the self shrinker equation \eqref{eq-self-shrinker} if and only if the curve $\gamma$ satisfies
\begin{equation}
k -\frac{y} {2}\sin\theta + \Bigl( \frac{r} {2} - \frac{n-1} {r}\Bigr)\cos\theta = 0,
\label{eq-symmetric-shrinker-general}
\end{equation}
where $k$ is the curvature of $\gamma$, and $\theta$ is the angle between the tangent to $\gamma$ and the $y$-axis.

If we parametrize $\gamma$ by Euclidean arc length, then we can rewrite \eqref{eq-symmetric-shrinker-general} as a system of three ordinary differential equations
\begin{equation}
y_s = \cos \theta,  \qquad
r_s = \sin \theta,  \qquad
\theta_s =  \Bigl(\frac{n-1} {r} - \frac{r} {2}\Bigr)\cos \theta
                + \frac{y} {2} \sin \theta.
\label{eq-shrinker-ode}
\end{equation}
This system of differential equations is regular in the region $\{ (y,r,\theta) \in \R\times \R_+\times\R : r>0\}$, and thus there is a unique solution of \eqref{eq-shrinker-ode} for any given point $(y_0, r_0)$ in the upper half plane, and any given initial angle $\theta_0\in\R$.  This solution can be extended uniquely and indefinitely, unless $r\searrow0$, i.e.~unless it reaches the $y$-axis.  These are familiar facts from Riemannian geometry once one realizes that curves satisfying equation~\eqref{eq-symmetric-shrinker-general} are exactly geodesics for the Huisken metric
\[
g = r^{n-1}e^{-(y^2+r^2)/4}\bigl\{ (dy)^2 + (dr)^2\bigr\}.
\]

\subsection{Analysis of $\Sigma_a$ near the tip}
We begin our construction of $\Sigma_a$ by proving the existence of a short segment of $\gamma$ near the $y$-axis.  The surface $\Sigma_a$ can only be smooth if $\gamma$ meets the $y$-axis perpendicularly.  We may therefore represent an initial segment of $\gamma$ as the graph of a function $y=h(r)$.  For such graphs the condition \eqref{eq-symmetric-shrinker-general} on $\gamma$ is equivalent with the differential equation
\begin{equation}
\label{eq-shrinker-graph}
\frac{h_{rr}} {1+h_r^2} + \Bigl(\frac{n-1} {r} - \frac{r} {2}\Bigr)h_r + \frac{h} {2} = 0, \qquad h(0)=a, \quad h_r(0)=0.
\end{equation}
The following arguments establish the existence of a solution to this problem on some short interval $r\in[0, r_a)$ for any $a\in\R$, but we will only need the solution for large $a$.  Therefore we expand the surface near the tip and assume $h$ is of the form
\[
h(r, a) = a - \frac{1} {a} \psi(ar, a)
\]
If we let $\rho = ar$, then \eqref{eq-shrinker-graph} for $h$ is equivalent with
\begin{equation}
\frac{\psi_{\rho\rho}} {1+\psi^2_\rho}
+ \frac{n-1}\rho \psi_\rho - \frac12
=
\frac{1} {2a^2} \bigl( \rho\psi_\rho - \psi\bigr),
\qquad \psi(0) = \psi'(0) = 0.
\label{eq-psi-ode}
\end{equation}
We may regard $\epsilon = 1/2a^2$ as a small parameter.  For $\epsilon=0$ ($a=\infty$) the equation for $\psi$ reduces to
\begin{equation}
\frac{\psi _{\rho\rho}} {1+\psi^2_\rho}
+ \frac{n-1}\rho \psi_\rho = \frac12
\qquad \psi(0) = \psi'(0) = 0,
\label{eq-translator-ode}
\end{equation}
which is exactly the ODE for the rotationally symmetric translating soliton.  The arguments above imply the existence of a solution to \eqref{eq-translator-ode}.  The following expansion is proved in [AV1995].
\begin{lemma}
\label{lem-translating-bowl} 
The differential equation \eqref{eq-translator-ode} has a unique solution $\Psi(\rho)$.  This solution is concave, decreasing, and for $\rho\to\infty$ satisfies
\begin{gather*}
\Psi(\rho) = \frac{1} {4(n-1)} \rho^2 -2\log \rho + C_0 +  \cO(\rho^{-2})\\
\Psi'(\rho) = \frac{1} {2(n-1)}\rho -\frac{2} {\rho} + \cO(\rho^{-3})\\
\Psi''(\rho) = \frac{1} {2(n-1)} + \frac{2} {\rho^2} + \cO(\rho^{-4}).
\end{gather*}
\end{lemma}
\begin{lemma}
\label{lem-tip-expansion}
The ODE \eqref{eq-psi-ode} has a unique solution which can be written as
\[
\psi = \psi(\rho, a) = \tilde\psi\bigl(\rho, \frac{1} {2a^2}\bigr)
\]
where $\tilde\psi(\rho,\epsilon)$ is a real analytic function of two variables.
\end{lemma}
\begin{proof}[Proof of Lemma~\ref{lem-tip-expansion}]
We can rewrite the ODE \eqref{eq-psi-ode} as a system for three variables $(\psi, \chi=\psi_\rho, \rho)$:
\[
\psi_\rho = \chi,
\qquad
\chi_\rho = (1+\chi^2)\Bigl( -\frac{n-1}\rho \chi + \frac12 + \epsilon(\rho\chi-\psi)\Bigr),
\qquad
\rho_\rho =1.
\]
This system is singular at $\rho=0$.  To remove the singularity we multiply with $\rho$ and get
\begin{equation}\left\{
\begin{aligned}
\rho\psi_\rho & = \rho \chi \\
\rho\chi_\rho &= (1+\chi^2)\Bigl( -(n-1)\chi + \tfrac12\rho + \epsilon\rho(\rho\chi-\psi)
\Bigr) \\
\rho\rho_\rho &=\rho
\end{aligned}\right.
\label{eq-psi-chi-rho-system}
\end{equation}
which is an autonomous system provided we take $\ln\rho$ as new ``time'' variable (if $\varsigma = \ln\rho$ then $\pd/\pd\varsigma = \rho\pd/\pd\rho$.)

For any $\epsilon$ the origin is a fixed point of the system \eqref{eq-psi-chi-rho-system}.  The linearization at the origin is
\[
\begin{pmatrix}
0&    0   &0\\
0& -(n-1) &\tfrac12\\
0& 0 &1
\end{pmatrix}
\]
whose eigenvalues are $\{-(n-1), 0, 1\}$.  The solution we are looking for is the fast unstable manifold of the origin corresponding to the eigenvalue $+1$.  The fast unstable manifold is an analytic curve and depends analytically on the parameter $\epsilon$.  It is tangent to the eigenvector $(0, 1, 2n)$ of the linearization, so that near the origin we can write it as a graph $(\psi,\chi) = (\psi(\rho,\epsilon), \chi(\rho, \epsilon))$ over the $\rho$ axis.  We conclude that \eqref{eq-psi-ode} with $\epsilon=1/(2a^2)$ has a unique solution $\tilde\psi(\rho,\epsilon)$, which is a real analytic function of $\rho$ and $\epsilon$.
\end{proof}
\subsection{Extending the leaf $\Sigma_a$}
We have constructed an initial segment of the curve $\gamma$ near $(a, 0)$.  The curve can then be uniquely extended by solving the system of ODE~\eqref{eq-shrinker-ode}.  While the extension is generally not a graph and may ``loop around'' many times (see Figure~\ref{fig-some-shrinkers}), the initial segment that we have constructed is a graph $y=h(r; a)$ defined for $0\leq r\leq M/a$.  Since $h_r<0$ on that interval we can also represent the initial segment as a graph
\[
r=u(y, a), \qquad y\in [y_{Ma}, a],
\]
where $y_{Ma} = a - a^{-1}\psi(M, a)$.  The condition $H+\frac12 X\cdot\nm=0$ for the surface $\Sigma_a$ is equivalent to this ODE for the function $u$
\begin{equation}
\frac{u_{yy}}{1+u_y^2} - \frac{y}{2} u_y + \frac{u}{2}-\frac{n-1}{u} = 0.
\label{eq-shrinker-ode-u}
\end{equation}
In the end it will follow that $u(y, a)$ is defined for all $y\geq 0$, but at this point we only know it is defined on some interval $y\in(y_*(a), a]$.

Consider the quantity
\begin{equation}
w \stackrel{\rm def} = \frac{yu_y}{\frac{u}{2}-\frac{n-1}{u}}
= -\frac{2yuu_y}{2(n-1) - u^2}
= y\frac{d}{dy} \ln\bigl(2(n-1)-u^2\bigr).
\label{eq-w-def}
\end{equation}
The second order equation \eqref{eq-shrinker-ode-u} implies the following first order equation for $w$
\begin{equation}
yw_y = w - \Bigl(\frac12 +\frac{n-1}{u^2}\Bigr) w^2 
+ \frac12 y^2\bigl(1 +u_y^2\bigr) (w-2).
\label{eq-w-ode}
\end{equation}

\begin{prop}
\label{prop-w-at-A}
$\lim\limits_{y\nearrow a} w(y) = \dfrac{2n}{n-1}$.
\end{prop}
\begin{proof}
For any $y<a$ we write $w(y)$ in terms of $\rho$ and $\psi(\rho,a)$:
\[
w = -\frac{y u_y} {\frac{n-1} {u} - \frac{u} {2}}
= - \frac{\bigl(a - \frac{\psi} {a}\bigr) \frac{-1} {\psi_\rho}}
         {\frac{(n-1)a} {\rho} - \frac{\rho} {2a}}
=\frac  {1-\frac{\psi^2} {a^2}}    {n-1-\frac{\rho^2} {2a^2}}
 \frac{\rho} {\psi_\rho}.
\]
It follows that
\[
\lim_{y\nearrow a} w(y) 
= \lim_{\rho\searrow 0}
  \frac{1-\frac{\psi^2} {a^2}} {1-\frac{\rho^2} {2a^2}}
  \frac{\rho} {\psi_\rho} =\frac{1} {\psi_{\rho\rho}(0)},
\]
by l'Hopital's rule.  We can compute $\psi_{\rho\rho}(0)$ by letting $\rho\to 0$ in \eqref{eq-psi-ode}.  Keeping in mind that $\psi=\psi_\rho=0$ at $\rho=0$, and thus $\lim_{\rho\to0}\psi_\rho/\rho = \psi_{\rho\rho}(0)$, we get
\[
n\psi_{\rho\rho}(0) = \frac{1} {2}, \text{ i.e. } \psi_{\rho\rho}(0) = \frac{1} {2n}.
\]

\end{proof}

\begin{prop} On the interval $(y_*(a),a)$ where the function $u(y, a)$ is defined we have $w > 2$.
\end{prop}
\begin{proof}
By proposition~\ref{prop-w-at-A} we know that $w(y)>2$ for $y$ close to $a$.  To reach a contradiction assume that there is a $y_1 \in (0,a)$ with $w(y_1)=2$, and choose $y_1$ to be the largest $y$ with this property.  Then the differential equation \eqref{eq-w-ode} for $w$ implies that at $y_1$ we have
\[
y_1\frac{dw} {dy}
= 2 - \left(\frac12 + \frac{n-1} {u(y_1)^2}\right) 2^2
= -\frac{4(n-1)} {u(y_1)^2} < 0.
\]
On the other hand $w>2$ on $(y_1, a)$ and $w(y_1)=2$ implies $w'(y_1) \geq 0$, and we have our contradiction.
\end{proof}
\begin{prop}
\label{prop-u-lower-bound}
For $y\in(y_*(a), a)$ we have $u(y) > \sqrt{2(n-1)}\sqrt{1-y^2/a^2}$.
\end{prop}
\begin{proof}
Using $u(a)=0$ and \eqref{eq-w-def} we can integrate to get
\[
\ln \frac{2(n-1)-u(y)^2} {2(n-1) - u(a)^2}
= - \int_y^a \frac{w(\bar y)} {\bar y} d\bar y
< -2 \ln \frac{a} {y}.
\]
This leads to $2(n-1)-u^2 > 2 y^2/a^2$ and thus $u(y)^2 > 2(n-1)\bigl(1-y^2/a^2\bigr)$.
\end{proof}
The analogous upper bound for $u$ is less simple.  To derive it we will first estimate $w(y_{Ma})$, and find the upper bound \eqref{eq-w-upper-bound} for $w(y)$ on an interval $[y_0, y_{Ma}]$, for an appropriate choice of the constant $y_0$ (it turns out $y_0=5\sqrt{n-1}$ will work).  Integration then leads to an upper bound for $u(y)$ on that same interval.

\begin{prop}
\label{prop-w-at-yM}
For any $M>0$ we have
\begin{equation}
w(y_{Ma}) 
= \frac{1}{n-1}  \frac{M} {\Psi'(M)} + \cO(a^{-2}), \qquad (a\to\infty).
\label{eq-w-at-yA}
\end{equation}
For $M\to\infty$ we have
\begin{equation}
\frac{M} {\Psi'(M)} = 2(n-1)+\frac{8(n-1)^2} {M^2} + \cO(M^{-4}).
\label{eq-w-in-terms-of-M}
\end{equation}
\end{prop}
\begin{proof}
Our initial segment ends at
\[
y=y_{Ma} = a - a^{-1} \psi(M, a) = a+\cO(a^{-1}).
\]
At $y=y_{Ma}$ we have $u=M/a$ and
\[
u_y = - \psi_\rho(M,a)^{-1}.
\]
Thus we also have
\begin{align*}
w(y_{Ma})
&= - \frac{2y_{Ma} \cdot (M/a) \cdot \bigl(-1/\psi_\rho(M, a)\bigr)}
     {2(n-1) - (M/a)^2}\\
&= 2 \left(\frac{1-\frac{\psi} {a^2}} {2(n-1)-\frac{M^2} {a^2}} \right)
     \frac{M} {\psi_\rho(M, a)} \\
&=\left( \frac{1 + \cO(a^{-2})}{n-1}\right)
   \frac{M} {\Psi'(M) +\cO(a^{-2})}
\end{align*}
To prove \eqref{eq-w-in-terms-of-M} we use the asymptotic expansion from Lemma~\ref{lem-translating-bowl} for $\Psi(\rho)$ for large $\rho$:
\begin{multline*}
\frac{M} {\Psi'(M)}
= \frac{M} {\frac{M}{2(n-1)} - \frac2M + \cO(M^{-3})} \\
=\frac{2(n-1)} {1 - 4(n-1)/M^2 + \cO(M^{-4})}
=2(n-1)+\frac{8(n-1)^2} {M^2} + \cO(M^{-4}),
\end{multline*}
as claimed.
\end{proof}
\begin{prop}
\label{prop-w-upper-barrier}
There exist constants $K$ and $y_0$ for which one can choose $M>0$ such that
\begin{equation}
\label{eq-w-upper-bound}
w(y) \leq 2 +  \frac{K} {a^2-y^2} + \frac{K} {y^2}
\end{equation}
holds for all $y\in[y_0, y_{Ma}]$ with $y>y_*(a)$, provided $a$ is sufficiently large.
\end{prop}
The proof below will show that one can choose any $K>16(n-1)$ and $y_0>\sqrt 2K$.
We will choose
\begin{equation}
\label{eq-K-y0-choice}
K= 20(n-1) \text{ and } y_0 = 5\sqrt{n-1}.
\end{equation}
We will later on show that $y_*(a)\leq 0$ for large enough $a$ so that the condition $y>y_*(a)$ is trivially fulfilled for all $y\geq y_0 $.
\begin{proof}
Define
\[
w_1(y) 
= \frac{1} {y^2} + \frac{1} {a^2-y^2}
= \frac{a^2} {y^2\bigl(a^2-y^2\bigr)}.
\]
We will show that
\[
\bw = 2+Kw_1
\]
is an upper barrier for \eqref{eq-w-ode} on the interval $(y_0,y_{Ma})$ when $K=20(n-1)$ and $y_0=8\sqrt{n-1}$, in the sense that it satisfies
\begin{subequations}%
\begin{align}
&y\bw_y < \bw - \Bigl(\frac12 +\frac{n-1}{u^2}\Bigr) \bw^2 + \frac12 y^2\bigl(1 +u_y^2\bigr) (\bw-2)
\label{eq-wbar-super-solution}\\
&w(y_{Ma}) \leq \bw (y_{Ma})
\label{eq-wbar-bigger-at-yMb}
\end{align}%
\end{subequations}
This implies that $w(y)\leq \bw(y)$ for all $y\in (y_0, y_{Ma})$.

To prove \eqref{eq-wbar-super-solution} we begin by estimating the RHS in \eqref{eq-wbar-super-solution}.  We note that by Proposition~\ref{prop-u-lower-bound} we have $u^2\geq 2(n-1) (1-y^2/a^2)$, which implies
\[
\frac{1} {2} + \frac{n-1} {u^2}
\leq \frac{1} {2} + \frac{a^2} {2(a^2-y^2)} 
\leq \frac{a^2} {a^2-y^2} 
= y^2 w_1.
\]
We also note that for $y\in[y_0, y_{Ma}]$ we have
\[
w_1(y) \leq 2\,\max \Bigl\{ \frac{1} {y_0^2}, \frac{1} {a^2-y_{Ma}^2}\Bigr\}.
\]
For large $a$ we can estimate $a^2-y_{Ma}^2$ by
\begin{align}
a^2-y_{Ma}^2 &= a^2 - \Bigl(a- \frac1a \psi(M, a)\Bigr)^2\notag \\
&=2\psi(M, a) - \frac{1} {a^2}\psi(M, a)^2\notag \\
&=2\Psi(M) + \cO(a^{-2}),
\label{eq-A-yMA-estimate}
\end{align}
so that
\[
\max _{[y_0, y_{Ma}]} w_1(y)
\le 2 \max\Bigl\{ \frac{1}{y_0^2}, \frac{1}{\Psi(M)+\cO(a^{-2})}\Bigr\}.
\]
and
\[
\max _{[y_0, y_{Ma}]} \bw(y)
\le  \max\Bigl\{ 2+\frac{2K}{y_0^2},\; 2+\frac{2K}{\Psi(M)+\cO(a^{-2})}\Bigr\}.
\]
We have chosen $y_0$ so that $y_0 > \sqrt{2K}$, and we will choose $M$ so large that $\Psi(M) > 2K$.  Then
\[
2<\bw < 3 \text{ for }y\in (y_0, y_{Ma})
\]
if $a$ is sufficiently large.

We can now estimate the RHS in \eqref{eq-wbar-super-solution} by
\begin{align*}
\text{RHS}
&= \bw - \Bigl(\frac12 +\frac{n-1}{u^2}\Bigr) \bw^2
   + \frac12 y^2\bigl(1 +u_y^2\bigr) (\bw-2) \\
&\geq -y^2 w_1\cdot \bw^2 + \frac12 y^2 Kw_1\\
&\geq -9y^2 w_1 + 10 y^2 w_1\\
&=y^2 w_1,
\end{align*}
on the interval $[y_0, y_{Ma}]$, and assuming $a$ is large enough.

Turning to the LHS in \eqref{eq-wbar-super-solution} we first compute
\begin{multline*}
y\frac{dw_1}{dy} 
=-\frac{2}{y^2} +\frac{2y^2}{\bigl(a^2-y^2\bigr)^2} 
\leq \frac{2y^2}{\bigl(a^2-y^2\bigr)^2}  \\
\leq \frac{2}{a^2-y_{Ma}^2} \frac{a^2}{a^2-y^2} 
\leq \frac{2}{a^2-y_{Ma}^2} y^2 w_1.
\end{multline*}
Using $\Psi(M) > 2K$ and \eqref{eq-A-yMA-estimate} we conclude that the LHS in \eqref{eq-wbar-super-solution} satisfies
\[
y \frac{d\bw}{dy} 
= K y\frac{dw_1} {dy} 
\leq \frac{2K}{a^2-y_{Ma}^2} y^2 w_1 
= \frac{2K} {2\Psi(M)+\cO(a^{-2})} y^2w_1 
< y^2w_1 
\]
on the interval $(y_0,y_{Ma})$, and assuming $a$ is large enough.  It follows that \eqref{eq-wbar-super-solution} holds.

To prove \eqref{eq-wbar-bigger-at-yMb} we note that, because of \eqref{eq-A-yMA-estimate},
\[
\bw(y_{Ma})
= 2+\frac{K} {y_{Ma}^2} + \frac{K} {b^2-y_{Ma}^2}
\geq  2 + \frac{K} {2\Psi(M)} + \cO(b^{-2}).
\]
On the other hand we have, by Proposition~\ref{prop-w-at-yM},
\[
w(y_{Ma}) = \frac{1}{n-1} \frac{M} {\Psi'(M)} + \cO(a^{-2}).
\]
For large $M$ we also have, by the same Proposition,
\[
\frac{1}{n-1}\frac{M}{\Psi'(M)}
= 2 + \frac{8(n-1)}{M^2} + o(M^{-2}).
\]
Thus if $K>16(n-1)$, and if $M$ is large enough we get $w(y_{Ma}) < \bw(y_{Ma})$ for all sufficiently large $a$.

We have now proved both \eqref{eq-wbar-super-solution} and \eqref{eq-wbar-bigger-at-yMb}.  Since $\bw$ is an upper barrier for $w$ on the interval $\max\{y_0, y_*(a)\} \leq y\leq y_{Ma}$ the Proposition follows.
\end{proof}

\begin{prop}
\label{prop-u-upper-bound}
If $a$ is sufficiently large then $y_*(a)\leq 0$, and for $8\sqrt{n-1}\leq y \leq y_{Ma}$ one has
\begin{equation}
u(y)^2 \leq
2(n-1)\left[1 - \Bigl(1-\cO\bigl(\frac{\ln a}{a^2}\bigr) \Bigr) \frac{y^2-10(n-1)}{a^2}
\right].
\label{eq-u-upper-bound}
\end{equation}
In particular, on any bounded interval $8\sqrt{n-1} \le y\le L$ one has
\begin{equation}
u(y)^2 \leq
2(n-1)\left[1 - \frac{y^2-10(n-1)}{a^2} + \cO\bigl(\frac{\ln a}{a^4}\bigr)
\right].
\label{eq-u-upper-bound-y-leq-L}
\end{equation}
\end{prop}
\begin{proof}
We integrate the upper bound for $w(y)$ between $y$ and $y_{Ma}$
\begin{align*}
\ln \frac{2(n-1)-(M/a)^2}{2(n-1)-u(y)^2}
&= \int_{y}^{y_{Ma}} \frac{w(\eta)}{\eta} d\eta \\
&\leq \int_{y}^{y_{Ma}}
    \Bigl\{ \frac{2}{\eta} + \frac{K}{\eta^3} + \frac{K}{\eta(a^2-\eta^2)} \Bigr\}d\eta\\
&\leq\ln \frac{y_{Ma}^2}{y^2} + \frac{K}{2y^2}
+ \Bigl[\frac{K}{2a^2} \bigl\{\ln \eta^2 - \ln (a^2-\eta^2)\bigr\}\Bigr]_y^{y_{Ma}}\\
&\leq\ln \frac{y_{Ma}^2}{y^2} + \frac{K}{2y^2} + \frac{CK\ln a}{a^2}.
\end{align*}
Thus
\begin{align}
2(n-1)-u(y)^2
&\geq \Bigl(2(n-1)-\cO\bigl(\frac{1}{a^2}\bigr)\Bigr)
      \frac{y^2}{y_{Ma}^2} e^{-K/2y^2}
      \Bigl(1+\cO\bigl(\frac{\ln a}{a^2}\bigr)\Bigr)\notag\\
&=\Bigl(2(n-1)-\cO\bigl(\frac{\ln a}{a^2}\bigr)\Bigr) \frac{y^2-K/2}{a^2}.
\label{eq-2-minus-usq-estimate}
\end{align}
where we have used $e^{-x}\ge 1-x$ for $x\ge 0$.

This implies \eqref{eq-u-upper-bound} under the assumption that $a$ is large enough (recall that we have chosen $K=20(n-1)$ in \eqref{eq-K-y0-choice}).  We still have to show that $y_*(a)\leq 0$.

The upper bound \eqref{eq-u-upper-bound} for $u$, combined with the complementing lower bound from Proposition~\ref{prop-u-lower-bound} implies that the solution $u(y, a)$ of \eqref{eq-shrinker-ode-u} is a priori bounded and bounded away from $u=0$ on the interval $[8\sqrt{n-1}, y_{Ma}]$.  The bound for $w$ implies that the derivative $u_y = \frac{w} {y} \bigl(\frac u2 - \frac {n-1}u\bigr)$ also is a priori bounded on $[8\sqrt{n-1}, y_{Ma}]$.  It follows that the solution $u(y)$ to \eqref{eq-shrinker-ode-u} can be extended from $y=a$ all the way down to $y=8\sqrt{n-1}$.

At $y=8\sqrt{n-1}$ our estimate \eqref{eq-2-minus-usq-estimate} implies that $2(n-1)-u^2 = \cO(a^{-2})$.  Together with our bound for $w(5\sqrt{n-1})$ and the relation $u_y = \frac{w}{2uy}(u^2-2(n-1))$ we then find that $u_y(8\sqrt{n-1})=\cO(a^{-2})$.  In other words, $(u(8\sqrt{n-1}), u_y(8\sqrt{n-1}))$ is $\cO(a^{-2})$ close to $(\sqrt{2(n-1)},0)$.  Standard theorems on the dependence of solutions to ODE on parameters then imply that the solution can be continued from $y=8\sqrt{n-1}$ down to $y=0$ (and beyond).  Thus $y_*(a)\leq0$ for large enough $a$.
\end{proof}

\subsection{Proof of Lemma~\ref{lem-uA-inner-expansion}}
We now use the bounds in Proposition~\ref{prop-u-lower-bound} and equation \eqref{eq-u-upper-bound} to derive finer asymptotics of the minimizers $u(y,a)$ on bounded intervals $[0, L]$ for large $a$.

\begin{prop}
\label{prop-u-sqrt2-roughestimate} Let $L>0$ be given.  For $y\in [5, 4L]$ we have
\[
\left| u-\sqrt{2(n-1)}\Bigl(1 - \frac{y^2}{2 a^2}\Bigr)\right| \leq \frac{C_n}{a^2}
\]
where $C_n$ only depends on the dimension $n$, and where $a$ must be sufficiently large.
\end{prop}
\begin{proof}
From Propositions ~\ref{prop-u-lower-bound} and~\ref{prop-u-upper-bound} we know that $u^2$ is bounded on the interval $[8\sqrt{n-1}, 4L]$ by
\[
2(n-1)\Bigl(1 - \frac{y^2}{a^2}\Bigr) 
\leq u^2 
\leq 2(n-1)\Bigl(1 - \frac{y^2}{a^2}\Bigr)
  + \frac{20(n-1)^2}{a^2} 
  + \cO\Bigl(\frac{\ln a}{a^4}\Bigr).
\]
We may assume that $|\cO((\ln a)/a^4)| \leq 1/a^2$ provided we choose $a$ sufficiently large.  Dividing by $2(n-1)$ we arrive at
\[
1 - \frac{y^2}{a^2}
\leq \frac{u^2}{2(n-1)} 
\leq 1 - \frac{y^2}{a^2}
  + \frac{C_n}{a^2}.
\]
with $C_n = \bigl(20(n-1)^2+1\bigr)/(2(n-1))$.
Use a Taylor expansion to take the square root:
\[
1 - \frac{y^2}{2a^2} + \cO(y^2/a^4)
\leq \frac{u}{\sqrt{2(n-1)}} 
\leq 1 - \frac{y^2-C_n}{2a^2} + \cO(y^2/a^4).  
\]
For $y\leq 4L$ we may assume that $|\cO(y^2/a^4)| \leq 1/a^2$, is $a$ is large enough.  After replacing the error term $\cO(y^2/a^4)$ by $1/a^2$ the Proposition immediately follows.
\end{proof}

To derive the more precise estimate from Lemma \eqref{lem-uA-inner-expansion} we look at the almost linear equation satisfied by the difference between $u$ and $\sqrt{2(n-1)}$.  Thus we define $v$ by
\[
u = \sqrt{2(n-1)} \bigl(1 + \frac{v}{a^2}\bigr).
\]
Then the differential equation \eqref{eq-shrinker-ode-u} for $u$ implies that $v$ satisfies
\begin{equation}
\label{eq-v} 
v_{yy} = 
\bigl(1+\eps^2 v_y^2\bigr) 
\Bigl\{ 
\frac{y}{2} v_y -\frac{2+\eps v}{2+2\eps v} v \Bigr\}
\end{equation}
where $\eps = a^{-2}$.

The inequality in Proposition~\ref{prop-u-sqrt2-roughestimate} implies that the solutions $v(y)$ we get for different values of $a$ all satisfy
\begin{equation}
- \frac12 y^2  - C_n
\leq v
\leq - \frac12 y^2 + C_n.
\label{eq-v-bound}
\end{equation}
on the interval $[8\sqrt{n-1}, 4L]$.  Since the function $u$ is concave, $v$ is also concave, and the uniform bounds on $[8\sqrt{n-1},4L]$ for $v$ imply that their derivatives are bounded on a smaller interval, say $[10\sqrt{n-1}, 3L]$.  Using the differential equation \eqref{eq-v} we also get bounds for the second derivatives.  Thus for some constant $C$ we have
\[
|v| + |v_y| + |v_{yy}| \leq C
\]
on the interval $[10\sqrt{n-1}, 3L]$ for all large enough $a$.  Ascoli's theorem tells us that any sequence $a_i\to\infty$ has a subsequence for which $v(y, a_i)$ converges in $C^1$, and by the differential equation also in $C^2$.  Any possible limit $\bv$ is a solution of
\begin{equation}
\label{eq-v-limit}
v_{yy} = \frac{y}{2} v_y - v.
\end{equation}
This linear differential equation is known as a Hermite equation.  It has one polynomial solution, namely
\[
v_0(y) = y^2-2.
\]
The general solution can be given in terms of Hermite functions, most of which are not polynomial.  To choose a specific solution of \eqref{eq-v-limit} we can find a solution which is odd.  There are many possible representations, e.g.~one can represent the function as a power series
\[
v_1(x)= 
\sum_{k=0}^\infty 
\frac{(y/2)^{2k+1}}
     {(-\tfrac12)(\tfrac12)(\tfrac32)\cdots(k-\tfrac32)} .
\]
Another representation is in terms of a contour integral (see e.g. Courant-Hilbert \cite[Ch.7, p.508]{CHv1}).  One can also substitute $v_1(y) = \psi(y)(y^2-2)$ in the differential equation and solve for $\psi(y)$.  This leads to\footnote{%
The integral is singular at $\eta=\sqrt{2}$ even though the solution $v_1(y)$ is smooth everywhere.  It turns out that one can regard the integral as a contour integral along any path from $0$ to $y$ that avoids the singularity at $\sqrt2$.  The residue of the integrand $e^{\eta^2/4}/(\eta^2-2)^2$ vanishes, so the integral is independent of the chosen path from $\eta=0$ to $\eta=y$.}
\[
v_1(y) = -(y^2-2) \int_0^y \frac{e^{\eta^2/4} d\eta}{\bigl(\eta^2-2\bigr)^2}\;.
\]
For $y\to\infty$ the solution $v_1$ is much larger and grows much faster than the polynomial solution $v_0(y) = y^2-2$.  One has
\begin{equation}
\label{eq-v1-growth} 
v_1(y) = \frac{2+o(1)}{y^3} e^{y^2/4}
= e^{y^2/4 + o(y^2)} \qquad (y\to\infty).
\end{equation}
We return to the possible limits of $v(y, \eps)$ as $\eps\to0$.  Assume that for some sequence $\eps_i\to0$ one has $v(y, \eps_i)\to \bv(y)$.  Then there are $\alpha$, $\beta$ such that
\[
\bv(y) = \alpha (y^2-2) + \beta v_1(y).
\]
Evaluating this at $y=3L$ and $y=2L$ we get this system of linear equations for $\alpha$ and $\beta$:
\[
\bv(3L) = \alpha (9L^2-2) + \beta v_1(3L), \text{ and } \bv(2L) = \alpha (4L^2-2) + \beta v_1(2L).
\]
After solving these we find
\begin{align*}
\alpha &= -\frac{v_1(2L)\bv(3L) - v_1(3L)\bv(2L)}{(4L^2-2)v_1(3L) - (9L^2-2)v_1(2L)}\\
\beta &= \frac{(4L^2-2)\bv(3L) - (9L^2-2)\bv(2L)}{(4L^2-2)v_1(3L) - (9L^2-2)v_1(2L)}
\end{align*}
In these expressions $v_1(3L) = e^{9L^2/4 + o(L^2)}$ while $v_1(2L)=e^{L^2+o(L^2)}$, so we can rewrite $\alpha$ as
\[
\alpha = \frac{\bv(2L) - e^{-5L^2/4+o(L^2)}\bv(3L) } {4L^2-2 - e^{-5L^2/4+o(L^2)}(9L^2-2) }.
\]
The limit $\bv(y)$ must satisfy \eqref{eq-v-bound}, so that we can write $\bv(y) = -y^2/2 +\theta$ with $|\theta| < C_n$.  Also, $\bv(3L) = \cO(L^2) = e^{o(L^2)}$, and $\bv(2L)=e^{o(L^2)}$.  We get
\[
\alpha 
= - \frac{2L^2 + \theta + e^{-5L^2/4+o(L^2)}} {4L^2 - 2 - e^{-5L^2/4+o(L^2)}} 
= - \frac{1} {{2}} +\cO(L^{-2}).
\]
A similar computation applied to $\beta$ leads to
\begin{align*}
\beta = \frac{1} {v_1(3L)} \frac{\bv(3L) - \frac{9L^2-2}{4L^2-2}\bv(2L)} {1 + e^{-5L^2/4+o(L^2)}} = e^{-9L^2/4+o(L^2)}.
\end{align*}
On the interval $[10\sqrt{n-1}, L]$ we therefore can estimate the two terms in $\bv(y) = \alpha (y^2-2) + \beta v_1(y)$ by
\[
\alpha (y^2-2) =- \frac{1}{2}(y^2-2) + \cO(y^2/L^2)
\]
and
\[
\beta v_1(y) = e^{-9L^2/4 + o(L^2)} e^{y^2/4+o(L^2)} = e^{-5L^2/4+o(L^2)} = \cO(e^{-L^2}).
\]
Adding these two estimates and substituting them in $u = \sqrt{2(n-1)} (1 + v/a^2)$ we get
\[
u(y) = \sqrt{2(n-1)}  \Bigl(1 - \frac{y^2-2} {2a^2} \Bigr)
+ \frac{1} {a^2} \cO\bigr(\frac{y^2} {L^2} + e^{-L^2}\bigr).
\]
On any finite interval $[10\sqrt{n-1}, M]$ we have $y^2/L^2\leq M^2/L^2$ so that
\[
u(y) = \sqrt{2(n-1)} \Bigl(1 - \frac{y^2-2} {2a^2} \Bigr)
+ \frac{1} {a^2} \cO\bigr(L^{-2} + e^{-L^2}\bigr).
\]
Since we can make $\cO\bigr(L^{-2} + e^{-L^2}\bigr)$ arbitrarily small by choosing $L$ large enough, we finally arrive at the asymptotic expansion \eqref{eq-uA-inner-expansion}, which concludes the proof of Lemma~\ref{lem-uA-inner-expansion}.

\subsection{Estimating the unit normal of $\Sigma_a$ near the cylinder}
In the derivation of the inner-outer estimates we will need an asymptotic estimate of the unit normal $\nm$ to the hypersurfaces $\Sigma_a$, at least in a neighborhood of the cylinder $r=\sqrt{2(n-1)}$.  Since the hypersurface $\Sigma_a\subset\R^{n+1}$ is rotationally symmetric we can parametrize it by
\[
(y, \omega) \in J\times S^{n-1} \mapsto \bigl(y, u(y, a) \bomega\bigr) \in \R\times\R^n.
\]
The unit normal is therefore given by
\[
\nm = \frac{\bigl(-u_y, \bomega\bigr)} {\sqrt{1+u_y^2}}
\]
Using the quantity $w = 2yuu_y/(u^2-2)$ we can write this normal as a function of $(y, u)$, namely
\begin{equation}
\label{eq-angle-as-fn-of-yu}
\nm = \bigl(-\sin\varphi,  \cos \varphi\, \bomega\bigr)
\text{ where }
\tan \varphi =  u_y = \frac{w} {2y} (u^2 - 2).
\end{equation}

\subsection{The normal variation}
\label{sec-normal-variation}
The normal variation $V$ of the family of hypersurfaces $\Sigma_a$, is defined by choosing a smooth family of parametrizations $X_a:\R^n \to \R^{n+1}$ of $\Sigma_a$ and setting
\[
V = \nm \cdot \frac{\pd X_a} {\pd a}.
\]
The normal velocity does not depend on the choice of the particular parametrizations $X_a$.

\begin{lemma}
\label{lem-V-positive-on-Sigma-b}
There is a $y_0>0$ such that $V>0$ on the part of $\Sigma_a$ on which $y\geq y_0$.

Different minimizers $\Sigma_a$ and $\Sigma_{a'}$ ($a\neq a'$) do not intersect in the region $y\geq y_0$.  The $\Sigma_a$ smoothly foliate the region within the cylinder $r<\sqrt{2(n-1)}$ with $y>y_0$.
\end{lemma}

We prove the lemma by studying a linear differential equation $\cL(V)=0$ that $V$ satisfies.  On most of the surface $\Sigma_a$ we can find that the function $W=e^{\|X\|^2/8}$ is a supersolution for $\cL$ (i.e.~$\cL(W)<0$).  In the region near the tip, defined by $r\leq M/a$, we use our expansion of the surface $\Sigma_a$ in powers of $a^{-2}$ to get a good approximation of $V$ at $r=M/a$.  Comparison with the supersolution $W$ then allows us to conclude that $V\neq0$ on $\Sigma_a$ when $y$ is large enough.

\subsubsection{The Jacobi equation}

\begin{prop}
The normal variation $V$ satisfies
\begin{equation}
\label{eq-jacobi}
\cL(V) \stackrel{\rm def}=
\Delta V - g^{kl} \nabla_k \phi \nabla_l V + \bigl(\|A\|^2+\tfrac12\bigr)V = 0
\end{equation}
where $\phi = \tfrac14\|X\|^2$, $A$ is the second fundamental form of $\Sigma_a$, and where $g_{ij} = \nabla_iX \cdot \nabla_jX$ is the induced metric on $\Sigma_a$.
\end{prop}

\begin{proof}
Since all surfaces $\Sigma_a$ satisfy the shrinker equation $H+\frac12 X\cdot\nm=0$, we have
\[
\frac{\pd} {\pd a} \Bigl(H+\tfrac12 X\cdot \nm\Bigr) = 0.
\]
The first variation of the mean curvature is
\[
\frac{\pd H} {\pd a} = \Delta V +|A|^2V.
\]
The first variation of the unit normal is the tangential gradient of the velocity $V$,
\[
\frac{\pd\nm} {\pd a} = -g^{kl} \nabla_kX \nabla_l V,
\]
and hence
\[
\frac{\pd X\cdot\nm} {\pd a}
= \frac{\pd X} {\pd a}\cdot \nm - X\cdot \bigl(\nabla_k V \nabla_kX\bigr)
= V - \bigl(X\cdot \nabla_kX\bigr) \nabla_k V
= V - 2g^{kl}\nabla_l\phi \nabla_k V.
\]
Combine with the equation for $\frac\pd{\pd a}H$ and we find \eqref{eq-jacobi}.
\end{proof}

\subsubsection{The normal variation near the tip}

We can parametrize $\Sigma_a$ near the tip by $X_a : (\rho, \bomega) \in \R\times S^{n-2} \to \R^n$ given by
\[
X_a(\rho,\bomega)
= \bigl(y, r\bomega\bigr)
= \Bigl(a - a^{-1} \psi(\rho, a), \frac{\rho} {a} \bomega \Bigr).
\]
The unit normal and first variation are then given by
\[
\nm = \frac{\bigl(1, \psi_\rho\,\bomega\bigr)} {\sqrt{1+\psi_\rho^2}},\qquad
\frac{\pd X_a} {\pd a} = \bigl(1+\cO(a^{-2}) , \cO(a^{-2})\,\bomega\bigr),
\]
so that the normal variation is
\[
V = \nm\cdot \frac{\pd X} {\pd a} = \frac{1+\cO(a^{-2})} {\sqrt{1+\psi_\rho^2}}.
\]
For large $a$ we get
\[
V = \frac{1} {\sqrt{1+\Psi'(\rho)^2}} + \cO(a^{-2}),
\]
uniformly for $\rho\in[0,M]$.  In order to compare $V$ with the supersolution $W$ in the intermediate region later on, we will also need $V_\rho$, or rather the ratio $V_\rho/V$.  This ratio is given by
\begin{equation}
\frac{V_\rho}{V} 
= - \frac{\Psi'(\rho)\Psi''(\rho)}{1+\Psi'(\rho)^2}
+ \cO(a^{-2}).
\label{eq-Vrho-over-V}
\end{equation}

\subsubsection{The normal variation at the tip}
We note that at $\rho=0$, i.e.,~at the tip, we have $X_a(0, \bomega) = (a, 0)$ and $\nm = (1,0)$, so that
\begin{equation}
\label{eq-V-at-the-tip-is-one}
V = \nm\cdot \frac{\pd X_a} {\pd a} = 1.
\end{equation}
In particular, $V>0$ at the tip.  To prove that $V>0$ for $y\geq y_0$ we only have to show that $V\neq 0$ in this region.

\subsubsection{Upper barrier for $V$ in the intermediate region}
\begin{prop}
The function $W = e^{\phi/2} = e^{\|X\|^2/8}$ satisfies $\cL(W)\leq0$ on the intermediate region $y_0 \leq y\leq y_{Ma}$ for large enough $a$ and $y_0$.
\end{prop}
\begin{proof}
We compute
\begin{align*}
\cL(e^{\theta\phi})
&= \Delta e^{\theta\phi} - \nabla e^{\theta\phi}\cdot\nabla\phi + \bigl(\|A\|^2+\tfrac12\bigr)e^{\theta\phi}\\
&= \theta e^{\theta\phi}\Delta\phi +\theta^2e^{\theta\phi}\|\nabla \phi\|^2 -\theta e^{\theta\phi}\nabla\phi\cdot\nabla\phi
+\bigl(\|A\|^2+\tfrac12\bigr)e^{\theta\phi}\\
&= \theta e^{\theta\phi}\Delta\phi +\bigl(\theta^2-\theta\bigr) e^{\theta\phi}\|\nabla \phi\|^2
+\bigl(\|A\|^2+\tfrac12\bigr) e^{\theta\phi}\\
&= e^{\theta\phi} \Bigl\{ \theta\Delta\phi + (\theta^2 - \theta) \|\nabla\phi\|^2 + \|A\|^2+\tfrac12 \Bigr\}.
\end{align*}
Using $H+\frac12 X\cdot \nm=0$ we find the following derivatives for $\phi=\frac14\|X\|^2$
\begin{align*}
\nabla \phi &= g^{kl}\nabla_k\phi \nabla _lX =g^{kl}\tfrac12 \bigl( X\cdot \nabla_k X\bigr) \nabla_lX
= \tfrac12 \Bigl( X- (X\cdot\nm)\nm\Bigr)\\
\|\nabla\phi\|^2 &= \tfrac14 \Bigl(\|X\|^2 - (X\cdot\nm)^2\Bigr)
= \phi - H^2\\
\Delta\phi &= \tfrac12 g^{kl}\nabla_k\Bigl(\nabla_l X\cdot X\Bigr) = \frac{n-1} {2} + \tfrac12 X\cdot(H\nm) = \frac{n-1} {2} - H^2
\end{align*}
and therefore
\begin{align*}
e^{-\theta\phi}\cL(e^{\theta\phi}) &= \theta\frac{n-1} {2} - \theta H^2 -\theta(1-\theta)\phi
+ \theta(1-\theta)H^2 + \|A\|^2+\tfrac12 \\
&= \frac{\theta n+1-\theta} {2} +\|A\|^2 - \theta^2 H^2 -\theta(1-\theta)\phi.
\end{align*}
From here on we set $\theta=\tfrac 12$, so that
\[
e^{-\phi/2}\cL\bigl(e^{\phi/2}\bigr) = \frac{ n+1} {4} +\|A\|^2 - \tfrac14 H^2 -\tfrac14 \phi,
\]
and thus
\begin{equation}
\label{eq-W-upper-barrier}
e^{-\phi/2}\cL(W) \leq \frac{n+1}{4} + \|A\|^2 - \tfrac{1}{16}\|X\|^2.
\end{equation}
The principal curvatures of the hypersurface $\Sigma_a$ are given by \eqref{eq-principal-curvatures}, so that
\[
\|A\|^2
= \frac{u_{yy}^2}{\bigl(1+u_y^2\bigr)^3} + \frac{n-1}{u^2\bigl(1+u_y^2\bigr)^2}
\le \left(\frac{u_{yy}}{1+u_y^2}\right)^2 + \frac{n-1}{u^2}.
\]
Using the ODE \eqref{eq-shrinker-ode-u} for $u(y)$, we get
\[
0<\frac{-u_{yy}}{1+u_y^2}
= \tfrac12\bigl(-yu_y +u\bigr) - \frac{n-1}{u}
< \tfrac12\bigl(-yu_y +u\bigr).
\]
Hence, using $u<\sqrt{2(n-1)}$ on $[y_0, y_{Ma}]$,
\[
\|A\|^2
\leq \tfrac12 y^2 u_y^2 + \frac12 u^2 + \frac{n-1}{u^2}
\leq \tfrac12 y^2 u_y^2 + n-1 + \frac{n-1}{u^2}
\]
On $[y_0, y_{Ma}]$ we also have $u_{yy}<0$, so that
\[
0 > u_y(y) > u_y(y_{Ma}) 
= \frac{-1}{\psi_\rho(M, a)}
= \frac{-1}{\Psi'(M)} + \cO(a^{-2}).
\]
Moreover, $u$ is concave, so $u^{-1}$ is convex, and we can find an upper bound for $u^{-1}$ by interpolating between $u(y_0) = \sqrt{2(n-1)}+\cO(a^{-2})$ and $u(y_{Ma}) = M/a$.  Keeping in mind that $y_{Ma} = a + \cO(a^{-1})$, and also $y_{Ma}-y_0 = a \bigl(1-\cO(a^{-1})\bigr)$, we find
\begin{align*}
u(y)^{-1}
&\le u(y_0)^{-1} + \frac{y-y_0} {y_{Ma}-y_0} \bigl( u(y_{Ma})^{-1} - u(y_0)^{-1}\bigr)\\
&\le \frac{1}{\sqrt{2(n-1)}} 
   +\cO(a^{-2}) + \frac{y}{y_{Ma}} \bigl(1-y_0/y_{Ma}\bigr)^{-1} \frac{a}{M}\\
&\le \frac{1}{\sqrt{2(n-1)}} + \cO(a^{-2}) + \frac{y}{M} \bigl(1+\cO(a^{-1})\bigr)\\
&\le \frac{y}{M} + C_n,
\end{align*}
where $C_n$ is some constant that only depends on $n$.
Since $(\alpha+\beta)^2 \le 2\alpha^2 + 2\beta^2$ for all $\alpha, \beta$, we see that on the interval $[y_0, y_{Ma}]$
\[
u(y)^{-2} \leq 2 \frac{y^2}{M^2} + 2C_n^2
\]
holds if $a$ is sufficiently large.

Combining our estimates for $u_y$ and $u^{-2}$ we find that on the interval $[y_0, y_{Ma}]$ the curvature is bounded by
\begin{align*}
\|A\|^2
&\le \tfrac12 y^2u_y^2 + \frac{n-1} {u^2} + n - 1 \\
&\le \Bigl(\frac{1} {2\Psi'(M)^2} + \cO(a^{-2})\Bigr) y^2
+ \frac{2(n-1)}{M^2}y^2 + C'_n \\
&\le \Bigl(\frac{1} {2\Psi'(M)^2} + \frac{4(n-1)}{M^2} + \cO(a^{-2})\Bigr) y^2 + C'_n.
\end{align*}
Hence we can choose $M$ so that for large enough $b$ one has
\[
\|A\|^2 \leq C_n' + \frac{1} {32} y^2 \leq C_n' + \frac{1} {32}\|X\|^2,
\]
and thus
\[
e^{-\phi/2} \cL(W)
\leq C_n' - \bigl(\tfrac1{16} - \tfrac1{32}\bigr) \|X\|^2
\leq C_n' - \tfrac1{32} \|X\|^2
\leq C_n'- \tfrac1{32}y_0^2.
\]
Thus if $y_0$ is large enough, we get $\cL(W)<0$ for $y_0\leq y \leq y_{Ma}$.
\end{proof}

\begin{prop}
For large $b$ one has
\begin{equation}
\label{eq-Wrho-over-W-expanded}
\frac{W_\rho} {W} 
= - \tfrac 14  \Psi'(M) + \cO(b^{-2}).
\end{equation}
\end{prop}
\begin{proof}
We parametrize $\Sigma_a$ by
\[
X(\rho, \bomega) =\Bigl( a - a^{-1} \psi(\rho, a) , \frac\rho a \bomega \Bigr), \qquad (\rho>0, \bomega\in S^{n-2}).
\]
Then $X_\rho = \bigl(a^{-1}\psi_\rho, a^{-1}\bomega\bigr)$, and
\[
\phi_\rho = \tfrac12 X\cdot X_\rho = -\tfrac12 \psi_\rho (\rho, a) + \cO(a^{-2}) = -\tfrac12 a\Psi'(\rho) + \cO(a^{-2}).
\]
The Proposition now follows from $W_\rho/W = (\ln W)_\rho = \frac12 \phi_\rho$.
\end{proof}

\subsubsection{Proof of Lemma~\ref{lem-V-positive-on-Sigma-b}}
With the intention of reaching a contradiction we suppose that $V=0$ at some $y_1\geq y_0$.  Let $\Sigma^*$ be the part of $\Sigma_a$ on which $y_1\leq y \leq y_{Ma}$.

Since $W>0$ everywhere on $\Sigma^*$, and since $\Sigma^*$ is compact, there is a smallest $\lambda>0$ for which $V\leq \lambda W$ on $\Sigma^*$.  Again because $\Sigma^*$ is compact there is a $y_2\in[y_1, y_{Ma}]$ at which $V=\lambda W$.  Since $V=0$ when $y=y_1$ we clearly must have $y_2 > y_1$.

By the maximum principle it follows from $\cL(V)=0$ and $\cL(\lambda W)<0$ on $\Sigma^*$ that the point of contact $y_2$ between $V$ and $\lambda W$ cannot be an interior point.  Thus $y_2=y_{Ma}$, i.e., $V<\lambda W$ for $y<y_{Ma}$ and $V=\lambda W$ at $y=y_{Ma}$.  In terms of the coordinate $\rho$ on $\Sigma_a$, which increases as $y$ decreases, this implies $V_\rho \leq \lambda W_\rho$ at $y=y_{Ma}$ (i.e.~at $\rho=M$).  Eliminating $\lambda$ we conclude
\begin{equation}
\frac{V_\rho} {V}  \leq \frac{W_\rho} {W} \text{ at }y=y_{Ma},\text{ i.e.~when } \rho=M.
\label{eq-to-be-contradicted}
\end{equation}

On the other hand, we recall that at $y=y_{Ma}$ we had the two expansions \eqref{eq-Vrho-over-V} and \eqref{eq-Wrho-over-W-expanded}:
\[
\frac{V_\rho} {V} = - \frac{\Psi'(M)\Psi''(M)} {1+\Psi'(M)^2} + \cO(a^{-2})
\text{ and } 
\frac{W_\rho} {W} = - \tfrac 14 \Psi'(M) + \cO(a^{-2}) 
\text{ for }a\to\infty.
\]
We can approximate the leading terms in these expansions by using Lemma~\ref{lem-translating-bowl} which says that for large $M$
\[
\Psi'(M) = \frac{M} {2(n-1)} + \cO (M^{-1}),
\qquad 
\Psi''(M) = \frac{1}{2(n-1)} + \cO(M^{-2}).
\]
We find
\begin{gather*}
-\frac{\Psi'(M)\Psi''(M)} {1+\Psi'(M)^2} = -\frac{1} {M} + \cO(M^{-3}), \\
-\tfrac14 \Psi'(M) = -\frac{M} {8(n-1)} + \cO(M^{-1}).
\end{gather*}
If we choose $M$ large enough we can guarantee that
\[
-\tfrac 14 \Psi'(M) < -\frac{\Psi'(M)\Psi''(M)} {1+\Psi'(M)^2},
\]
With this choice of $M$ we will also have $\frac{W_\rho} {W} < \frac{V_\rho} {V}$ at $y=y_{Ma}$ if $a$ is sufficiently large.

Thus our assumption that $V=0$ at some $y=y_1$ implies both that $V_\rho/V \leq W_\rho/W$ and $V_\rho/V > W_\rho/W$ at $y=y_{Ma}$.  This contradiction shows that $V\neq 0$ at all $y\in[y_0, y_{Ma}]$, if $a$ is sufficiently large.

\subsection{The foliation outside the cylinder}
\label{sec-foliation-outside}
We now verify that Kleene and M{\o}ller's ``trumpets'' foliate a conical region $\sqrt{2(n-1)}\le r\le b_0 y$, $y\ge y_0$, as claimed in Theorem~\ref{thm-about-the-self-shrinker-foliation}, and that their normals satisfy the estimate from Lemma~\ref{lem-tan-phi-at-cylinder}.

Each self-shrinker $\tS_b$ is obtained by revolving the graph of a solution $\tu_b$ of \eqref{eq-ss} about the $y$-axis.  For each $b>0$ Kleene and M{\o}ller proved existence of a unique $\tu_b$ with $\tu_b(y) = by + \cO(y^{-1})$ ($y\to\infty$).  They showed that $\tu_b(y)$ is defined for all $y\ge0$, and that it is a convex function, so that $\tu_b(y)\ge by$ and $\tu_b'(y) \le b$.  They observed  that the $\tu_b$ can be obtained by a contraction mapping argument, and hence that the $\tu_b$ depend smoothly on the parameter $b>0$.  To show that the $\tS_b$ define a foliation we therefore only have to show that the normal variation does not change sign.  We can do this using the same method as for the foliation $\Sigma_a$ in the interior of the cylinder:  in fact, this case is a bit easier since the analysis of the distant region $y\to\infty$ is already contained in Kleene\&M{\o}ller's work.

For any $b>0$ we parametrize $\tS_b$ by
\[
X_b(y, \bomega) = \bigl(y, \tu_b(y)\bomega\bigr).
\]
Then the unit normal is $\nm = (1+\tu_{b,y}^2)^{-1/2}\bigl(\tu_{b,y}, \bomega\bigr)$ and the normal variation of the family of immersions $X_b$ is
\[
V = \frac{1}{\sqrt{1+\tu_{b,y}^2}} \frac{\pd \tu_b}{\pd b}.
\]
In view of the asymptotic expansion $\tu_b(y) = by + \cO(y^{-1})$ we get 
\[
V = \frac{y + \cO(y^{-1})}{\sqrt{1+\tu_{b,y}^2}}.
\]
which shows that for large $y$ the normal variation is positive, and that it grows according to $V=\cO(y)$ as $y\to\infty$.

We recall the function $W = e^{\|\vX\|^2/8} $ and in particular the equation \eqref{eq-W-upper-barrier} which it satisfies.  Since $\|A\|$ is bounded on $\tS_b$, it follows from \eqref{eq-W-upper-barrier} that $W$ is indeed an upper barrier for the Jacobi equation $\cL(V)=0$.  Since $W$ grows much faster than $V$ as $y\to\infty$ the maximum principle implies that $V$ cannot vanish at any $y\ge y_0$.  (Otherwise $V/W$ would attain a maximal value, say $m$, at some $y_1>y_0$, and then $V\le mW$ with equality at $y_1$ would contradict the maximum principle.)

Thus the $\tS_b$ do indeed foliate an open region outside the cylinder.  Using the monotonicity of $b\mapsto \tu_b(y)$ which we have just established, one can show that $\lim_{b\to0} \tu_b(y) = \sqrt{2(n-1)}$, so that the $\tS_b$ foliate the entire region with $y\ge y_0$ between any given $\tS_{b_0}$ and the cylinder $r=\sqrt{2(n-1)}$.

To complete the proof we analyze the normals $\nm$ near the cylinder.  Consider again
\[
w(y) = \frac{2y\tu\tu_y}{\tu^2-2(n-1)},
\]
where we have dropped the subscript on $\tu_b$ for brevity.

The quantity $w$ is defined on some interval $(y_1, \infty)$ where $\tu(y)>\sqrt{2(n-1)}$ for all $y>y_1$.  We will show that $y_1$ can be chosen independently from the slope $b$ of the shrinker $\tS_b$.

The asymptotic behavior of $\tu(y)$ as $y\to\infty$ implies that
\[
\lim_{y\to\infty} w(y) = 2.
\]
The quantity $w-2$ satisfies the differential equation \eqref{eq-w-ode}, which we can write as
\[
\frac{d}{dy}(w-2) = \alpha(y) (w-2) - \frac{n-1}{y\tu^2} w^2.
\]
where
\[
\alpha(y) = \frac y2 (1+\tu_y^2) - \frac{w}{2y}
\]
If $y>\sqrt{2}$ then $w<2$ implies $\alpha(y)>0$, and thus $\frac{dw}{dy}<0$.  Therefore $\lim_{y\to\infty} w = 2$ forces $w>2$ for all $y>y_1$.

Assume that $y_1$ was chosen so that $w\le 4$ on the interval $[y_1,\infty)$.  Then 
\[
\alpha(y) \ge \frac y2 -\frac 2y \text{ for } y\ge y_1.
\]
By assumption we also have $\tu\ge \sqrt{2(n-1)}$ for $y\ge y_1$, so that 
\[
\frac{d}{dy}(w-2) \ge \Bigl(\frac{y}{2} - \frac2y\Bigr) (w-2) - \frac{8}{y},
\]
Multiplying with $y^2 e^{-y^2/4}$ we get
\[
\frac{d}{dy} \bigl(y^2e^{-y^2/4}(w-2)\bigr) \ge -y^2e^{-y^2/4} \frac{8}{y} = -8ye^{-y^2/4},
\]
which upon integration leads to
\[
y^2e^{-y^2/4}(w-2) \le \int_y^\infty 8\eta e^{-\eta^2/4}\,d\eta
=16 e^{-y^2/4}.
\]
We have found
\[
w-2 \le \frac{16}{y^2}.
\]
This estimate shows that as long as $y\ge 2\sqrt2$ we will have $w-2\le 2$, i.e.~$w\le 4$.  Thus the leaves outside the cylinder satisfy Lemma~\ref{lem-tan-phi-at-cylinder} with $K=16$ for $y\ge 2\sqrt{2}$.


\end{document}